\documentclass{article}
\usepackage{graphicx}
\usepackage{amssymb,amsmath,amsthm,mathtools}
\usepackage{epstopdf}

\usepackage{bbm}
\usepackage[marginratio=1:1,height=600pt,width=460pt,tmargin=100pt]{geometry}
\usepackage{layout, color}
\usepackage{bbm}
\usepackage{enumerate}
\usepackage{authblk}

\usepackage{algorithm}
\usepackage{algpseudocode}

\usepackage{setspace}

\usepackage{hyperref}

\usepackage{url}
\usepackage[caption=false]{subfig}
\usepackage[font=small]{caption}
\usepackage{mdwlist}
\usepackage{multirow}
\usepackage{amsthm}
\usepackage{color}

\newcommand{\R}{\mathbb{R}} 
 
\newcommand{\E}{\mathbb{E}}
\newcommand{\calN}{\mathcal{N}}

\newcommand{\calM}{{\cal M}}
\newcommand{\calL}{{\cal L}}

\newtheorem{theorem}{Theorem}[section]
\newtheorem{proposition}[theorem]{Proposition}
\newtheorem{assumption}{Assumption}
\newtheorem{problem}{Problem}
\newtheorem{corollary}[theorem]{Corollary}
\newtheorem{lemma}[theorem]{Lemma}
\newtheorem{definition}[theorem]{Definition}

\theoremstyle{remark}
\newtheorem{remark}{Remark}

\newtheorem{conjecture*}{Conjecture}

\theoremstyle{plain}

\usepackage{xcolor}

\title{Bi-stochastically normalized graph Laplacian:
convergence to manifold Laplacian and robustness to outlier noise}

\author{
Xiuyuan Cheng
\thanks{Department of Mathematics, Duke University, Durham, NC.
Email: xiuyuan.cheng@duke.edu. 
}
~~~~~~~~~~~~~~~~
Boris Landa
\thanks{Program in Applied Mathematics, Yale University, New Haven, CT.
Email: boris.landa@yale.edu }
}

\date{\vspace{-20pt}}

\begin{document}

\maketitle

\begin{abstract}
\noindent
Bi-stochastic normalization provides an alternative normalization of graph Laplacians in graph-based data analysis and can be computed efficiently by Sinkhorn-Knopp (SK) iterations. This paper proves the convergence of bi-stochastically normalized graph Laplacian to manifold (weighted-)Laplacian with rates, when $n$ data points are i.i.d. sampled from a general $d$-dimensional manifold embedded in a possibly high-dimensional space. Under certain joint limit of $n \to \infty$ and kernel bandwidth $\epsilon \to 0$, the point-wise convergence rate of the graph Laplacian operator (under 2-norm) is proved to be $  O( n^{-1/(d/2+3)})$ at finite large $n$ up to log factors, achieved at the scaling of $\epsilon \sim n^{-1/(d/2+3)} $. When the manifold data are corrupted by outlier noise, we theoretically prove the graph Laplacian point-wise consistency which matches the rate for clean manifold data plus an additional term proportional to the boundedness of the inner-products of the noise vectors among themselves and with data vectors. Motivated by our analysis, which suggests that not exact bi-stochastic normalization but an approximate one will achieve the same consistency rate, we propose an approximate and constrained matrix scaling problem that can be solved by SK iterations with early termination. Numerical experiments support our theoretical results and show the robustness of bi-stochastically normalized graph Laplacian to high-dimensional outlier noise.

\end{abstract}



\section{Introduction}

Many graph-based data analysis methods start by constructing a graph adjacency or affinity matrix from data, where each node in the graph represents a data sample. The affinity matrix, or the associated graph Laplacian matrix, can be used for various downstream tasks of unsupervised learning, including dimension reduction by spectral coordinates (using the eigenvectors of the affinity or graph Laplacian matrix), metric learning and visualization by  low-dimensional embedding, modeling and computing the graph diffusion processes, and so on. The classical methods include ISOMAP \cite{balasubramanian2002isomap}, Laplacian Eigenmap \cite{belkin2003laplacian}, Diffusion Map \cite{coifman2006diffusion,talmon2013diffusion}, among others and recent follow-up works. Apart from being a pivotal method for unsupervised learning \cite{van2009dimensionality, talmon2013diffusion}, graph Laplacian methods are also readily incorporated into semi-supervised learning \cite{nadler2009semi,slepcev2019analysis,flores2022analysis} and supervised tasks as a model regularization approach \cite{ando2006learning,shuman2013emerging},
and have wide applications in data analysis and machine learning.

When using graph Laplacian methods, one can normalize the graph Laplacian in a certain manner before computing the eigen-decomposition or the diffusion process on the graph. Traditional normalization approaches include dividing the affinity matrix by the row-sum to construct a row-stochastic matrix, which corresponds to the Markov transition matrix of a graph random walk. Another normalization is by dividing the affinity matrix symmetrically using a certain power of the row-sum (called `$\alpha$-normalization') before the row-stochastic normalization, which was first introduced in \cite{coifman2006diffusion} to correct the influence of non-uniform data density. 
Weighted graph Laplacians, which result in a range of divergence form elliptic operators,
have been studied in \cite{hoffmann2022spectral} for data clustering.
To introduce the specifics of kernelized graph affinity and graph Laplacians, we provide the required preliminaries and set-ups in Section \ref{sec:prelim}.

On the theoretical side, a fundamental question is the convergence of the graph Laplacian matrix to some limiting operator as the number of data samples (the graph size) increases. 
A primary case considered in the field of study is the {\it manifold data} setting, namely when observed data samples lie on a {general} unknown low-dimensional manifold embedded in the ambient space. A simple and widely used statistical model is by assuming that data samples are drawn i.i.d. from some distribution on the manifold. 
Under the manifold data setting, the convergence of graph Laplacian to limiting manifold operator has been intensively studied in the literature. 
The operator {\it point-wise convergence}  (`point-wise' is with respect to the Laplacian operator, meaning the operator applied to some test function) has been established in \cite{hein2005graphs,hein2006uniform,belkin2007convergence,coifman2006diffusion,singer2006graph}, and extended to more general class of kernels, especially variable bandwidth kernels \cite{ting2011analysis,berry2016variable,cheng2021convergence}.
The {\it eigen-convergence}, namely the convergence of empirical eigenvalues and eigenvectors to the spectrum of manifold Laplacian, has been studied in \cite{belkin2007convergence,von2008consistency,burago2015graph,wang2015spectral,singer2016spectral,eldridge2018unperturbed}, among others, and recently in \cite{trillos2020error,calder2019improved,dunson2021spectral,cheng2021eigen,peoples2021spectral}.
These results are for traditionally normalized graph Laplacians.

An alternative approach for graph Laplacian normalization is the bi-stochastic normalization, which for a symmetric affinity matrix searches for symmetric scaling factors multiplied from left and right to the matrix to make it row- and column-stochastic simultaneously. 
The problem is known as the matrix scaling problem and can be solved by the classical Sinkhorn-Knopp algorithm \cite{sinkhorn1964relationship,sinkhorn1967concerning} or its symmetrized version \cite{zass2005unifying,knight2014symmetry}.
For graph-based analysis, the bi-stochastic normalization has been studied in several works \cite{coifman2013bi,marshall2019manifold,wormell2021spectral} where the settings differ from i.i.d. samples on a general manifold. 
When finite data are sampled on a hyper-torus, \cite{wormell2021spectral} proved the eigen-convergence rates. 
However, the point-wise rate of graph Laplacian
that can be derived using the analysis therein was worse than the best known rates (for traditionally normalized graph Laplacian),
which we will match to in this paper.
We review related works in more detail in Section \ref{subsec:literature}.

Note that the two types of convergence, point-wise convergence and eigen-convergence, differ in nature and also in practical implications: 
the eigen-convergence can provide the consistency of spectral embedding, whereas the point-wise convergence can be used to show the consistency of manifold de-noising and regression. 
In addition, while operator point-wise convergence can be used to obtain eigen-convergence in certain settings, see for example in \cite{calder2019improved,cheng2021eigen}, the latter does not necessarily imply the former, especially when the eigen-convergence rate bound 
only controls finitely many eigen-modes. 
We thus view operator point-wise convergence as a more fundamental object and focus on the point-wise rate of the bi-stochastically normalized graph Laplacian in this work.
We discuss the extension to eigen-convergence in Section \ref{subsec:discuss-eigenvector}.

A notable potential advantage of bi-stochastic normalization is its robustness to noise in the data.
It was analyzed and proved in \cite{el2016graph} that the Gaussian kernelized graph Laplacian can be made robust to noise with a modification of the diagonal entries. 
More recently, \cite{landa2021doubly} showed that bi-stochastic normalization can make the Gaussian kernelized graph affinity matrix robust to heteroskedastic noise. The noise robustness properties of bi-stochastic scaling were further analyzed in \cite{landa2023robust}. The experiments in this work echo the observation that bi-stochastic normalization rather than traditional normalization is crucial for robustness against certain types of high dimensional noise. These emerging results call for a more complete understanding of the convergence of the bi-stochastically normalized graph Laplacian to a manifold operator at finite sample size, under a general manifold data setting of i.i.d. samples and possibly with noise.

In this work, we aim to prove the convergence of the bi-stochastically normalized graph Laplacian to the manifold operator under a general setting of manifold data, either clean or with outlier noise.
Motivated by our analysis later in the paper, we introduce a modified problem of approximate bi-stochastic scaling with the constraint that the scaling factor is entry-wise uniformly lower-bounded by some $O(1)$ constant. 
This boundedness constraint incurs a mild modification of the Sinkhorn algorithm for theoretical purposes, 
which is conjectured to be omittable in practice. 
In all experiments conducted in this paper, this modification did not affect computational efficiency - on the contrary, since only approximate scaling is needed in the modified problem, one can terminate the Sinkhorn iterations early once a certain tolerance is achieved. 
The constrained scaling problem facilitates the analysis of the computed scaling factor. Specifically, we first show that the approximate matrix scaling problem has a non-empty solution set by theoretically constructing a scaling factor based on the manifold coordinates and density functions, which we call the `population scaling factor'. We then prove a 2-norm consistency of any computed (approximate) scaling factor to the population one, see Lemma \ref{lemma:eta1-eta2} and Lemma \ref{lemma:hatetac-baretac} for clean and noisy data respectively.

Based on the analysis of the approximate scaling factors, we prove the point-wise convergence of the (approximately) bi-stochastically normalized graph Laplacian to the weighted manifold Laplacian ($\Delta_p$ to be defined in \eqref{eq:def-lapp}), under 2-norm and with rates, when applied to a regular test function. For clean manifold data, the proved convergence rate matches that when the population scaling factor is used, and achieves the same point-wise convergence rate of traditionally normalized graph Laplacians, see Theorem \ref{thm:hatLn-2norm}. In addition, as a practical benefit of bi-stochastic normalization lies in its robustness to noise, we provide a theoretical justification of this by showing the convergence to a manifold operator when the influence of outlier noise is sufficiently small, see Theorem \ref{thm:hatL-convergence-noise}.  The recovery of the manifold operator using clean or noisy manifold data and the robustness to outlier noise are supported by numerical experiments. 

In summary, the contributions of the work include:

\begin{itemize}
    \item[-]
    For i.i.d. data sampled  on a general manifold embedded in an ambient space, we prove the point-wise convergence of (approximately) bi-stochastically normalized graph Laplacian to a manifold Laplacian operator with rates that match the proved rates of traditionally normalized graph Laplacians.
    
    \item[-] When the manifold data are corrupted by outlier noise, we prove the approximation to the manifold operator under certain boundedness conditions on the noise vectors. The conditions can be satisfied by high dimensional noise vectors that are possibly heteroskedastic.
    
    \item[-] We propose an approximate and constrained matrix scaling problem that can be solved by Sinkhorn-Knopp iterations with early termination in practice. We apply the algorithm to simulated datasets, and the numerical experiments support our theory.
\end{itemize}

In the rest of the section, we provide an overview of our main theoretical results, a clarification of notations, and a further review of the relevant literature. In Section \ref{sec:prelim}, we introduce the set-ups of manifold data and the kernelized graph Laplacian. The modified matrix scaling problem is introduced and analyzed in Section \ref{sec:approx-matrix-scaling-problem}. Section \ref{sec:theory-clean-data} proves the convergence to a weighted manifold Laplacian for clean manifold data, and the theory is extended to data with outlier noise in Section \ref{sec:theory-noise-outlier}. Numerical results are presented in Section \ref{sec:exp}, and in Section \ref{sec:discuss} we discuss future directions. Unless stated otherwise, all proofs are given in Section \ref{sec:proofs}.

 \begin{table}[t]
 \small
   \caption{
\label{tab:notations}
List of default notations, see more in Section \ref{subsec:notations}
} \vspace{-4pt}
 \begin{minipage}{0.45\linewidth}
 \begin{tabular}{  p{0.5cm}  p{5.75cm}   }
 \hline
  ${\calM}$ 	& $d$-dimensional manifold in $\R^m$  	  \\ 
  $d$			& intrinsic dimensionality of manifold data \\
  $m$		& dimensionality of the ambient space \\
 $p$			& data sampling density on ${\calM}$  \\
 $\Delta_p$       &  weighted manifold Laplacian, see \eqref{eq:def-lapp} \\	
 $x_i$      				& data samples		\\
$n$ 			& number of  data samples 	\\
 $\epsilon$ 	&  kernel bandwidth parameter 		 \\ 
$\varepsilon_{\rm SK}$	&  {$\| \cdot \|_\infty$-discrepancy  in the bi-stochastic normalization, see Definition \ref{def:eps-approx-scaling} }\\
$\varepsilon_z$			 &  {uniform bound on mutual inner-products of outlier noise w.h.p., see Assumption \ref{assump:A3}} \\
  \hline
\end{tabular}
\end{minipage}
 \begin{minipage}{0.5\linewidth}
 \begin{tabular}{  p{0.75cm}  p{6.25cm}   }
 \hline
 ${\bf 1}$ 				& all-ones vector \\
 $ \rho_X$				& function evaluation operator on data, see \eqref{eq:def-rho-X} \\
 $\bar{\eta}$ 			& population matrix scaling factor \\
  $\hat{\eta}$ 			& empirical matrix scaling factor \\
 $W$ 		 		& kernelized graph affinity matrix 			\\ 
 $D(W)$				& degree matrix of $W$, see \eqref{eq:def-DW} \\
 $D_a$				& diagonal matrix made from vector $a$ \\
 $L_{\rm un}$ 				& un-normalized graph Laplacian \\
 $L_{\rm rw}$				& random-walk graph Laplacian \\
 $\xi_i$ 					&  outlier noise vectors \\
 $ \odot $ 				& entry-wise multiplication \\
  $ \oslash $ 			& entry-wise division \\
  \hline
\end{tabular}
\end{minipage}
\end{table}

\subsection{Main results}\label{subsec:main-results}

The main theoretical results of the paper are summarized as follows, where $n$ is the number of i.i.d. data samples and $\epsilon > 0$ is the kernel bandwidth parameter, see \eqref{eq:def-W}:

\begin{itemize}
    \item 
    For clean manifold data, Theorem \ref{thm:hatLn-2norm} proves the 2-norm consistency of the point-wise convergence of the bi-stochastically normalized graph Laplacian to the manifold weighted Laplacian operator $\Delta_p$ defined in \eqref{eq:def-lapp}.
    The error is shown to be (up to log factors)
   \[
    O \left(\epsilon, n^{-1/2}  \epsilon^{- d/4 - 1/2} \right),
    \]
    which, at the optimal scaling of $\epsilon \sim n^{-1/(d/2+3)} $, gives an overall error of (up to log factors)
    \[
    O( n^{-1/(d/2+3)}).
    \]
    This  achieves the same rate of point-wise convergence of traditionally normalized graph Laplacians (where the affinity matrix is built with a smooth kernel like Gaussian) in previous works.
    
    \item
    When data have additive outlier noise, where the probability of being corrupted by noise is uniformly bounded away from one and the noise vectors satisfy certain boundedness conditions (the specifics are given in Assumption \ref{assump:A3}), Theorem \ref{thm:hatL-convergence-noise} proves the noise robustness of the bi-stochastically normalized graph Laplacian by showing a 2-norm consistency bound of (up to log factors)
    \[
    O \left(\epsilon, 
    n^{-1/2}  \epsilon^{- d/4 -1/2}, 
    \varepsilon_z \epsilon^{-3/2}
    \right).
    \]
    Compared to Theorem \ref{thm:hatLn-2norm}, the error has an additional term of $O(\varepsilon_z \epsilon^{-3/2})$, where $\varepsilon_z$ is the uniform bound, possibly with high probability, of the noise vectors' inner-products mutually among themselves and with the clean manifold data vectors. This shows that the convergence rate obtained under the setup of clean manifold data can be restored when the influence of outlier noise is small enough. 
    A prototypical noise model that satisfies our assumptions consists of high-dimensional random vectors that are conditionally independent given the clean data and can be heteroskedastic, see Remark \ref{rk:varepsilon-z}. 
    
\end{itemize}

\subsection{Notations}\label{subsec:notations}

We use $\R_+ = (0,\infty)$ to denote strictly positive numbers. 
$D_a = \text{diag}\{ a_1, \cdots, a_n\}$ stands for the diagonal matrix made from vector $a \in \R^n$. 
$D(W)$ denotes the degree matrix of a symmetric non-negative matrix $W$ as defined in \eqref{eq:def-DW}.
We use $\| v \|_p$ to denote for the $p$-norm of vector $v$, $ 1 \le p \le \infty $,
and $\| A \|_2$ for the matrix operator norm (induced by  vector 2-norm).
The symbol $ \odot $ stands for entry-wise multiplication, that is, for $u, v \in \R^n$, $u \odot v \in \R^n$ and $(u \odot v)_i =u_i v_i$.
The entry-wise multiplication is also defined for matrices, denoted as $A \odot B$, which is also known as Hadamard product. 
The symbol $ \oslash $ stands for entry-wise division, and $(u \oslash v)$ is defined when $v_i$ are all non-zero.
The default notations are listed in Table \ref{tab:notations}. 

The asymptotic notations are as follows: 
$f = O(g)$ means  $|f| \le C |g|$ in the limit for some $C> 0$, 
and $O_a(\cdot)$ declares the constant dependence on $a$,
 that is,  $|f| \le C_a |g|$ in the limit for some constant $C_a $ that depends on $a$,
and $a$ can be a quantity or a function to be specified in the context.
We may also omit the subscript $a$ and explain the constant dependence in the text.
$f = \Omega(g)$ means  for $f, g > 0$, $f/g \to \infty$ in the limit.
$f = o(g)$ means for $g > 0$, $|f|/g \to 0$ in the limit.
$f = \Theta(g)$, also denoted as $f \sim g$, means for $f$, $g \ge 0$, $C_1 g \le f \le C_2 g$ in the limit for some $C_1, C_2 >0$.
For compact presentation, 
$ O( g_1, g_2) $ means $O(|g_1|  +  |g_2| )$,
and similarly for $\Theta( g_1, g_2)$.
We also use  $ O( g_1, \cdots, g_k) $  to denote $O(|g_1| + \cdots +  |g_k| )$.

\subsection{Related works}\label{subsec:literature}

\paragraph{Convergence of graph Laplacian.}
For traditionally normalized graph Laplacians, the convergence of the graph Laplacian to a manifold operator has been studied in several places. Under the manifold data setting, the convergence of the integral operator of the normalized kernel function (rather than matrix) was shown in the original Diffusion Map paper  \cite{coifman2006diffusion}. It was also shown that with different $\alpha$ in the $\alpha$-normalization the integral operator recovers different limiting manifold operators.
The point-wise convergence of graph Laplacian matrix for i.i.d. sampled data drawn from a uniform density on manifold was shown in \cite{singer2006graph} with rates: Using a Gaussian kernelized graph affinity with kernel bandwidth $\epsilon$, the point-wise convergence of the graph Laplacian to the Laplace-Beltrami operator was proved to have an error of $O(\epsilon, n^{-1/2} \epsilon^{-d/4-1/2})$ for large finite $n$, with high probability and up to a log factor. This gives an overall rate of of $O(n^{-1/(d/2+3)})$ at the optimal choice of $\epsilon \sim n^{-1/(d/2+3)}$.
The point-wise convergence of the graph Laplacian using variable bandwidth kernels was proved in \cite{berry2016variable} with rates.
The same point-wise convergence rate of $O(n^{-1/(d/2+3)})$ as in \cite{singer2006graph} was proved in \cite{cheng2021convergence} for the graph Laplacian built with $k$-nearest-neighbor estimated bandwidth (known as the `self-tuned' graph Laplacian \cite{zelnik2005self}), where the kernel matrix is also normalized by the empirically estimated bandwidth values.
For the $\alpha=1$-normalized graph Laplacian, \cite{cheng2021eigen} proved the point-wise convergence rate of $O(n^{-1/(d/2+3)})$, see Theorem 6.2 in \cite{cheng2021eigen}. 
\cite{cheng2021eigen} also proved the eigen-convergence of the graph Laplacian built using the Gaussian kernel with rates,
which are $O(n^{-1/(d/2+2)})$ for eigenvalues and $O(n^{-1/(d/2+3)})$ for eigenvectors in 2-norm, both up to a log factor.
When the graph Laplacian was built using a compactly supported kernel function, \cite{calder2019improved} proved a point-wise convergence rate of $O(n^{-1/(d+4)})$, see Theorem 3.3 in \cite{calder2019improved}, and also extended the result to $k$NN graphs. \cite{calder2019improved} also proved an eigen-convergence rate of  ${O}(n^{- 1/(d + 4)})$ up to a log factor for such graphs when $d \ge 2$, which holds for both eigenvalue convergence and eigenvector convergence in 2-norm; The same rate of eigenvector convergence in $\infty$-norm was proved in \cite{calder2020lipschitz}.
While the current paper borrows certain techniques from these previous works, the results therein do not cover the bi-stochastically normalized graph Laplacian.

Many more variants of the kernelized graph Laplacian have been developed, including the use of anisotropic kernels \cite{singer2009detecting, talmon2013empirical,berry2016local}, landmark sets  \cite{bermanis2016measure,long2017landmark, shen2022scalability}, and so on. 
These methods improve the statistical and computational efficiency of graph Laplacian methods in various scenarios and may be used concurrently with bi-stochastic normalization. The theoretical understanding of the latter, which is the focus of the current paper, can potentially be extended to these modifications as well.

\paragraph{Bi-stochastic graph Laplacian.}
The bi-stochastic kernel was proposed in \cite{coifman2013bi} for graph-based data analysis without solving for matrix scaling. 
From the perspective of an integral kernel, \cite{marshall2019manifold} proved the point-wise convergence of the bi-stochastically normalized kernel operator to the manifold operator.

The bi-stochastic graph Laplacian was previously studied in \cite{wormell2021spectral}, which we comment in detail here.
Considering the setting of i.i.d. data lying on a hyper-torus, \cite{wormell2021spectral} proved that bi-stochastically normalized graph Laplacian achieves an eigenvalue convergence rate of $O( n^{-2/(d+8) +o(1)})$ which also holds with the standard (traditional) normalization, 
and eigenspace convergence using bi-stochastic normalization enjoys an improved rate of $O( n^{-4/(d+12) +o(1)})$. 
The paper also derived point-wise rate of the bi-stochastically normalized semi-group operator $P^n_\epsilon = I + \epsilon L^n_\epsilon$, where $L^n_\epsilon$ is the finite-sample graph Laplacian. The point-wise rate of $P^n_\epsilon$ to approximate $e^{\epsilon \calL}$ was shown to be $O(\epsilon^3, n^{-1/2}\epsilon^{-d/4})$ where $\calL$ is the weighted manifold Laplacian, see Proposition 3.6 and Theorem 3.7 in \cite{wormell2021spectral}. However, since $\frac{1}{\epsilon}(e^{\epsilon \calL} -I)$ only approximates $\calL$ up to $O(\epsilon)$, the induced point-wise rate of $L^n_\epsilon$ to approximate $\calL$ is $O(\epsilon, n^{-1/2}\epsilon^{-d/4-1})$, which at optimal $\epsilon$ leads to the eigenvalue rate therein as $O(n^{-2/(d+8)})$  
and is slower than the point-wise rate $O(n^{-2/(d+6)})$
 for traditionally normalized graph Laplacians \cite{singer2006graph}. In addition, the techniques in \cite{wormell2021spectral} cannot directly extend when the manifold is not a torus.
In the current work, we use a different technique to prove the point-wise convergence of bi-stochastic  graph Laplacian,
which applies to general manifold data (not necessarily on a hyper-torus).
Our result also matches the traditional point-wise rate of $O(n^{-2/(d+6)})$.
In addition, we theoretically analyze the robustness to data noise.

\paragraph{Noise-robustness of graph Laplacian.}
The noise robustness of graph Laplacian has been studied in \cite{el2016graph}, suggesting the special property of Gaussian kernels. More recently, \cite{ding2022impact} analyzed the influence of high dimensional noisy data on the spectrum of graph Laplacian. The noise robustness of  bi-stochastically normalized Gaussian kernelized affinity matrix was studied in \cite{landa2021doubly}, which proved a  decay of the influence of high dimensional noise at rate $O(m^{-1/2})$ as the ambient dimensionality $m$ increases. \cite{landa2023robust} analyzed and showed the benefits of bi-stochastic scaling in robust inference of manifold data. The current paper considers outlier noise in data and shows the robustness of  bi-stochastically normalized graph Laplacian, both theoretically (via the point-wise convergence to the limiting operator with rates) and in experiments.

\paragraph{Sinkhorn-Knopp (SK) iteration.}
To solve the bi-stochastic normalization which can be casted as a matrix scaling problem,
the current paper adopts the  accelerated  symmetric Sinkhorn iteration algorithm developed in \cite{wormell2021spectral},
which also provided  a convergence analysis.
Compared with the original SK algorithm \cite{sinkhorn1964relationship,sinkhorn1967concerning},
the symmetric SK iteration utilizes the symmetry of the matrix to be normalized as well as the resulting symmetry of the scaling factors,
which achieves an acceleration in convergence.
A similar heuristic was suggested earlier in \cite{marshall2019manifold}, and also studied for symmetry-preserving matrix scaling \cite{knight2014symmetry}. 
In this paper, we propose an approximate and constrained matrix scaling problem that can be solved by SK iterations with early termination,
see Algorithm \ref{algo:SK-early}.
In particular, the parameter $\varepsilon_{\rm SK}$, which sets the tolerance of the approximate bi-stochastic normalization,
is motivated by the statistical error analysis.
Our analysis justifies the usefulness of approximate bi-stochastic normalization and the $\varepsilon_{\rm SK}$ parameter allows a trade-off between accuracy and computational cost in practice.

The Sinkhorn algorithm has become a popular tool in computational optimal transport (OT)  to solve for entropy-regularized OT on a discrete system \cite{cuturi2013sinkhorn,peyre2019computational}. Recent developments include stochastic optimization \cite{genevay2016stochastic}, coordinate descent techniques \cite{altschuler2017near}, among others. New techniques of solving matrix scaling problems can potentially be applied to solve the (approximate) matrix scaling problem proposed in this paper.

\section{Preliminaries }\label{sec:prelim}

\subsection{Manifold data and manifold operator}

We need the following assumptions on the data manifold $\calM$ and density $p$ of clean manifold data.
The assumption of data with outlier noise will be introduced in Section \ref{sec:theory-noise-outlier}.

\begin{assumption}[A1]
\label{assump:A1}
${\calM}$ is a $d$-dimensional compact connected  
 $C^{\infty}$ manifold (without boundary) 
isometrically embedded in $\R^{m}$.
\end{assumption}

We consider the measure space $(\calM, p dV)$, where $dV$ is the local volume form of the manifold $\calM$, and $p$ is the density of the data distribution on $\calM$. In below, we denote by $p$ both the distribution and the density function.
In this work, we consider $n$ i.i.d samples $\{ x_i \}_{i=1}^n$ drawn from $p$ on $\calM$.

\begin{assumption}[A2]
\label{assump:A2}
The data density $p\in C^{6}(\calM)$ and for positive constants $p_{\rm min}, \, p_{\rm max}$, 
\[ 0< p_{\rm min}  \le p(x) \le p_{\rm max} < \infty, \quad\forall x\in{\calM}.
\]
\end{assumption}

Due to the compactness of $\calM$, the manifold has finite reach and bounded curvature. 
As long as $p$ is continuous on $\calM$ and never vanishing the uniform lower and upper boundedness follows.
The $C^6$ regularity of $p$ is technically required for the analysis, see Lemma \ref{lemma:construct-bar-eta}, and may be relaxed (e.g., to be in  H\"older class) following standard techniques. 
The assumption on $\calM$ may also be extended to account for non-compactness or boundary. 
We postpone these extensions for simplicity, as the current work focuses on the bi-stochastic normalization of graph Laplacians.

Let $\Delta = \Delta_\calM$ denote the Laplace-Beltrami operator on $\calM$, and $\nabla = \nabla_\calM$ the manifold gradient. 
The weighed manifold Laplacian operator is defined as
\begin{equation}\label{eq:def-lapp}
\Delta_p := \Delta +   \frac{\nabla p}{p } \cdot \nabla,
\end{equation}
and when $p$ is uniform, $\Delta_p$ is reduced to $\Delta$.
The weighted manifold Laplacian $\Delta_p$ is the Fokker-Planck operator of the diffusion process on the manifold towards the equilibrium distribution $p$. We will show that $\Delta_p$ is the limiting operator of the bi-stochastically normalized graph Laplacian. 

\subsection{Kernelized graph affinity matrix}

Given data samples $\{ x_i \}_{i=1}^n$, 
the graph affinity matrix $W$ can be computed from a kernel function $k (x,x')$
by letting $W_{ij}$, the affinity between node $i$ and $j$, equal $k(x_i, x_j)$.
When $k$ is symmetric and non-negative, the kernelized graph affinity matrix $W$ is (real) symmetric and has non-negative entries. 
A widely used choice of $k$ is Gaussian RBF kernel. The Gaussian kernelized graph affinity matrix $W$ is defined as
\begin{equation}\label{eq:def-W}
W_{ij } = \frac{1}{n}
\epsilon^{-d/2} g \left(  \frac{ \| x_i - x_j\|^2}{\epsilon} \right),
\end{equation}
where  $\epsilon > 0$ is the kernel bandwidth parameter, and 
\begin{equation}\label{eq:def-g-gaussian}
g( \xi) = \frac{1}{( 4 \pi )^{d/2}} e^{- \xi/ 4}, \quad \xi \in [0, \infty).
\end{equation}

In the literature on graph Laplacians and Diffusion Maps, there are two conventions for the usage of ``$\epsilon$''.
One group of works use $\epsilon$ to stand for the scale of distance.
E.g. in an ``$\epsilon$-graph'' the affinity matrix is $h( {\|x_i - x_j\|}/{\epsilon})$ where $h(\xi ) = {\bf 1}_{[0,1]}$,
that is, two nodes $i$ and $j$ are connected if $\| x_i - x_j \| \le \epsilon$. 
The other group of works uses $\epsilon$ to stand for diffusion time, e.g., the Gaussian kernel affinity $\exp\{- \| x_i - x_j\|^2/\epsilon\}$, 
and this $\epsilon$ is on the scale of the squared distance. 
In this work, we follow the second convention as has been used in \cite{coifman2006diffusion,singer2006graph,berry2016variable,cheng2021convergence}, among others.

The definition \eqref{eq:def-W}\eqref{eq:def-g-gaussian}
leads to the kernel function in $\R^{m}$ to be Gaussian.
The constant factors $( 4 \pi )^{-d/2}$  and $\epsilon^{-d/2}$ in the definition of $W $ are for theoretical convenience, and are not needed  in practice (especially knowledge of the intrinsic dimensionality $d$): 
for the random-walk graph Laplacian and the bi-stochastic normalized graph Laplacian considered in this paper, these constant factors are cancelled out in computation. 

The {\it degree matrix} $D$ of $W $, denoted as $D(W)$ in this paper, is defined as 
\begin{equation}\label{eq:def-DW}
D_{ii} = \sum_{j = 1}^N W_{ij}, \quad i=1, \cdots, n,
\quad  \text{and } D_{ij} = 0, \quad \forall i \neq j.
\end{equation}
Typically in the analysis of kernelized graph Laplacian one considers the joint limit when $n \to \infty$ and the kernel bandwidth $\epsilon $ approaches zero.
In this paper, we consider the joint limit (with more specific asymptotic conditions below) and our theory presents non-asymptotic result which holds for finite $n$.

\begin{remark}[connectivity regime]
\label{rk:connectivity-regime}
Since the diagonal entries of $W$ equal $(4\pi\epsilon)^{-d/2}/n$ and is positive, the degree matrix $D(W)$ has strictly positive entries on the diagonal. 
If the diagonal entries of $W$ are set to be zero (denoted as $W^0$ below), since the data samples are i.i.d. drawn from $p$ on $\calM$, a standard concentration argument gives  that $D(W^0)_{ii}$ for all $i$, with high probability, are bounded from below by some $O(1)$ positive constant under the regime that $\epsilon^{d/2} =\Omega( {\log n}/{n}) $, see  Lemma \ref{lemma:degree-conc-W}.
The condition $\epsilon^{d/2} =\Omega( {\log n}/{n}) $ is known as the `connectivity regime', and in this paper our theoretical assumption on the largeness of $\epsilon$ as $n$ increases is always under this regime. 
\end{remark}

\subsection{Traditionally normalized graph Laplacians}

Given a symmetric and non-negative-entry graph affinity matrix $W$ and its degree matrix $D(W)$, the {\it un-normalized} graph Laplacian is defined as 
\[
L_{\rm un} := D(W ) - W.
\]
Assuming that $D(W)$ has all strictly positive diagonal entries
(which holds for $W$ as in \eqref{eq:def-W} and also $W^0$ with zero diagonal entries under the connectivity regime, see Remark \ref{rk:connectivity-regime}),
the {\it normalized} graph Laplacian, also known as random-walk graph Laplacian, is defined as
\[
L_{\rm rw} := I - D(W )^{-1} W.
\]

To handle non-uniform data density on the manifold, an $\alpha$-normalized graph Laplacian was introduced in \cite{coifman2006diffusion}, for $\alpha  \in [0,1]$.
It is constructed by first using the $\alpha$-th power of the row sums of $W$ to normalize $W$ symmetrically,
which produces another affinity matrix $\tilde{W}^{(\alpha)}$, and then computing the random-walk graph Laplacian associated with  $\tilde{W}^{(\alpha)}$.
It was shown in  \cite{coifman2006diffusion} that the $\alpha$-normalized graph Laplacian recovers the weighted manifold Laplacian operator (the Fokker-Planck operator) $\Delta_p$ when  $\alpha = 1/2$, and the Laplace-Beltrami operator $\Delta$ when $\alpha = 1$.
As the bi-stochastically normalized graph Laplacian also converges to the operator $\Delta_p$ in the large sample limit, we compare it with the $\alpha = 1/2$-normalized graph Laplacian in experiments and show that the former has better robustness when data have outlier noise. 

\section{Approximate and constrained matrix scaling}\label{sec:approx-matrix-scaling-problem}

We will apply bi-stochastic normalization to the graph affinity matrix $W$ (possibly setting to be zero the diagonal entries, see more later).
Generally, we consider matrix $A$ which is real symmetric and has non-negative entries.
Instead of the standard matrix scaling of $A$, namely finding $\eta \in \R_+^n$ such that $D_\eta A D_\eta$ is exactly bi-stochastic, we consider approximate bi-stochasticity. 

\begin{definition}[$\varepsilon_{\rm SK}$-approximate scaling factor]
\label{def:eps-approx-scaling}
Given a real symmetric  $n$-by-$n$ matrix $A$ with non-negative entries, 
we say that $\eta \in \R_+^n$ is an $\varepsilon_{\rm SK}$-approximate bi-stochastic scaling factor of $A$ if 
    $D_{{\eta}} A D_{{\eta}} {\bf 1} = {\bf 1} + e$,
    and $ \| e \|_\infty \le \varepsilon_{\rm SK}$,
    where $\bf 1$ stands for vector of all ones in $\R^n$.
\end{definition}

We also introduce the constraint of lower-bounded scaling factors, and solve for the following problem.

\begin{problem}[$(C_{\rm SK},\varepsilon_{\rm SK})$-matrix scaling problem]\label{prob:C-eps-matrix-scaling}
Given positive constants $C_{\rm SK}$ and $\varepsilon_{\rm SK}$,
find an  $\eta \in \R_+^n$ which is an $\varepsilon_{\rm SK}$-approximate bi-stochastic scaling factor of $A$
and satisfies that 
$\min_{i} \eta_i \ge C_{\rm SK}$.
\end{problem}

Theoretically, we will show that, when $W$ is built from clean manifold data and $(C_{\rm SK}, \varepsilon_{\rm SK})$ are properly chosen, the solution set of Problem \ref{prob:C-eps-matrix-scaling} is not empty - by constructing a solution which we call the {\it population} scaling factor (Section \ref{subsec:population-scaling}). 
Furthermore, we will show that any solution $\eta$  to the $(C_{\rm SK},\varepsilon_{\rm SK})$-scaling problem approximates the population scaling factor, 
to be specified below (Section \ref{subsec:lemmas-scaling}). 
The lower bound  $C_{\rm SK}$ will be used in our analysis, in Section \ref{subsec:lemmas-scaling} and the subsequent proofs.
In algorithm, the entry-wise lower bound incurs a minor modification of the SK iterations, see more in Section \ref{subsec:modified-SK}.

\subsection{Population approximate scaling factor}\label{subsec:population-scaling}

In this subsection we show that for some $C_{\rm SK}$ depending on manifold $\calM$ and density $p$, there is at least one scaling factor $\bar{\eta}$ which solves Problem \ref{prob:C-eps-matrix-scaling} where 
\begin{equation}\label{eq:varepsilon-bar-eta}
\varepsilon_{\rm SK} = O \left( \epsilon^2, \sqrt{\frac{\log n }{n \epsilon^{d/2}}} \right).
\end{equation}
We construct $\bar{\eta}$ by evaluating a function $q_{\epsilon}$ on $\calM$ on the $n$ data points,
 that is,  $\bar{\eta} = \rho_X (q_{\epsilon})$, where the  function evaluation of $f$ on dataset $X$ is denoted by $\rho_X(f)$ defined as 
 \begin{equation}\label{eq:def-rho-X}
\rho_X (f) := (f(x_1), \cdots, f(x_n)) \in \R^n.
\end{equation} 
The following lemma provides a construction of $q_\epsilon$ that fulfills our purpose. 

\begin{lemma}[Existence of population scaling factor]
\label{lemma:construct-bar-eta}
Under Assumptions (A1)(A2), 
there exists a function $r \in C^4(\calM)$, 
$r$ is determined by $p$ and manifold extrinsic coordinates, 
such that 
\[
q_\epsilon: = p^{-1/2}(1 + \epsilon r)
\]
is $C^4$ on $\calM$ and  satisfies the following:

(i)
There are positive constants $q_{\rm min} $, $q_{\rm max} $, and $\epsilon_{0} $ which are determined by $r$ and $p$, 
s.t. for $\epsilon < \epsilon_{0}$,
 $0 < q_{\rm min} \le q_\epsilon (x) \le q_{\rm max}$ for all $x\in \calM$.

(ii) Suppose $x_i \sim p$ i.i.d. on $\calM$, $i=1,\cdots,n$,  and $ W$ is defined as in \eqref{eq:def-W}.
If as $n \to \infty$, $\epsilon \to 0+$ and $\epsilon^{d/2} = \Omega(\log n/n)$,
then when $n $ is sufficiently large,  w.p. $> 1-2  n^{-9}$, 
$\bar{\eta} := \rho_X (q_\epsilon)$ is an $\varepsilon_{\rm SK}$-approximate bi-stochastic scaling factor of $W$
where $\varepsilon_{\rm SK}$ satisfies \eqref{eq:varepsilon-bar-eta}.
\end{lemma}

 \begin{remark}
 \label{rk:W-diagonal-zero}
 As suggested by the proof of Lemma \ref{lemma:construct-bar-eta}, 
 the diagonal entries in $W$ only contribute to higher order error, see \eqref{eq:bar-eta-i-1} and \eqref{eq:bound-1-bar-eta-1}.
 As a result,  if we replace $W$ with $W^0$ which is by setting to zero the diagonal entries of $W$, the lemma remains to hold,
that is, $\bar{\eta}$ is an $\varepsilon_{\rm SK}$-approximate scaling factor of $W^0$ under the stated high probability event.
 \end{remark}

\subsection{Lemmas about $(C_{\rm SK},\varepsilon_{\rm SK})$-matrix scaling problem}\label{subsec:lemmas-scaling}

We first derive the following lemma showing that 
two approximate scaling factors of Problem \ref{prob:C-eps-matrix-scaling}  are close to each other
when $A$ is positive semi-definite (PSD) and the two scaling factors satisfy certain boundedness conditions.

\begin{lemma}[Comparison of approximate scaling factors]
\label{lemma:eta1-eta2}
Suppose $A$ is symmetric, 
$A \succeq 0$, $\eta_1, \eta_2 \in \R_+^n$ are two $\varepsilon_{\rm SK}$-approximate bi-stochastic scaling factor of $A$, $\varepsilon_{\rm SK} < 0.1$, 
and  there exist $C_1, C_2 > 0$ such that
\begin{equation}\label{eq:cond-C1-C2-bdd}
\min_i (\eta_1)_i \ge  C_1,
\quad
\max_i (\eta_2)_i \le C_2.
\end{equation}
Let $D_{{\eta_i}} A D_{{\eta_i}} {\bf 1} = {\bf 1} + e(\eta_i)$ for $i=1,2$, and
define $u \in \R^n$ by $\eta_2 = \eta_1 \odot ({\bf 1} + u)$,  then
\[
\| u \|_2 \le 
\frac{1}{0.9 } 
\frac{C_2}{C_1} (\| e(\eta_1)\|_2 + \|e(\eta_2)\|_2 ).
\]
\end{lemma}

\begin{remark}
\label{rk:hateta-W0}
Lemma \ref{lemma:eta1-eta2} extends to non-PSD (symmetric non-negative entry) $A$ as long as 
(i) there is $\delta > 0$ and $\delta < 0.4 C_1/C_2^3$ such that $A + \delta I  \succeq 0$
and (ii) $\max_i (\eta_1)_i \le C_2$. 
The extension is Lemma \ref{lemma:eta1-eta2-nonPSD},
which is included in Section \ref{sec:proofs} with proof.
A similar argument is used in the proof of  Lemma \ref{lemma:hatetac-baretac} later.
\end{remark}

Note that Lemma \ref{lemma:eta1-eta2} requires uniform boundedness of scaling factors from both below and above,
while Problem \ref{prob:C-eps-matrix-scaling} only enforces the lower bound. 
The next elementary lemma gives that when the row sum of $A$ are uniformly bounded from below (corresponding to the degree matrix of $W$), then the uniform lower boundedness of an approximate scaling factor implies the uniform upper boundedness. 

\begin{lemma}[Uniform upper-boundedness of empirical scaling factor]
\label{lemma:C1-implies-C2}
Suppose the symmetric matrix $A$ with non-negative entries satisfies that 
$\sum_j A_{ij} \ge C_3 > 0 $ for all $i$,
and $\eta \in \R_+^n$ is an 
$\varepsilon_{\rm SK}$-approximate bi-stochastic scaling factor of $A$, 
and $\min_j \eta_j \ge C_1 >0$, then
\[
\max_j \eta_j \le \frac{1+\varepsilon_{\rm SK}}{C_1 C_3}. 
\]
\end{lemma}

As shown in Lemma \ref{lemma:construct-bar-eta}, the population scaling factor $\bar{\eta}$ satisfies uniform boundedness by $O(1)$ constants $q_{\rm min}$ and $q_{\rm max}$.
The empirically solved approximate scaling factor $\hat{\eta}$ observes the lower boundedness of $O(1)$ constant by problem constraint,
and then, combined with the concentration of the degree matrix (Lemma \ref{lemma:degree-conc-W}),
Lemma \ref{lemma:C1-implies-C2} guarantees that $\hat{\eta}$  will also observe the uniform upper bound of some $O(1)$ constant. 
Then Lemma \ref{lemma:eta1-eta2} applies to bound the 2-norm of $u$ as the entry-wise relative discrepancy between $\bar{\eta}$ and $\hat{\eta}$.
This bound will be a key step in the analysis of graph Laplacian convergence,
matching the convergence of (approximate) bi-stochastic normalized graph Laplacian to that normalized by population scaling factors.

\begin{algorithm}[t]
	\caption{Accelerated symmetric Sinkhorn-Knopp iterations with lower-bound constraint}
	\label{algo:SK-early}
	
	Input: symmetric $n$-by-$n$ matrix $A$ with non-negative entries, 
		 positive constants $C_{\rm SK}$ and $\varepsilon_{\rm SK}$.
		 
	Output: approximate scaling factor $\eta$, number of iterations $k$
	
	\begin{algorithmic}[1]		
		\Function{ApproxSymSK}{$A$, $C_{\rm SK }$, $\varepsilon_{\rm SK}$}
			\Comment{Default values $C_{\rm SK}=$ 0.01, $\varepsilon_{\rm SK}=$  \texttt{1e-3}}	
		\State Initial scaling factor $\eta \in \R_+^n$ by letting  $\eta_i \leftarrow 1/\sqrt{(d_A)_i}$, where $d_A  \leftarrow A {\bf 1}$ is the row sum
		\State Projection step$^*$: set $\eta_i $ to be $C_{\rm SK }$ for those $i$ where $\eta_i < C_{\rm SK }$
		\For{$ k = 1$ to MaxIte}
			\Comment{MaxIte is the max number of iteration, default 50}	
		\State Compute the discrepancy $e  \leftarrow  D_\eta A D_\eta {\bf 1} - {\bf 1} $
		\If{$\| e \|_\infty < \varepsilon_{\rm SK} $}
		\State End the for loop
		\EndIf
		\State $u \leftarrow {\bf 1}  \oslash (A \eta)$
		\State $v \leftarrow {\bf 1}  \oslash (Au)$
		\State Compute $\eta \in \R_+^n$ by $\eta_i \leftarrow \sqrt{ u_i v_i }$
		\State Projection step$^*$: set $\eta_i $ to be $C_{\rm SK }$ for those $i$ where $\eta_i < C_{\rm SK }$
		\EndFor	
		\State \textbf{return} $\eta$, $k$
		\EndFunction	
	\end{algorithmic}
	($^*$ There are cases when the convergent factor $\eta$ is uniformly bounded from below, say by $C$, 
	and then by  setting $C_{\rm SK}$ properly e.g. $C_{\rm SK}=C/2$ the projection step is never used throughout the SK iterations.)
\end{algorithm}

\subsection{Modified Sinkhorn algorithm}\label{subsec:modified-SK}

We are to solve the $(C_{\rm SK},\varepsilon_{\rm SK})$-matrix scaling problem  (Problem \ref{prob:C-eps-matrix-scaling}) of $W$ with
\begin{equation}\label{eq:C-eps-modified-SK}
C_{\rm SK} = q_{\rm min}, 
\quad 
\varepsilon_{\rm SK} =  O \left( \epsilon^2, \sqrt{\frac{\log n }{n \epsilon^{d/2}}} \right),
\end{equation}
where $q_{\rm min}$ is as in Lemma \ref{lemma:construct-bar-eta}, and $\varepsilon_{\rm SK}$ is same as in \eqref{eq:varepsilon-bar-eta}.
Lemma \ref{lemma:construct-bar-eta} implies that, under the good event which happens with high probability, 
$\bar{\eta} = \rho_X(q_\epsilon) $ is a solution to the problem with $(C_{\rm SK},\varepsilon_{\rm SK})$ as in \eqref{eq:C-eps-modified-SK},
and thus the solution set is not empty.
Generally, the solution set is not unique and we terminate the algorithm when one valid solution $\hat{\eta}$ is found.
As suggested by Remarks \ref{rk:W-diagonal-zero} and \ref{rk:hateta-W0},  and the graph Lapalcian convergence analysis in below (Remark \ref{rk:clean-GL-W0}), one may construct the bi-stochastically normalized graph Laplacian from $W^0$ instead of $W$. 
Later in Section \ref{sec:theory-noise-outlier}, we use the zero-diagonal affinity matrix when data has outlier noise. 
In all experiments we solve Problem  \ref{prob:C-eps-matrix-scaling} on $W^0$ for simplicity.

The classical Sinkhorn-Knopp (SK) iteration, modified to include a projection step for the lower-boundedness constraint, 
suffices to solve our problem  with early termination,
 i.e., as long as the target approximation accuracy $\varepsilon_{\rm SK}$ is achieved. 
In this paper, we use the {\it accelerated symmetric Sinkhorn algorithm}  \cite{marshall2019manifold, wormell2021spectral},
and a convergence analysis of the algorithm can be found in \cite{wormell2021spectral}.
The algorithm is summarized  in Algorithm \ref{algo:SK-early}.
While technically we enforce the $\min_i \hat{\eta}_i \ge C_{\rm SK}$ constraint by a projection step in the iteration,
 in experiments we have found that under the manifold data (possibly with noise) settings of this paper
 the original (symmetric) SK iterations converge to an approximate solution $\hat{\eta}$ satisfying uniform lower-boundedness, and the lower-boundedness by some $O(1)$ constant is automatically satisfied  throughout the iterations, see more in Section \ref{sec:exp}. 
 This suggests that the $ C_{\rm SK}$ boundedness constraint may not be needed in the algorithm. 
 Under the conjecture that the constraint can be removed, 
 one may turn off the projection step in Algorithm \ref{algo:SK-early} and eliminate the parameter $ C_{\rm SK}$ in practice.

\section{Convergence of graph Laplacian}\label{sec:theory-clean-data}

Given $n$ i.i.d. data samples and the kernelized affinity matrix $W$ as in \eqref{eq:def-W}, suppose an approximate scaling factor $\hat{\eta}$ is found which solves Problem \ref{prob:C-eps-matrix-scaling} of $W$ with $(C_{\rm SK},\varepsilon_{\rm SK})$ as in \eqref{eq:C-eps-modified-SK}. 
We have introduced the notation $D(W)$ to stand for the degree matrix of a symmetric non-negative-entry affinity matrix $W$.
The bi-stochastic normalized graph Laplacian is defined as 
\begin{equation}\label{eq:def-hatL}
\hat{L}_n = D(\hat{W})- \hat{W}, 
\quad \hat{W} := D_{\hat{\eta}} W D_{\hat{\eta}}.  
\end{equation}
In this section, we prove the convergence of $\hat{L}_n$ to manifold (weighted) Laplacian operator $\Delta_p$ as defined in \eqref{eq:def-lapp}, 
With $\bar{\eta} = \rho_X(q_\epsilon) $,
we also define 
\begin{equation}\label{eq:def-barL}
\bar{L}_n = D(\bar{W}) - \bar{W},
\quad \bar{W}:= D_{\bar{\eta}} W D_{\bar{\eta}},
\end{equation}
and we will show that $\bar{L}_n$ also converges to $\Delta_p$.
By definition and under the good event of Lemma \ref{lemma:construct-bar-eta},
both $\bar{\eta}$ and $\hat{\eta}$ are solutions to the $(C_{\rm SK},\varepsilon_{\rm SK})$-matrix scaling problem of $W$, 
and this means that
\[
\bar{W} {\bf 1}  = {\bf 1}  +  e(\bar{\eta}),
\quad
\hat{W} {\bf 1}  = {\bf 1}  +  e(\hat{\eta}),
\quad
\| e(\bar{\eta})\|_\infty, \| e(\hat{\eta})\|_\infty \le \varepsilon_{\rm SK},
\quad
\min_{i} \bar{\eta}_i, \min_i \hat{\eta}_i \ge C_{SK} > 0.
\]
We prove the convergence of $\hat{L}_n$ by first establishing the convergence of $\bar{L}_n$ in Section \ref{subsec:convergence-barL},
and then comparing $\hat{L}_n$ to $\bar{L}_n$ in Section \ref{subsec:convergence-hatL}.

\subsection{Convergence of $\bar{L}_n$}\label{subsec:convergence-barL}

\begin{proposition}[Pointwise convergence of $\bar{L}_n$]
\label{prop:barLn-pointwise}
Under (A1)(A2), suppose $x_i \sim p$ i.i.d. on $\calM$, $i=1,\cdots, n$,
and $W$ defined as in \eqref{eq:def-W}.
If as $n \to \infty$, $\epsilon \to 0+$ and $\epsilon^{d/2+1} = \Omega( \log n /n)$,
then for any $f \in C^4(\calM)$, when $n$ is large enough, w.p.$>1- 2 n^{-9}$,
\begin{equation}\label{eq:eqn-ptwise-barL}
\left( - \frac{1}{\epsilon}\bar{L}_n \rho_X f \right)_i  
= \Delta_p f(x_i) + O \left( \epsilon, \sqrt{ \frac{\log n }{n \epsilon^{d/2+1}}}\right), 
\quad \forall i =1,\cdots, n,
\end{equation}
and the constant in big-O is uniform for all $x \in \calM$ and depends on $(\calM, p)$ and $f$. 
\end{proposition}

Proposition \ref{prop:barLn-pointwise} gives the $\infty$-norm consistency of the Laplacian matrix $\bar{L}_n$ applied to $f$, namely
\begin{equation}\label{eq:barLn-Errpt}
\| \left( - \frac{1}{\epsilon}\bar{L}_n \rho_X f  \right) -  \rho_X ( \Delta_p f )   \|_\infty 
= O \left( \epsilon, \sqrt{ \frac{\log n }{n \epsilon^{d/2+1}}}\right).
\end{equation}
The proof follows similar technique of analyzing the point-wise convergence of kernelized graph Laplacian \cite{singer2006graph,cheng2021convergence}, 
and is included in Section \ref{sec:proofs}.

\begin{remark}[Requirement on the largeness of $\epsilon$]
\label{rk:eps-regime}
Under the condition that $\epsilon^{d/2+1} = \Omega( \log n /n)$,  the error term in the r.h.s. of \eqref{eq:eqn-ptwise-barL} is  $o(1)$. 
As shown in the proof of the proposition, the bound \eqref{eq:eqn-ptwise-barL} holds as long as $\epsilon$ is beyond the connectivity regime up to a $\log (1/\epsilon)$ factor
(specifically, when \eqref{eq:condition-Bernstein} holds). 
\end{remark}

\subsection{Convergence of $\hat{L}_n$}\label{subsec:convergence-hatL}

Next, we prove the 2-norm convergence of  $\frac{1}{\epsilon}\hat{L}_n f$  to $- \Delta_p f$ where $f$ is a smooth function on $\calM$. 
The proof, included in Section \ref{sec:proofs}, 
is based on the closeness of $\hat{\eta}$ to $\bar{\eta}$
which allows to compare $\hat{L}_n$ to $\bar{L}_n$ and bound the difference.

 \begin{theorem}\label{thm:hatLn-2norm}
Under the same condition of Proposition \ref{prop:barLn-pointwise}, let $f \in C^4(\calM)$.
Suppose $\hat{\eta} \in \R_+^n$ is a solution to the $(C_{\rm SK}, \varepsilon_{\rm SK})$-matrix scaling problem of $W$ with $(C_{\rm SK},\varepsilon_{\rm SK})$ as in \eqref{eq:C-eps-modified-SK},
that is, $C_{\rm SK}$ is an $O(1)$ constant depending on $(\calM, p)$, $\varepsilon_{\rm SK} = O \left( \epsilon^2, \sqrt{\frac{\log n }{n \epsilon^{d/2}}} \right)$ and suppose $\varepsilon_{\rm SK} < 0.1$.
Then for sufficiently large $n$, w.p. $> 1- 6 n^{-9}$,
 \[
\frac{1}{\sqrt{n}} \|  \left( - \frac{1}{\epsilon} \hat{L}_n (\rho_X f)  \right) - \rho_X ( \Delta_p f) \|_2
 =    O \left(\epsilon, \sqrt{\frac{\log n \log (1/\epsilon) }{n \epsilon^{d/2+1}}} \right).
 \]
The constant in big-O depends on $(\calM, p)$ and $f$. 
  \end{theorem}

  \begin{remark}[Rate of convergence]
  \label{rk:point-wise-rate}
The error bound in Theorem \ref{thm:hatLn-2norm}  can be interpreted as the sum of an $O(\epsilon)$  bias error and a variance error of the order of $O( \epsilon^{-(d/2+1)/2}n^{-1/2})$ up to log factors. As a result, at the optimal scaling of $\epsilon \sim n^{-1/(d/2+3)} $, this gives an overall error of 
\[
O( n^{-1/(d/2+3)} \sqrt{\log n \log (1/\epsilon)}).
\]
  The rate in Theorem \ref{thm:hatLn-2norm}  is the same as the point-wise convergence rate of Diffusion Map graph Laplacian in literature, see, e.g., in \cite{singer2006graph} when data density is uniform
   and  \cite[Theorem 6.2]{cheng2021eigen} for the $\alpha=1$ normalized operator when the density is non-uniform.
  In those works, the point-wise rate is proved for each location $x_i$ and equivalently the $\infty$-norm consistency of $\left( - \frac{1}{\epsilon} \hat{L}_n (\rho_X f)  \right) - \rho_X ( \Delta_p f) $.
The 2-norm consistency proved by Theorem \ref{thm:hatLn-2norm} proves is weaker than those $\infty$-norm consistency results. 
We will show in Section \ref{subsec:discuss-eigenvector} that the 2-norm consistency can lead to the 2-norm consistency of eigenvectors under additional assumptions.
  \end{remark}

  \begin{remark}[$W$ with zero-diagonal]
  \label{rk:clean-GL-W0}
One can verify that the same convergence as in Theorem \ref{thm:hatLn-2norm} holds if we replace $W$ with $W^0$ in computing the $(C_{\rm SK}, \varepsilon_{\rm SK})$-scaling problem and in constructing $\hat{L}_n$. The proof is included in Section \ref{sec:proofs}.
  \end{remark}

\subsection{Discussion on eigen-convergence}\label{subsec:discuss-eigenvector}

In this section, we discuss the extension of the point-wise convergence in Theorem \ref{thm:hatLn-2norm} to eigen-convergence, namely the convergence of the eigenvalues and eigenvectors of $\hat L_n$ to that of the limiting manifold operator (up to constant factors).
We show that the proved 2-norm point-wise consistency 
can lead to 2-norm consistency of eigenvectors under a crude consistency (error less than an $O(1)$ bound) of the eigenvalues. The argument follows the variational approach of eigen-convergence of graph Laplacian \cite{calder2019improved,cheng2021eigen} and considers the first $k_{\rm max}$ eigen-pairs for a fixed $k_{\rm max}$.

To be specific, by Theorem \ref{thm:hatLn-2norm}, we expect the eigen-convergence of $\frac{1}{\epsilon}\hat{L}_n $  to $- \Delta_p$. 
Suppose $(-\Delta_p) \psi_k = \mu_k \psi_k$.
We restrict to the case where $p \in C^\infty(\calM)$, and then the eigenfunctions  $\psi_k \in C^\infty(\calM) $ form an orthonormal basis in $L^2(\calM, p dV)$
where the inner-product is defined as $\langle f, g \rangle_p = \int_{\calM} f(x) g(x) p(x) dV(x) $. 
Since $-\Delta_p$ is a self-adjoint operator under the inner-product $\langle \cdot ,  \cdot  \rangle_p$
with non-negative eigenvalues, let the eigenvalues $\mu_k$ be sorted from small to large as $0 = \mu_1 \le \mu_2 \le \cdots$. Here $\mu_1 =0$ because the constant function is an eigenfunction with eigenvalue 0.
In summary, we have the population eigen-pairs as
\[
-\Delta_p \psi_k = \mu_k \psi_k, \quad k = 1,2, \cdots, 
\quad \psi_k \in C^\infty(\calM), 
\quad  \langle \psi_k, \psi_l \rangle_p = \delta_{kl},
\]
and let the empirical eigen-pairs be
\[
\frac{1}{\epsilon}\hat{L}_n u_k = \lambda _k u_k, \quad  k=1,\cdots, n,
\quad u_k^Tu_l = \delta_{kl},
\]
where $\lambda_k$ are also sorted from small to large.
We have $\lambda_1 = \mu_1 = 0$, where $u_1$  and $\psi_1$ are constant vector and function respectively.
Consider the case where $\mu_k$ are distinct for simplicity, noting that the argument can extend to the case where $\mu_k$ has higher multiplicities, see e.g. \cite{cheng2021eigen}.
Let $K:=k_{\rm max}+1$ and define 
\begin{equation}\label{eq:def-phik}
\phi_k : =  \frac{1}{\sqrt{n}} \rho_X ( \psi_k), \quad k=1,\cdots, K.
\end{equation}
We also define
\begin{equation}\label{eq:def-gamma-K}
\gamma_K : = \frac{1}{2} \min_{1 \le k \le k_{max}} (\mu_{k+1} - \mu_k),
\end{equation}
which is a positive constant determined by the first $K$ population eigenvalues.

\begin{corollary}\label{cor:eigenvector}
Under (A1)(A2), suppose $x_i \sim p$ i.i.d. on $\calM$, $i=1,\cdots, n$, $p \in C^\infty(\calM)$,
as $n \to \infty$, $\epsilon \to 0+$ and $\epsilon^{d/2+1} = \Omega( \log n /n)$.
Let $W$ be as in \eqref{eq:def-W},  $\hat \eta$ and $\varepsilon_{\rm SK}$ as in Theorem \ref{thm:hatLn-2norm}, 
and suppose there is $n_0$ s.t. when $n > n_0$, 
\begin{equation}\label{eq:eigenvalue-crude}
 | \lambda_k - \mu_k | < \gamma_K, \quad k =1, \cdots K.
 \end{equation}
 Then, for sufficiently large $n$, w.p.$>1- 6 n^{-9} - 2K n^{-10}$,
 there exist scalars $\alpha_k \neq 0$ satisfying $|\alpha_k| =1 + o(1)$, such that 
\[
 \| u_k - \alpha_k \phi_k \|_2  =  O \left(\epsilon, \sqrt{\frac{\log n \log (1/\epsilon) }{n \epsilon^{d/2+1}}} \right),
 \quad 1 \le k \le k_{max}.
\]
 \end{corollary}
 The corollary shows that the eigenvector 2-norm convergence has the same rate as the operator point-wise convergence in Theorem \ref{thm:hatLn-2norm}.
This type of eigenvector 2-norm consistency can further lead to eigenvalue convergence with rates
based on the Dirichlet form convergence as has been shown in \cite[Section 5.2]{cheng2021eigen}, which is omitted here.
 The proof of Corollary \ref{cor:eigenvector} follows the same argument of \cite{cheng2021eigen} and we include a proof in Appendix \ref{app:more-prooofs}  for completeness.
 The condition \eqref{eq:eigenvalue-crude} was verified in \cite{cheng2021eigen} (for the traditional graph Laplacian) by showing an eigenvalue upper bound and lower bound using a variational approach. 
 In this work, our technique of comparing $\hat \eta$ to $\bar \eta$ can be combined with the analysis therein to show an eigenvalue upper bound for the bi-stochastic graph Laplacian. However, the $O(1)$ lower-bound would be difficult to obtain using existing techniques.
 We leave the further investigation of eigen-convergence to future work.

\section{On data with outlier noise}\label{sec:theory-noise-outlier}

Consider the following manifold-data plus noise model of sample $x_i \in \R^m$, 
\begin{equation}
x_i = x_i^c + \xi_i,  \quad i=1, \cdots, n,
\end{equation}
where $x_i^c \sim p $ i.i.d. and lies on a $d$-dimensional manifold $\calM$ and superscript $^c$ stands for `clean data', and $\xi_i$ is an $m$-dimensional noise vector. In this section, we consider the setting where $\xi_i$ admits an outlier noise distribution, and theoretically prove the robustness of bi-stochastic normalized graph Laplacian with respect to outlier noise in data. 
All proofs are deferred to Section \ref{sec:proofs}.

\subsection{Kernel matrix on data with outlier noise}

We construct the kernel matrix $W$ from $x_i$ as  in \eqref{eq:def-W} except for that we require $W_{ii} = 0$.
For $i \neq j$,
\begin{align}
\| x_i - x_j\|^2
& = \| x_i^c - x_j^c \|^2 - 2( x_i^c - x_j^c)^T (\xi_i - \xi_j) - 2 \xi_i^T\xi_j + \| \xi_i \|^2 +  \| \xi_j \|^2  \nonumber \\
& = \| x_i^c - x_j^c \|^2 + \| \xi_i \|^2 +  \| \xi_j \|^2 + r_{ij},  \nonumber \\
r_{ij} & = - 2( x_i^c - x_j^c)^T (\xi_i - \xi_j) - 2 \xi_i^T\xi_j.
  \label{eq:xi-xj-2-1}
\end{align}
As will be shown in the analysis later, for the robustness of graph Laplacian to outlier noise to hold  it suffices to have the noise vector $\xi_i$ being zero  (then the sample $x_i$ is an in-lier) with some positive probability uniformly bounded from zero plus a uniform boundedness of $|r_{ij}|$ for off-diagonal entries. This is specified in the following condition on noise vector $\xi_i$:

\begin{assumption}[A3]
\label{assump:A3}
 The noise vectors $\xi_i = b_i z_i$,  
 the law of $x_i^c$, $b_i$ and $z_i$ are such that $(x_i^c, b_i, z_i)$ are i.i.d. across $i$,  and 
 
 (i) $b_i | x_i^c \sim {\rm Bern}( p_i ) $, 
 $p_i$ may depend on $x_i^c$,
 and for some constant $p_{\rm out}$, 
 $0 \le p_i \le  p_{\rm out} < 1$ for all $i$.

 (ii) In the joint asymptotic regime of large $n$ and small $\epsilon$ being considered,
for large enough $n$, there is a good event $E_{(z)}$ over the randomness of all  $x_i^c$, $b_i$ and $z_i$
that happens w.p.$>1-\delta_z$, under which $r_{ij}$ defined as in \eqref{eq:xi-xj-2-1} satisfies that 
\begin{equation}\label{eq:bound-rij-A3}
\sup_{ i \neq j, 1 \le i,j \le n} | r_{ij} | \le \varepsilon_z, \quad \text{ and }\varepsilon_z= o(\epsilon).
\end{equation}
\end{assumption}

Note that when $b_i = b_j =0$,  $r_{ij} = 0$. The bound \eqref{eq:bound-rij-A3} may be fulfilled by various noise models.
In particular, heteroskedasstic noise is allowed where the conditional distribution $b_i, z_i | x_i^c$ may depend on $x_i^c$.
The i.i.d. of the $(x_i^c, b_i, z_i)$ and condition (i) in (A3) may be further relaxed as long as \eqref{eq:Wc0-C3-claim} (in the proof of Lemma \ref{lemma:hatetac-upper}) 
and \eqref{eq:bound-rij-A3} can hold. We keep the current assumption for simplicity.

\begin{remark}[High dimensional noise]\label{rk:varepsilon-z}
We give an exemplar noise model that fulfills Assumption (A3) and allows heteroskedasticity:
Let $z_i = \sigma( x_i^c ) g_i$, where $g_i \sim \calN(0, (1/m) I_m)$ independently from $x_i^c, b_i$, and $\sigma: \calM \to \R$ is a uniformly bounded function. 
One can let  $b_i \sim {\rm Bern}(p_{\rm out})$ i.i.d. for $0 \le p_{\rm out} < 1$, or make $b_i$  heteroskedastic  as well. 
For each $i \neq j$, the noise-noise cross term in  \eqref{eq:xi-xj-2-1} if not vanishing equals $z_i^T z_j = \sigma( x_i^c )\sigma( x_j^c )  g_i^T g_j$, which, due to the uniform boundedness of $\sigma$, satisfies that  $|z_i^T z_j | \le \Theta(  \sqrt{ \gamma \log m/m  } )$ w.p.$>1- 2m^{-\gamma}$ for any positive constant $\gamma$. 
The data-noise cross term in \eqref{eq:xi-xj-2-1} consists of inner-products like $(x_i^c)^T z_k =  \sigma( x_k^c )  (x_i^c)^T g_k$, and is again bounded by $ \Theta(  \sqrt{ \gamma \log m/m  } )$ with the same high probability even when $k=i$, because $g_k $ is independent from $x_i^c$'s, the manifold is compact and thus has bounded diameter and $\sigma$ is bounded.
This leads to $\varepsilon_z \sim \sqrt{ {\gamma \log m }/{m} }$ and $\delta_z \sim n^2 m^{-\gamma} $ by taking a union bound over $i \neq j$.
Then $\varepsilon_z = o(\epsilon)$ and the smallness of $\delta_z$ can be achieved given the proper choice of $\gamma$ and sufficiently fast increase of $m$ as $n$ increases and $\epsilon$ decreases.
The robustness of bi-stochastically normalized graph Laplacian for this type of outlier noise is supported by experiments in Section \ref{sec:exp}.
The argument also generalizes to non-normal vectors. e.g., when $g_i$ has independent entries of finite moments or is uniformly distributed on high-dimensional hypersphere. Generally, the bound \eqref{eq:bound-rij-A3} can hold whenever a concentration-type argument applies to bound $|r_{ij}|$ uniformly small as the ambient dimensionality $m$ increases, and the $O(m^{-1/2})$ decay of error was also previously shown in \cite{landa2021doubly}.
\end{remark}

Let the kernelized affinity matrix $W$ be defined as  in \eqref{eq:def-W}, and we denote the matrix by setting the diagonal entries of $W$ to be zero as $W^{0}$,  that is
\begin{equation}\label{eq:def-W0}
W = W^{0} + \beta_n I, \quad 
\beta_n 
 = \frac{\epsilon^{-d/2}}{n } \frac{1}{(4\pi)^{d/2}}.
\end{equation}
For $i \neq j$, note that 
\begin{align}
g \left( \frac{\|x_i - x_j\|^2}{\epsilon} \right) 
& =  \frac{1}{(4\pi)^{d/2}}\exp \left\{ - \frac{ \| x_i - x_j\|^2 }{4 \epsilon} \right\}  \nonumber \\
& =  g \left( \frac{\|x_i^c - x_j^c\|^2}{\epsilon} \right)  
  e^{- \frac{\| \xi_i \|^2}{4\epsilon}}  e^{- \frac{ \| \xi_j \|^2}{4 \epsilon} }  e^{ -\frac{r_{ij} }{4 \epsilon} }. \label{eq:gij-expand-1}
\end{align}
Next, we define $W^c$ as the affinity matrix made from clean data vectors $x^c_i$, that is,
\begin{equation}
W^c_{ij } = \frac{\epsilon^{-d/2}}{n} g \left(  \frac{ \| x_i^c - x_j^c\|^2}{\epsilon} \right),
\end{equation}
and define $W^{c,0}$ by setting-zero the diagonal entries of $W^c$, i.e.,
\begin{equation}\label{eq:def-Wc0}
W^c = W^{c,0} +\beta_n I.
\end{equation}
By \eqref{eq:gij-expand-1}, we have for $i \neq j$,
\begin{equation}\label{eq:Wij0-relation}
W_{ij}^{0} 
= \left(  e^{- \frac{\| \xi_i \|^2}{4\epsilon}}   W_{ij}^c e^{- \frac{ \| \xi_j \|^2}{4 \epsilon} }  \right) e^{ -\frac{r_{ij} }{4 \epsilon} }.
\end{equation}
Define  $\rho \in \R_+^n$ as
\begin{equation}\label{eq:def-rho}
\rho_i = e^{-\frac{\|\xi_i\|^2}{4\epsilon}}, \quad i=1, \cdots, n,
\end{equation}
and define  $H$ to be a real symmetric matrix  such that 
\begin{equation}\label{eq:def-Hij}
e^{ -\frac{r_{ij} }{4 \epsilon} } = 1 + H_{ij },  \quad H_{ii } = 0.
\end{equation}
At last, define 
\begin{equation}\label{eq:def-W'}
W':= D_\rho W^{c,0} D_\rho,
\end{equation}
which is a real symmetric matrix with zero diagonal entries,
and then  by \eqref{eq:Wij0-relation} we also have 
\begin{equation}\label{eq:W0-and-W'}
W^0 = W' + W' \odot H.
\end{equation}
The following lemma shows the approximation to $W^0$ by $W'$  under (A3).

\begin{lemma}[Comparing $W^0$  to $W'$]
\label{lemma:uniform-H}
Under Assumption (A3), for large enough $n$ such that $\varepsilon_z/ \epsilon < 0.1$ and under the good event $E_{(z)}$ in (A3),
\begin{equation}\label{eq:W0-W'-relation}
 \sup_{ 1 \le i,j \le n}|H_{ij}| \le \frac{\varepsilon_z}{\epsilon}.
\end{equation}
\end{lemma}

\subsection{Modified Sinkhorn on data with outlier noise }
\label{subsec:modified-sinkhorn-noise}

We solve for the $(C_{\rm SK},\varepsilon_{\rm SK})$-scaling factor problem of the diagonal-zero-out matrix $W^0$ defined in \eqref{eq:def-W0}, where $W$ is built from noisy data, and throughout this section we set 
\begin{equation}\label{eq:def-epsSK-noise}
C_{\rm SK } = q_{\rm min},
\quad
\varepsilon_{\rm SK} = O \left( \epsilon^2, \sqrt{\frac{\log n }{n \epsilon^{d/2}}}, \frac{\varepsilon_z}{\epsilon} \right).
\end{equation}
Define the scaling factor $\bar{\eta} \in \R_+^n$ such that 
\begin{equation}\label{eq:def-bareta-noise}
\bar{\eta} \odot \rho = \bar{\eta}^c  : = \rho_X( q_\epsilon).  
\end{equation}
We first show in next lemma that  with high probability
$\bar{\eta}$ gives a valid solution to the $(C_{\rm SK} ,\varepsilon_{\rm SK})$-matrix scaling problem. 

\begin{lemma}[Population scaling factor $\bar{\eta}$ for noisy data]
\label{lemma:bareta-noise}
Under Assumptions (A1)(A2)(A3), 
suppose as $n \to \infty$,  $\epsilon \to 0+$, $\epsilon^{d/2} = \Omega(\log n/n)$,
then when $n$ is large 
under the intersection of a good event $E_{(1)}$ 
(over the randomness of $x_i^c$'s only)
which happens w.p. $>1-2 n^{-9}$  and $E_{(z)}$ in (A3),
\begin{equation}\label{eq:bareta-noise-existence}
D_{\bar{\eta}} W^0 D_{\bar{\eta}} {\bf 1} = 1 + e(  \bar{\eta} ), 
\quad  
\| e(  \bar{\eta} ) \|_\infty = O \left( \epsilon^2, \sqrt{\frac{\log n }{n \epsilon^{d/2}}}, \frac{\varepsilon_z}{\epsilon}  \right).
\end{equation}
\end{lemma}

\subsection{Robustness of graph Laplacian to outlier noise}

Suppose $\hat{\eta} \in \R_+^n$ is a solution of the $(C_{\rm SK} ,\varepsilon_{\rm SK})$-matrix scaling problem of $W^0$ found by Algorithm \ref{algo:SK-early},
and the bi-stochastically normalized graph Laplacian is defined as
\[
\hat{L}:= D(\hat{W} ) - \hat{W}, \quad \hat{W} : = D_{\hat{\eta}} W^0 D_{\hat{\eta}}.
\]
The analysis proceeds in three steps,

\begin{itemize}

\item 
Step 1. Analyzing the graph Laplacian associated with $\bar{W}' : = D_{\bar{\eta}} W' D_{\bar{\eta}}$. This is the same analysis as in Proposition \ref{prop:barLn-pointwise} except that we use the zero-diagonal matrix $W^{c,0}$ to replace $W^c$. 

\item 
Step 2. Comparing the  graph Laplacian associated with $\bar{W}'$ to that associated with  $\hat{W}' : = D_{\hat{\eta}} W' D_{\hat{\eta}}$. This is similar to the proof of Theorem \ref{thm:hatLn-2norm}, where one needs to replace the population scaling factor $\bar{\eta}$ with the empirical one $\hat{\eta}$ and control the error.

\item
Step 3. Comparing the  graph Laplacian associated with $\hat{W}'$ to that associated with  $\hat{W}  = D_{\hat{\eta}} W^0 D_{\hat{\eta}}$. This step makes use of the bound of the residual $(W^0-W')$ proved in Lemma \ref{lemma:uniform-H}.
\end{itemize}

Recall that $D(W)$ stands for the degree matrix of $W$, and we define the notations
\begin{align}\label{eq:def-three-Ls-noise}
\bar{L}':= D(\bar{W}' ) - \bar{W}',
\quad
\hat{L}':= D(\hat{W}' ) - \hat{W}',
\quad
\hat{L}:= D(\hat{W} ) - \hat{W}.
\end{align}

\subsubsection{Step 1. Convergence of $\bar{L}'$}

By definition of $W'$ in \eqref{eq:def-W'} and $\bar{\eta}^c$ in \eqref{eq:def-bareta-noise}, 
\begin{equation}\label{eq:barW'-2}
\bar{W}'  = D_{\bar{\eta}} W' D_{\bar{\eta}} 
= D_{\bar{\eta}} D_\rho W^{c,0} D_\rho D_{\bar{\eta}} 
= D_{\bar{\eta}^c}  W^{c,0}  D_{\bar{\eta}^c}.
\end{equation}
Comparing to \eqref{eq:def-barL}, 
$\bar{W}'$ equals  $\bar{W}$ therein by setting zero the diagonal entries,
which does not affect the definition of graph Laplacian, that is, 
\[
\bar{L}' 
= D(\bar{W}' ) - \bar{W}' 
= D(\bar{W}) - \bar{W} =  \bar{L}_n \quad \text{in \eqref{eq:def-barL}.}
\]
As a result, Proposition \ref{prop:barLn-pointwise} gives the same convergence result for $\bar{L}'$ as follows (no proof is needed):

\begin{proposition}\label{prop:barL'-convergence}
Under the same asymptotic condition of $n$ and $\epsilon$ as in Proposition \ref{prop:barLn-pointwise},
for $f \in C^4(\calM)$ and  large enough $n$,
under a good event  $E_{(0)}$ (over the randomness of $x_i^c$'s only) which happens w.p.$>1-2n^{-9}$,
\begin{equation}\label{eq:barLnprime-Errpt}
\| \left( - \frac{1}{\epsilon}\bar{L}' \rho_X f  \right) -  \rho_X ( \Delta_p f )   \|_\infty 
= O \left( \epsilon, \sqrt{ \frac{\log n }{n \epsilon^{d/2+1}}}\right),
\end{equation}
and the constant in big-O depends on $(\calM, p)$ and $f$. 
\end{proposition}

\subsubsection{Step 2. Convergence of $\hat{L}'$}

Analogously to \eqref{eq:def-bareta-noise}, for the empirical scaling factor $\hat{\eta}$ of $W^0$, we define
\begin{equation}\label{eq:def-hateta-noise}
\hat{\eta}^c  : = \hat{\eta} \odot \rho.
\end{equation}
Then similarly to \eqref{eq:barW'-2} we have that
\begin{equation}\label{eq:hatW'-2}
\hat{W}'  = D_{\hat{\eta}} W' D_{\hat{\eta}}
= D_{\hat{\eta}^c}  W^{c,0}  D_{\hat{\eta}^c}.
\end{equation}
To extend the analysis in Theorem \ref{thm:hatLn-2norm}, one would want to use the closeness between $\hat{\eta}^c$ and $\bar{\eta}^c$.
However, note that while $\bar{\eta}^c$ is uniformly bounded from below and above by $O(1)$ constants same as before, 
for $\hat{\eta}^c$ we no longer have the lower boundedness nor the upper boundedness directly:
while $\hat{\eta}_i \ge C_{\rm SK}$ uniformly but $\hat{\eta}^c_i = \hat{\eta}_i \rho_i$ and $\rho_i$ can be exponentially small when $b_i =1$. 
Thus the proof in Section \ref{subsec:convergence-hatL} based on the two-sided boundedness of the scaling factors does not directly apply and needs modification. 

To proceed, we first establish a uniform $O(1)$ upper-bound of $\hat{\eta}^c$ in the following lemma.

\begin{lemma}[Upper boundedness of $\hat{\eta}^c$]
\label{lemma:hatetac-upper}
Under (A1)(A2)(A3), suppose as $n\to \infty$,  
$\epsilon \to 0+$, $\epsilon^{d/2} = \Omega(\log n /n)$.
Then there is $O(1)$ constant  $C_2 >0$ s.t.
for large $n$ and under the intersection of $E_{(z)}$ and another good event $E_{(2)}$ 
(over the randomness of $(x_i^c, b_i)$'s only)
which happens w.p. $>1-2n^{-9}$,
\[
\max_i \hat{\eta}^c_i \le C_2 .
\]
\end{lemma}

Recall that by the construction of $\hat{\eta}$ and the property of $\bar{\eta}$ in Lemma \ref{lemma:bareta-noise},
for large $n$ and under the needed good events,
\begin{equation}\label{eq:hateta-bareta-epsSK}
D_{\hat{\eta}} W^0 D_{\hat{\eta}}  {\bf 1} = {\bf 1} + e(\hat{\eta}),
\quad
D_{\bar{\eta}} W^0 D_{\bar{\eta}}  {\bf 1} = {\bf 1} + e(\bar{\eta}),
\quad 
\| e(\hat{\eta}) \|_\infty,  \| e(\bar{\eta}) \|_\infty \le \varepsilon_{\rm SK}.
\end{equation}
Define $u \in \R^n$ by 
\begin{equation}\label{eq:def-u-hateta-bareta}
\hat{\eta} = \bar{\eta} \odot (1+u), 
\quad   \text{where by $\hat{\eta}, \bar{\eta} \in \R_+^n$, } 1+u_i >0.
\end{equation}
The following lemma bounds $u$ in 2-norm as an extension of Lemma \ref{lemma:eta1-eta2},
and it gives the comparison of $\hat{\eta}$ to $\bar{\eta}$
and equivalently that of $\hat{\eta}^c$ to $\bar{\eta}^c$.

\begin{lemma}[Comparison of $\hat{\eta}$ to $\bar{\eta}$] \label{lemma:hatetac-baretac}
Under (A1)(A2)(A3), suppose as $n\to \infty$,  
$\epsilon \to 0+$, $\epsilon^{d/2} = \Omega(\log n /n)$.
Then for large $n$ s.t. $\varepsilon_{\rm SK} < 0.1$,
 and under the intersection of good events 
$E_{(z)}$ in (A3), 
$E_{(1)}$ in Lemma \ref{lemma:bareta-noise}, 
and 
$E_{(2)}$ in Lemma \ref{lemma:hatetac-upper}, 
let constants $q_{\rm min}$ as in Lemma \ref{lemma:construct-bar-eta}(i) and $C_2$ as in Lemma \ref{lemma:hatetac-upper},
\begin{equation}
\| u \|_2 \le \frac{C_2}{0.4 q_{\rm min}} \sqrt{n}   \varepsilon_{\rm SK}.
\end{equation}
\end{lemma}

We are ready to prove the central result in Step 2.

\begin{proposition}\label{prop:hatL'-convergence}
Under Assumptions (A1)(A2)(A3), suppose as $n \to \infty$, $\epsilon \to 0+$ and $\epsilon^{d/2} = \Omega( \log n /n)$.
Then for any $f \in C^4(\calM)$, 
with large $n$  and under the intersection of good events 
$E_{(z)}$ in (A3), 
$E_{(1)}$ in Lemma \ref{lemma:bareta-noise}, 
and 
$E_{(2)}$ in Lemma \ref{lemma:hatetac-upper}, 
 \[
\|     \hat{L}' (\rho_X f)   - \bar{L}' (\rho_X f)  \|_2
 = O \left(  \sqrt{\epsilon   \log (1/\epsilon) } \right)  \sqrt{n}  \varepsilon_{\rm SK},
 \]
and the constant in big-O depends on $(\calM, p)$ and $f$.
\end{proposition}

\subsubsection{Step 3. Convergence of $\hat{L}$}

\begin{proposition}\label{prop:hatL-convergence}
Under the same condition as Proposition \ref{prop:hatL'-convergence},
for any $f \in C^4(\calM)$, for large $n$ and under the intersection of $E_{(z)}$, $E_{(1)}$ and $E_{(2)}$ as therein,
 \[
\|     \hat{L} (\rho_X f)   - \hat{L}' (\rho_X f)  \|_\infty
 =O \left(  \sqrt{ \frac{  \log (1/\epsilon) } {\epsilon } } \right)   {\varepsilon_z},
 \]
and the constant in big-O depends on $(\calM, p)$ and $f$.
\end{proposition}

We are ready to prove the main theorem of this section, combining the three steps.

\begin{theorem}\label{thm:hatL-convergence-noise}
Under Assumptions (A1)(A2)(A3), suppose as $n \to \infty$, $\epsilon \to 0+$ and $\epsilon^{d/2+1} = \Omega( \log n /n)$.
Then for any $f \in C^4(\calM)$, 
when $n$ is large, w.p. $> 1- 6n^{-9} - \delta_z$,
 \begin{equation}\label{eq:bound-thm-hatL-noise}
\frac{1}{\sqrt{n}} \|  \left( - \frac{1}{\epsilon} \hat{L} (\rho_X f)  \right) - \rho_X ( \Delta_p f) \|_2
 =    O \left(\epsilon, 
 	\sqrt{\frac{\log n \log (1/\epsilon) }{n \epsilon^{d/2+1}}}, 
 	 \frac{\varepsilon_z}{\epsilon}\sqrt{\frac{  \log (1/\epsilon) }{ \epsilon}}  \right).
 \end{equation}
The constant in big-O depends on $(\calM, p)$ and $f$.
\end{theorem}

Comparing to the bound with clean manifold data in Theorem \ref{thm:hatLn-2norm}, we see that under the noise model in this paper the outlier noise brings another 
$O( {\varepsilon_z}{\epsilon}^{-3/2}\sqrt{ \log (1/\epsilon) }  )$ term in the error, which can be made small if $\varepsilon_z$ is sufficiently small.

\section{Experiments}\label{sec:exp}

In this section, we numerically compute the (approximately) bi-stochastically normalized graph Laplacian on simulated clean manifold data 
and with additive outlier noise.
Codes of the experiments are released at \url{https://github.com/xycheng/bistochastic_graph_laplacian}.

\subsection{Clean manifold data}\label{subsec:exp-clean}

\begin{figure}[t]
\centering
\includegraphics[height=.24\linewidth]{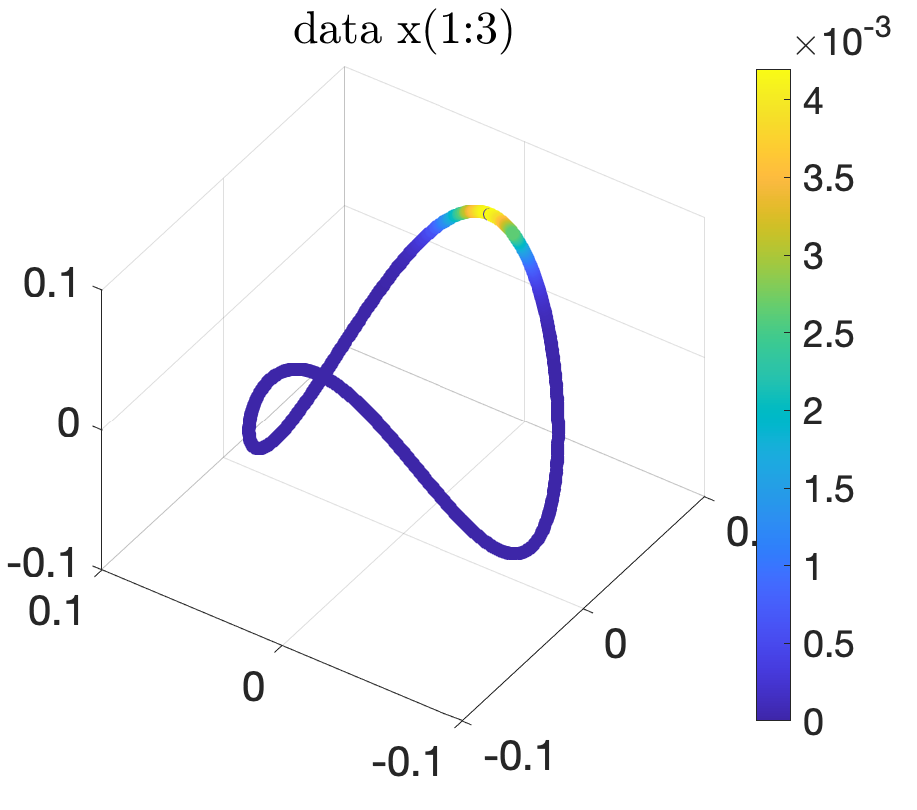}
\includegraphics[height=.24\linewidth]{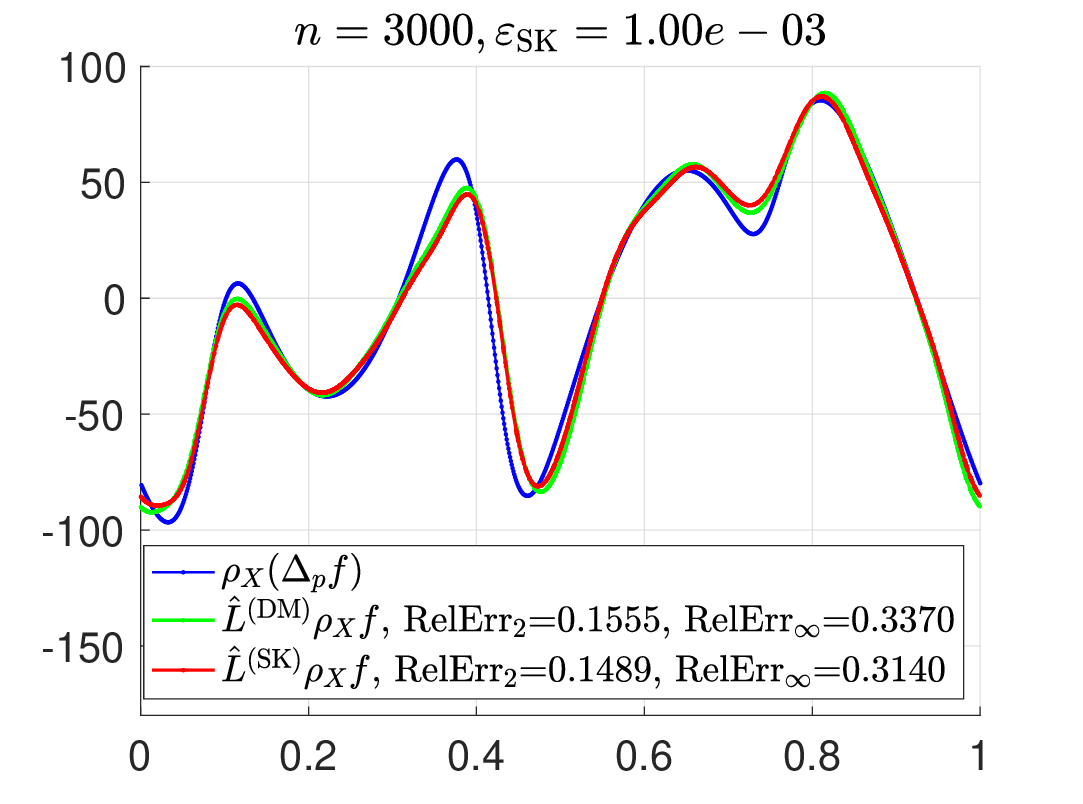}
\includegraphics[height=.24\linewidth]{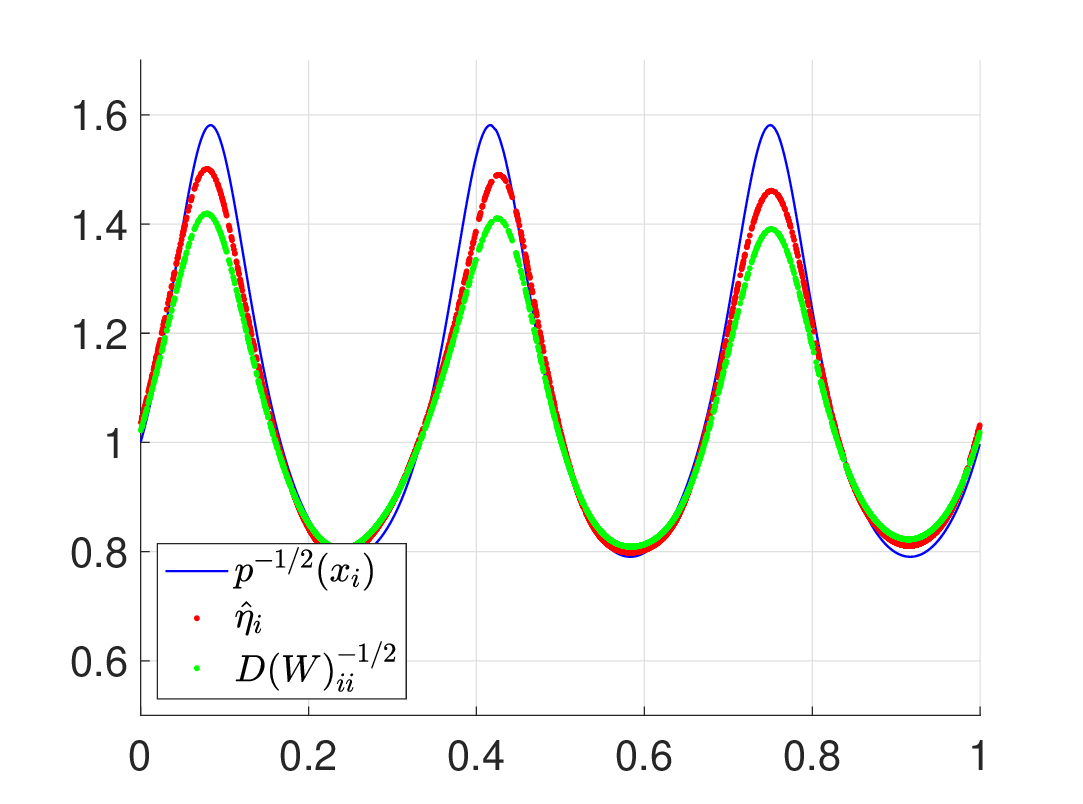}\\
\includegraphics[height=.24\linewidth]{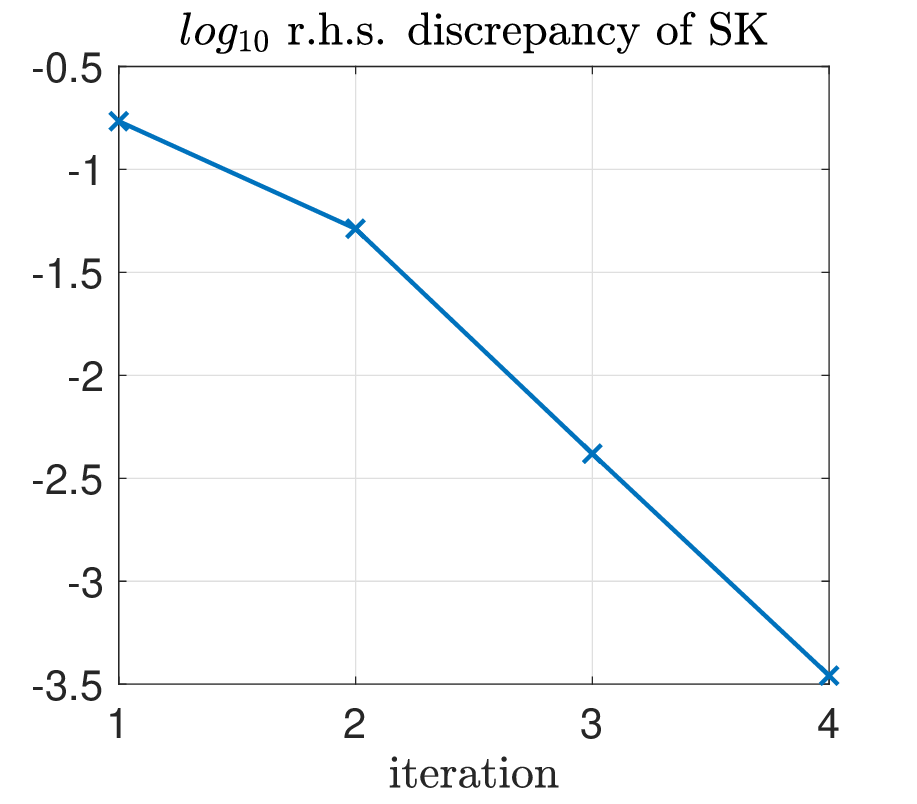}
\includegraphics[height=.24\linewidth]{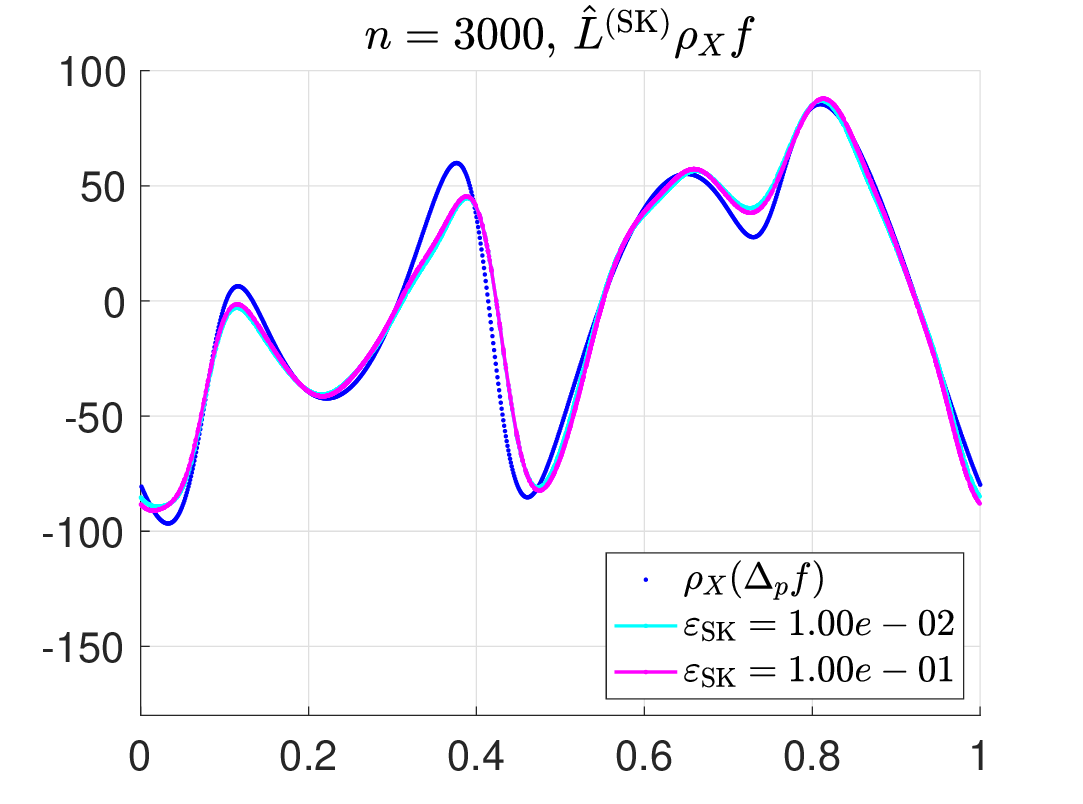}
\includegraphics[height=.24\linewidth]{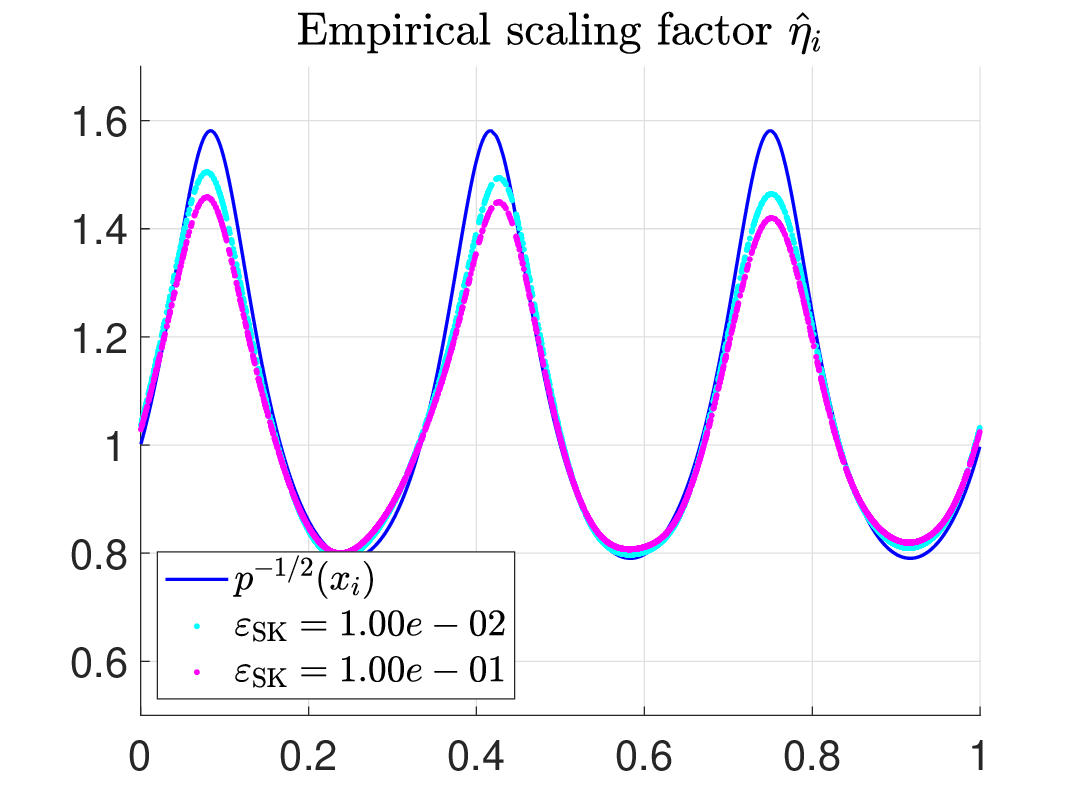}
\vspace{-5pt}
\caption{
\small
The results of one simulation of clean manifold data. 
The kernel bandwidth parameter is $\epsilon=$ \texttt{5.0119e-4}.
The averaged error over multiple simulations is shown in Figure \ref{fig:Ln-error-1d}.
Top panel: 
(Left) Data samples lying on a one-dimensional closed curve embedded in $\R^4$, where the first three coordinates are shown and colored by the values of kernel affinity $W_{i_0 j}$ on $x_j$ for a fixed $i_0$.
(Middle) Computed values of $\hat{L} \rho_X f$ compared with the true values of $\Delta_p f$.
(Right)  Comparison of $\hat{\eta}_i$ and $p^{-1/2}(x_i)$.
Bottom panel:
Results of bi-stochastic normalization computed with different  $\varepsilon_{\rm SK}$.
(Left) The convergence of SK iterations, showing the value of $\log_{10}  \| D_{\hat{\eta}} W^0 D_{\hat{\eta}} {\bf 1} - {\bf 1} \|_\infty$ v.s. number of iterations;
(Middle-right) Computed values of $\hat{L}^{(\rm SK)} \rho_X f$ and of $\hat{\eta}_i$.
All plots where $x$-axis is $[0,1]$ are plotted v.s. the intrinsic coordinate of $x_i$ (the arclength). 
}
\label{fig:Ln-1d-one-run}
\end{figure}

\paragraph{Data simulation.}
The simulated dataset consists of $n=3000$ data points on a 1D closed curve embedded in $\R^4$, where the density $p$ is not uniform on the manifold. 
An illustration of the data samples is shown in the left of the top panel of Figure \ref{fig:Ln-1d-one-run}.
We construct a specific smooth function $f$ on the manifold, and using the intrinsic coordinate of the curve, the function $\Delta_p f$ can be analytically computed. 
The functions $p$ and $f$ are shown in Figure \ref{fig:data-1d}, and additional details of data generation are in Appendix \ref{app:exp-detail}.

\paragraph{Method.}
We compute and compare two types of graph Laplacians, namely,

\begin{itemize}
\item[(a)] The bi-stochastically normalized graph Laplacian as defined in \eqref{eq:def-hatL}, denoted as $\hat{L}^{(\rm SK)}$ in this section.
The superscript $^{(\rm SK)}$ stands for `Sinkhorn-Knopp'.

\item[(b)] The $\alpha= {1}/{2}$ normalized graph Laplacian $\hat{L}^{(\rm DM)}$ as introduced in \cite{coifman2006diffusion}, see \eqref{eq:def-hatL-DM}.
The superscript $^{(\rm DM)}$ stands for `Diffusion Map'. More details about computing $\hat{L}^{(\rm DM)}$ is given in Appendix \ref{app:exp-detail}.
\end{itemize}

For a graph Laplacian $\hat{L}$ being either (a) or (b) above, the relative error of the point-wise convergence of the graph Laplacian is measured by 
\begin{equation}\label{eq:def-L1-err}
\text{RelErr}_2:=  \frac{\| ( -\frac{1}{\epsilon}\hat{L} \rho_X f )-  \rho_X (\Delta_p f) \|_2}{\|\rho_X (\Delta_p f) \|_2}
\quad
\text{RelErr}_\infty:=  \frac{\| ( -\frac{1}{\epsilon}\hat{L} \rho_X f )-  \rho_X (\Delta_p f) \|_\infty}{\|  \rho_X (\Delta_p f) \|_\infty},
\end{equation}
which can be computed once $n$ data samples $X = \{x_i\}_{i=1}^n$ are drawn. 
We conduct the simulations for fixed $n$ and varying values of the kernel bandwidth parameter $\epsilon$. 
In computing  $\hat{L}^{(\rm SK)}$, the approximate bi-stochastic normalization is implemented by SK iterations until the right-hand-side (r.h.s.) discrepancy in $\infty$-norm is less than a tolerance $\varepsilon_{\rm SK}  =$ \texttt{1e-3}.
We also simulate 500 replicas of the experiment (by simulating the dataset 500 times) and compute the mean of the two relative errors in \eqref{eq:def-L1-err}.

\begin{figure}[t]
\centering
\includegraphics[height=.26\linewidth]{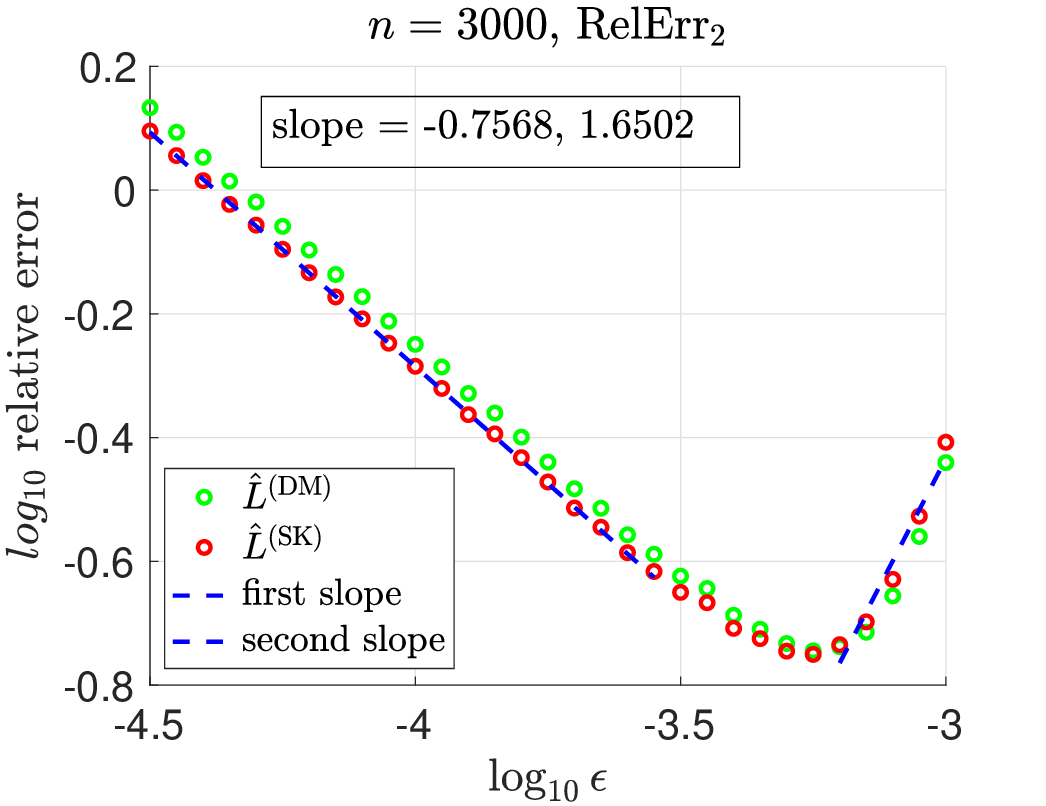}
\includegraphics[height=.26\linewidth]{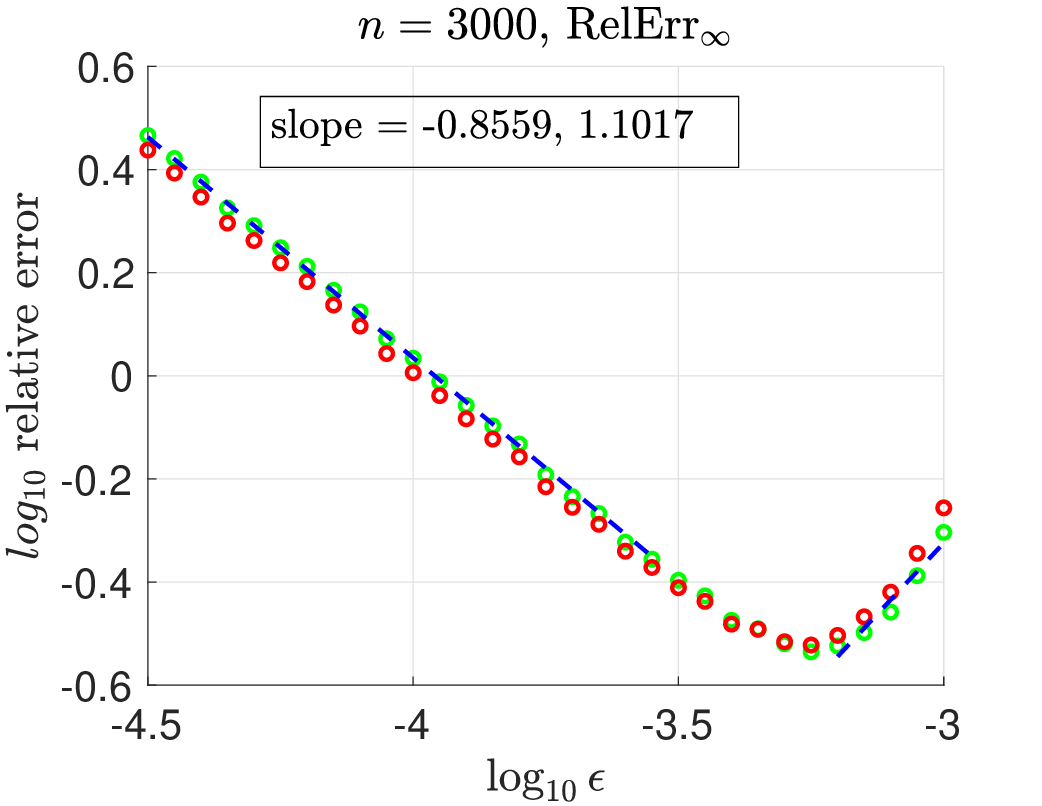}
\vspace{-5pt}
\caption{
\small
Average errors on manifold clean data.
Relative errors $\text{RelErr}_2$  as defined in \eqref{eq:def-L1-err}  of $\hat{L} \rho_X f$ computed by
(i) $\alpha$-normalized diffusion map graph Laplacian $\hat{L}^{(\rm DM)}$ 
and (ii) bi-stochastic normalized graph Laplacian $\hat{L}^{(\rm SK)}$ respectively.
plotted v.s. a range of values of $\epsilon$,
averaged over 500 replicas of simulation.
The two blue dashed lines fit the data on the two ends of small and large values of $\epsilon$,
where variance and bias error dominates respectively. 
(Right) Same plot for $\text{RelErr}_\infty$.
}
\label{fig:Ln-error-1d}
\end{figure}

\paragraph{Results.}
The result for one simulation is shown in Figure \ref{fig:Ln-1d-one-run}. For the SK iterations used to compute bi-stochastic normalizing factors, 
the algorithm typically terminates in about 5 iterations (the left plot in the bottom panel), starting from an all-one initialization of the scaling factor.
In all simulations here, we also observe that the lower-boundedness of empirical scaling factor $\hat{\eta}_i$ by a pre-specified constant $C_{\rm SK} = 0.5$ 
 is always fulfilled throughout the SK iterations. 
The results using a reduced number of SK iterations, equivalently with a larger tolerance $\varepsilon_{\rm SK}$, are shown in the bottom panel. 
For $\hat{L}^{(\rm SK)}$, the approximation of $\hat{L} \rho_X f$ to $\Delta_p f$ appears to hold with furtherly reduced number of iterations. 

The averaged errors are
plotted in log-log v.s. $\epsilon$ in Figure \ref{fig:Ln-error-1d}. 
In the left plot of 2-norm error, when the variance error dominates (at the end of small $\epsilon $) the error shows a scaling of about -0.75, which is consistent with the $\epsilon^{-(d/2+1)/2}$ rate in Theorem \ref{thm:hatLn-2norm} as $d=1$ here. 
At the large $\epsilon$ end when bias error dominates, the 2-norm error shows a scaling about 1.6 which cannot be explained by the theory. 
For the $\infty$-norm relative error, the $O(\epsilon)$ rate of bias error is revealed in the plot. 
Overall, these experimental results support the theory in Section \ref{sec:theory-clean-data}.
Note that while Theorem \ref{thm:hatLn-2norm} proves the 2-norm consistency only, 
the $\infty$-norm consistency of  $\hat{L}^{(\rm SK)}$ appears to hold as well, 
and the error is also close to that of $\hat{L}^{(\rm DM)}$.

As a remark, here we use un-normalized graph Laplacians taking the form as $L_{\rm un}:= D(K)-K$ and $K$ is the graph affinity matrix constructed in different ways. 
One may also consider normalized graph Laplacian  as $L_{\rm rw}:= I - D(K)^{-1}K$, the subscript $_{\rm rw}$ standing for `random walk'.
For exact bi-stochastic normalized affinity matrix $K$, $D(K) = I$, and the two are the same.
On clean manifold data, we found that using $L_{\rm rw}$, for both $\hat{L}^{(\rm SK)}$ and $\hat{L}^{(\rm DM)}$, gives similar results as in Figures  \ref{fig:Ln-1d-one-run} and \ref{fig:Ln-error-1d}.

\begin{figure}[t]
\centering
\includegraphics[height=.24\linewidth]{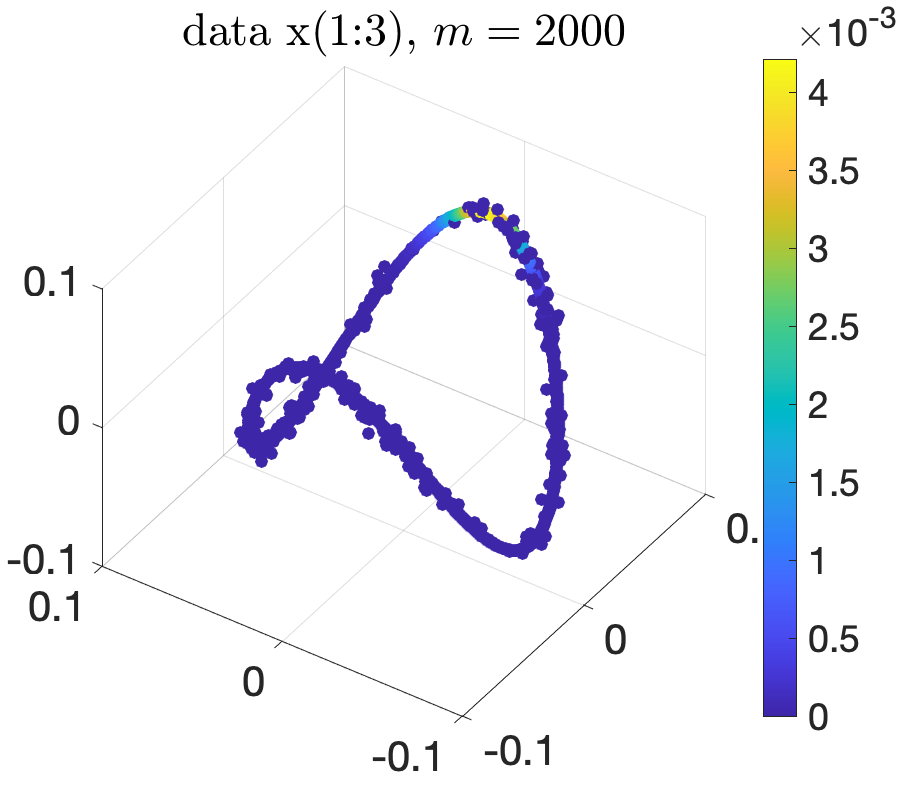}
\includegraphics[height=.24\linewidth]{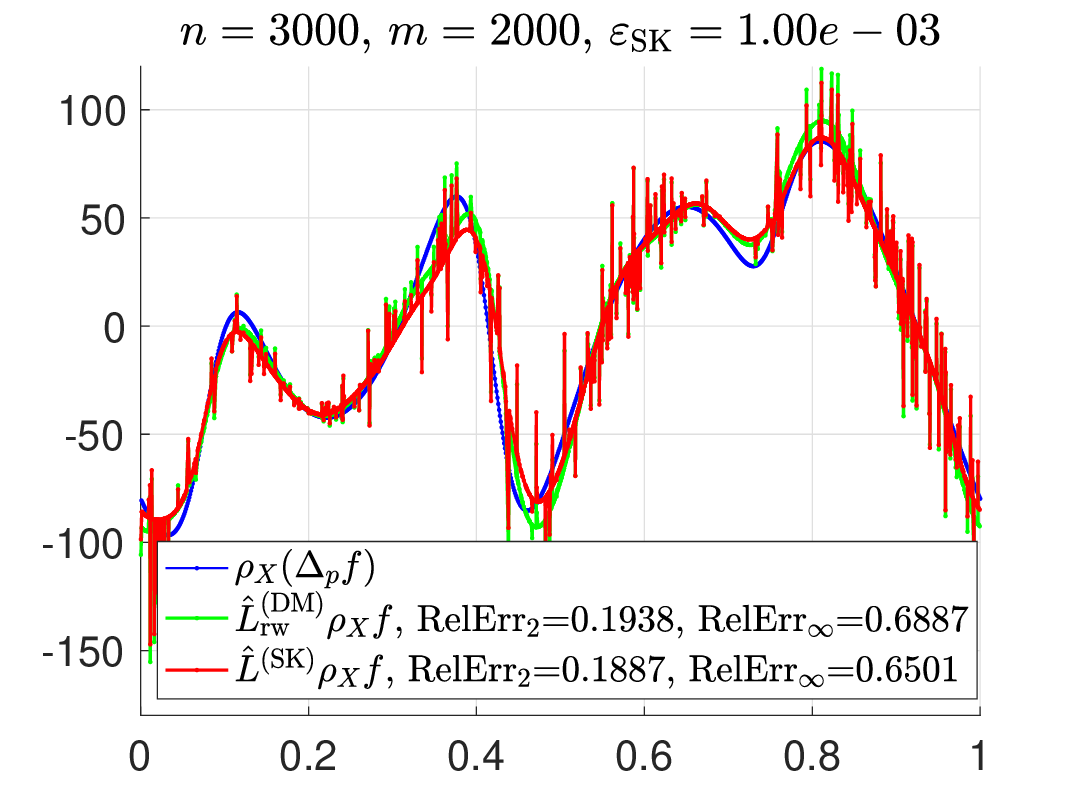}
\includegraphics[height=.24\linewidth]{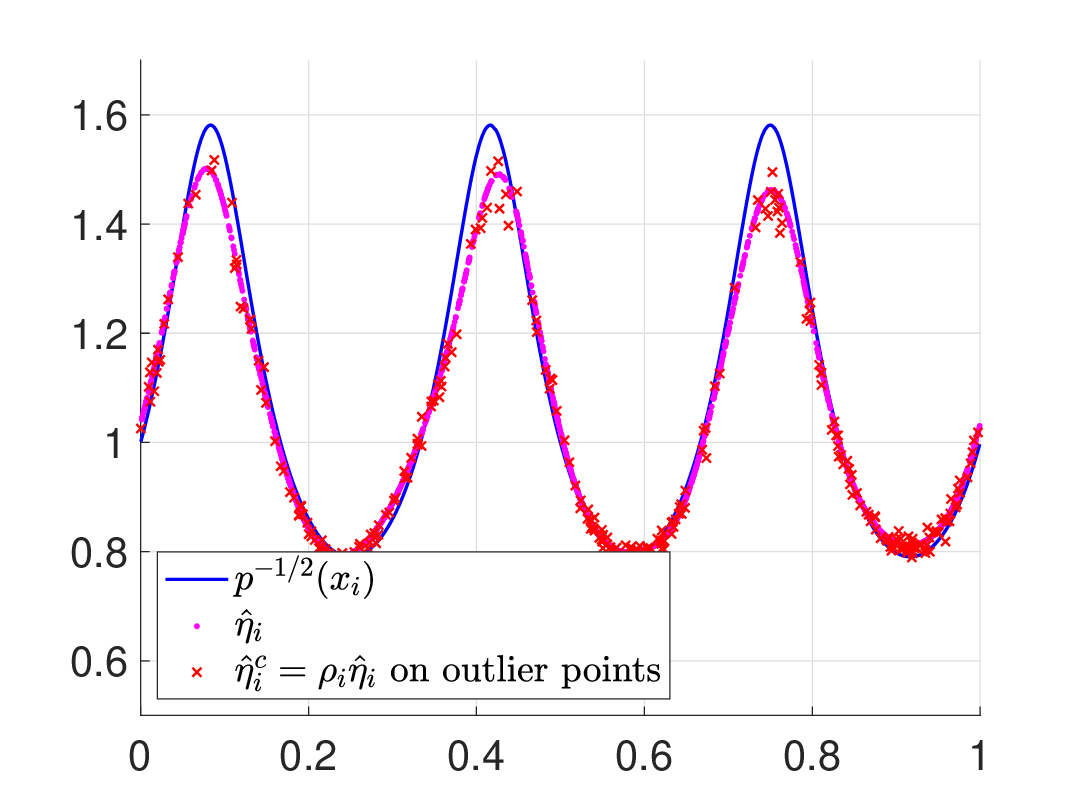}\\
\includegraphics[height=.24\linewidth]{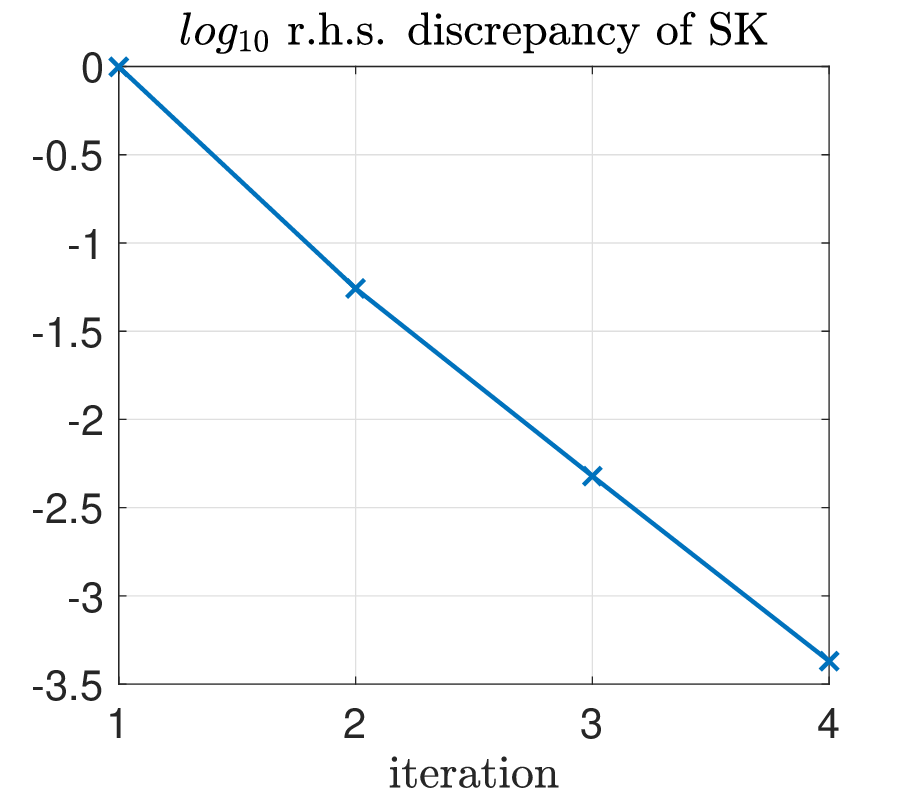}
\includegraphics[height=.24\linewidth]{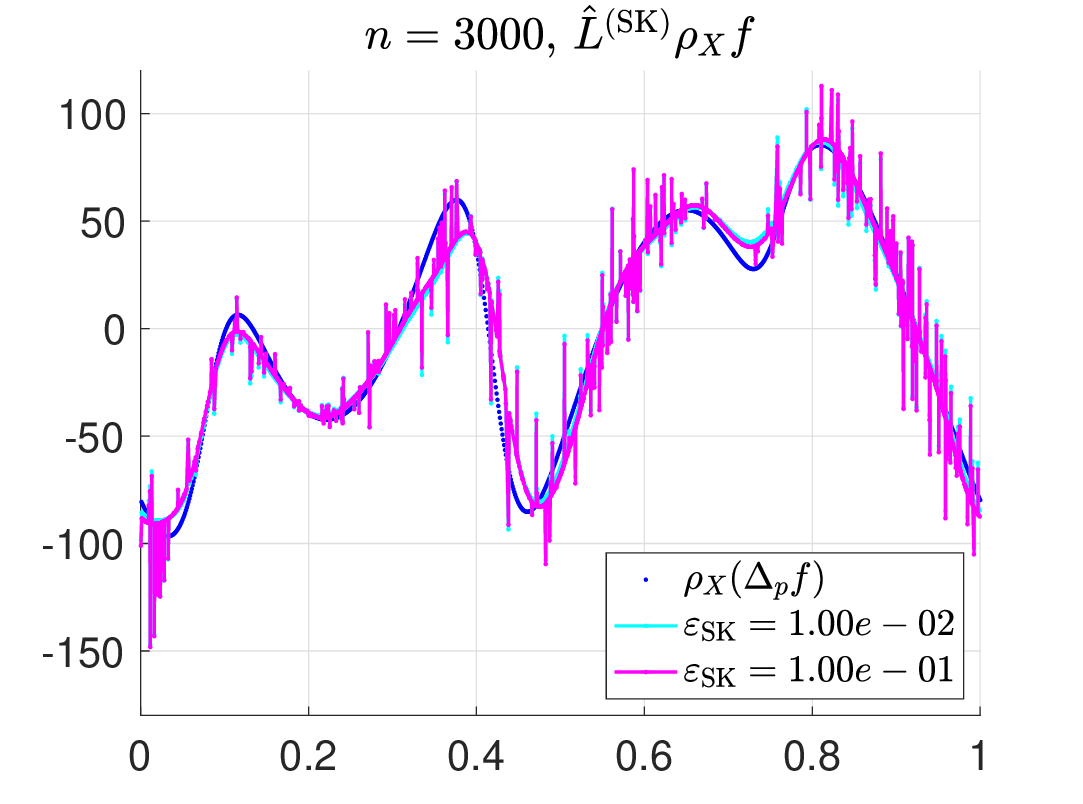}
\includegraphics[height=.24\linewidth]{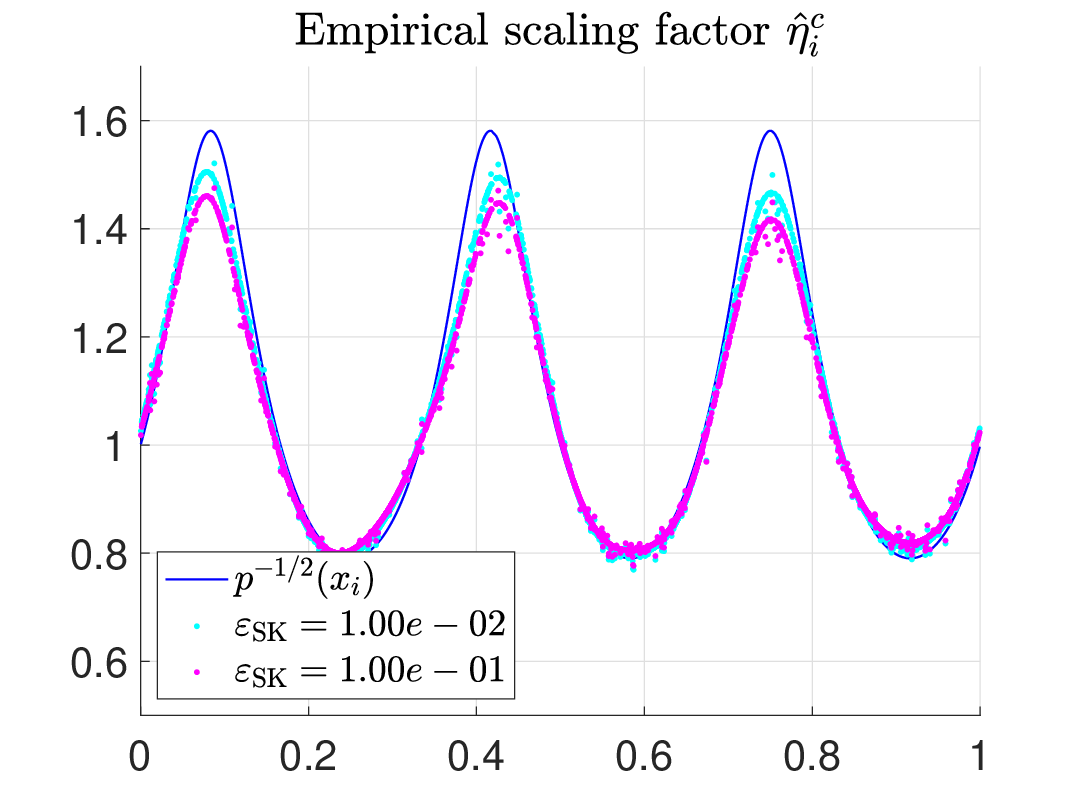}
\vspace{-5pt}
\caption{
\small
Results of $\hat{L} \rho_X f$ computed from manifold data corrupted by i.i.d outlier noise in $\R^{m}$, $m=2000$.
The kernel bandwidth parameter $\epsilon=$ \texttt{5e-4}.
Top panel: 
(Left) The first 3 coordinates of the data vectors, colored by the kernel affinity values $W_{i_0 j}$  on $x_j$ (point $i_0$ is an in-lier). 
(Middle) Computed values of $\hat{L} \rho_X f$  compared with the true values of $\Delta_p f$. 
(Right) Values of $\hat{\eta}_i$ on in-liers (those on out-liers are large and outside the plot axis) and those of $\hat{\eta}^c_i$ on the out-liers,
compared with  $p^{-1/2}(x_i)$.
Bottom panel:
Same plots as in the bottom panel of Figure \ref{fig:Ln-1d-one-run} on data with outlier noise, except that the (Middle) plot shows the values of $\hat{\eta}_i^c$.
}
\label{fig:Ln-1d-noise}
\end{figure}

\subsection{Manifold data with outlier noise}

We embed the data from the previous subsection into $\R^m$ (by letting the manifold data in $\R^4$ to be in the first four coordinates), and then add outlier noise to the data vectors.
The noise model follows a similar setup as in Remark \ref{rk:varepsilon-z} and satisfies  Assumption \ref{assump:A3} when the ambient dimensionality $m$ is large.
 We consider both heteroskedastic and i.i.d. (non-heteroskedastic) noise to be detailed below.

\subsubsection{Operator point-wise convergence }\label{subsec:point-wise-with-noise}

\paragraph{Data simulation.}
We set the ambient dimension to be $m= 2000$ and the
number of samples $n = 3000 $.
The clean manifold data $x_i^c$ is obtained by embedding the 1D manifold data in $\R^m$.
The outlier noise model follows the setup in Assumption \ref{assump:A3},
$b_i \sim {\rm Bern}( p_i)$,  and $z_i \sim \calN( 0, (\sigma_i^2/m) I_m )$, $\sigma_i = \sigma_{\rm out} \sqrt{ \gamma_i}  $,
and the parameters are to be specified below.
\begin{itemize}
\item[(i)]  i.i.d. noise.
  $p_i \equiv 0.1$, so that $b_i \sim {\rm Bern}(0.1)$ i.i.d.
$\gamma_i \equiv 1$ and $\sigma_{\rm out}^2 = 10^{-2}$,
and then  $z_i \sim \calN \left(0,  ({\sigma_{\rm out}^2}/{m}) I_m \right)$ i.i.d.
for $i=1,\cdots, n$. 
\item[(ii)]  Heteroskedastic noise.
For the distribution of $b_i\sim {\rm Bern}( p_i)$, we set $p_i = 0.05+ 0.9 (  (1-t(x_i^c)+ u_i) \mod 1 )$, where $u_i \sim {\rm Unif}(0,1)$ independently.
This makes the value of $p_i$ lie between 0.05 and 0.95 on $S^1$. 
For the distribution of $z_i$, we set 
$\sigma_{\rm out}^2 = {10^{-3}}$,
$\gamma_i = 0.9 \gamma_i^{(1)} + 0.1 \gamma_i^{(2)}$, where $\gamma_i^{(2)} \sim \rm {Unif}(0,3)$, and 
$\gamma_i^{(1)} = \gamma_{\rm o}(x_i^c)$, $\gamma_{\rm o}(x) = 10^{ 1-( (1+\sin(2\pi t(x)))/2 )^2 }$.
In both expressions of $p_i$ and $\gamma_i$,
$t(x) \in [0,1]$ is the intrinsic coordinate of a clean manifold data sample $x \in S^1$.
\end{itemize}
Note that the outlier proportion $p_i$ can be as large as 0.95 in the setup of heteroskedastic noise.

\paragraph{Method.}
The approximate bi-stochastic scaling is solved by setting the tolerance $\varepsilon_{\rm SK} =$ \texttt{1e-3} and $C_{\rm SK} = 0.1$ (but the projection step is never invoked). 
The other setup to compute the two Laplacians $\hat{L}^{\rm (SK)}$ and $\hat{L}^{\rm (DM)}$
and the relative errors are the same as in Section \ref{subsec:exp-clean}.

\begin{figure}[t]
\centering
\includegraphics[height=.26\linewidth]{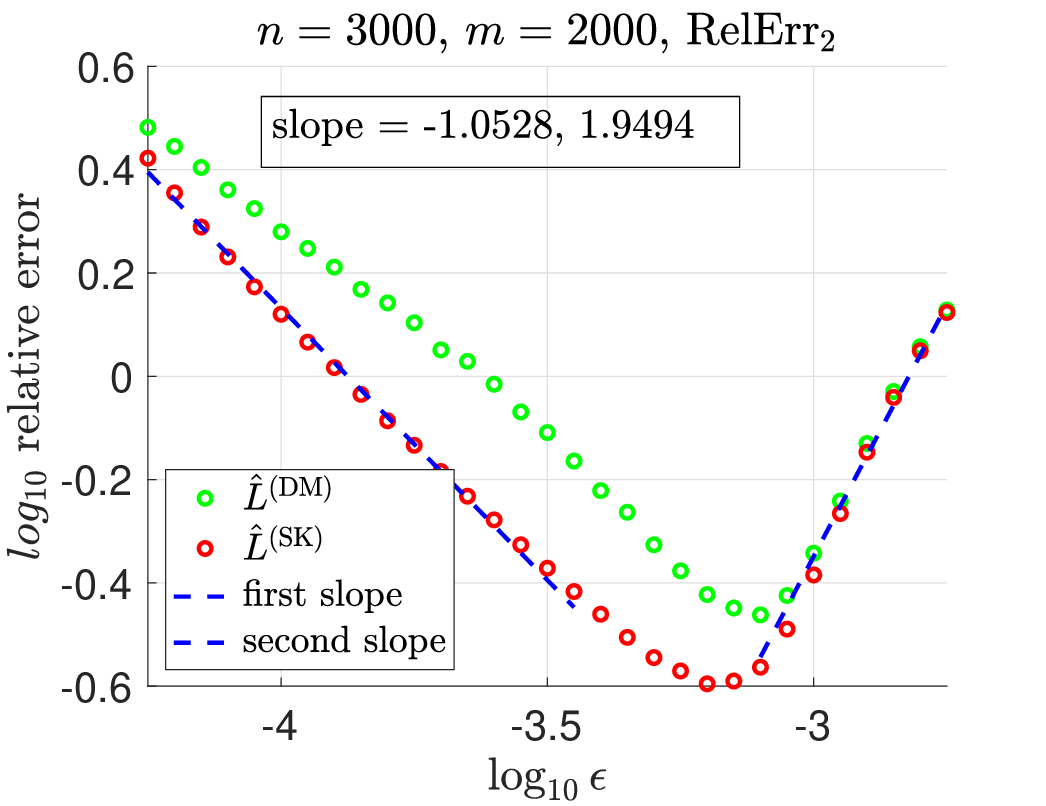}
\includegraphics[height=.26\linewidth]{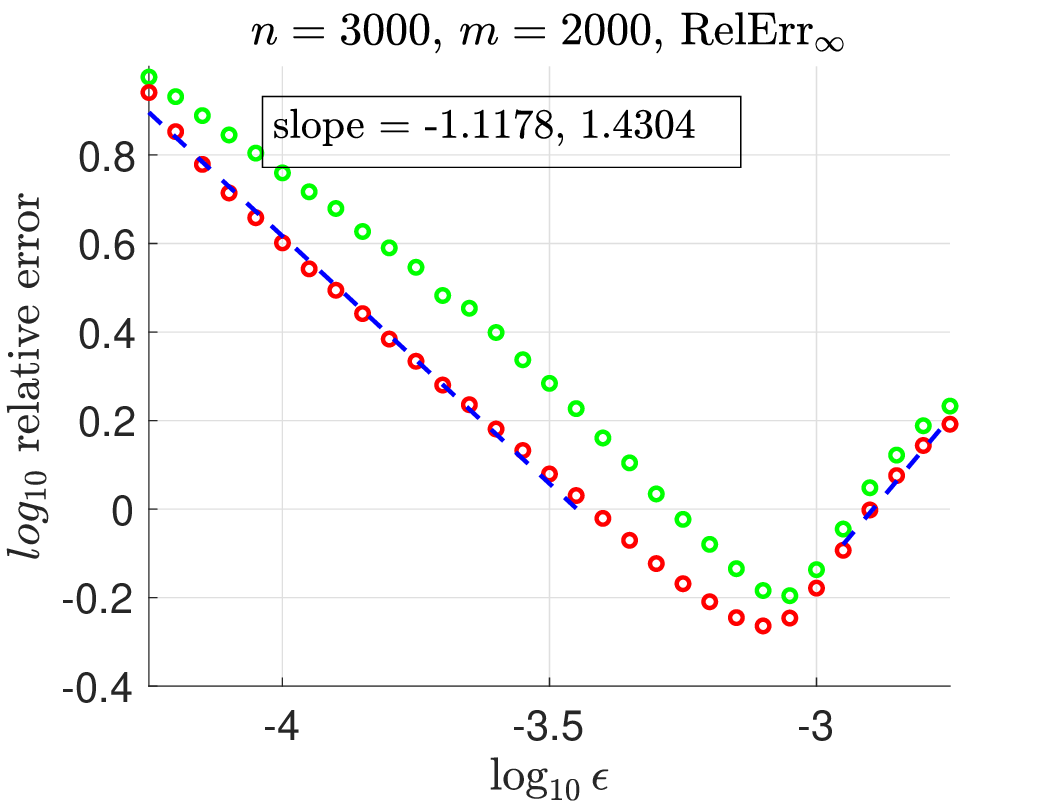}
\vspace{-5pt}
\caption{
\small
Same plots as in Figure \ref{fig:Ln-error-1d}
for average errors on manifold data with heteroskedastic outlier noise.
}
\label{fig:Ln-error-1d-noise}
\end{figure}

\begin{table}[t]
    \centering
    \small
    \begin{tabular}{ c |c c | c c}
    \hline
    			  				&   \multicolumn{2}{ c }{heteroskedastic noise}	& \multicolumn{2}{ |c }{i.i.d. noise}	  \\
    \hline
    			  				&   	1st pair 		&  2nd pair  		&  1st pair 		&  	2nd pair  \\
    $\hat{L}_{\rm rw}^{(\rm DM)}$	&  0.1268 (0.0363) 	&  0.2024 (0.0736) 	& 0.0089 (0.0036) 	&	0.0378 (0.0146)	\\ 
    $\hat{L}_{\rm rw}^{(\rm SK)}$		&  0.0042 (0.0018) 	&  0.0172 (0.0071) 	& 0.0030 (0.0015) 	&  	0.0112 (0.0045) 	\\ 
    \hline
    \end{tabular}
    \caption{
    \small
    MSE errors of the first (and second) pair of non-trivial eigenvectors of the traditional and bi-stochastically normalized graph Laplacians
    compared to limiting eigenfunctions $\{ \sin(2\pi t),  \cos(2\pi t) \}$ (and  $\{ \sin(4\pi t),  \cos(4\pi t) \}$) evaluated on $x_i^c$ 
    after eigen-space alignment. 
    The MSE values are averaged over 100 replicas and the standard deviation values are in the brackets. 
    Data and noise models and other experimental settings are the same as in Figure \ref{fig:embed-1d-noise-type1} for heteroskedastic noise
    and Figure \ref{fig:embed-1d-noise-type2} for i.i.d. noise. 
    }
    \label{tab:eigvect-mse}
\end{table}

\paragraph{Results.}
The result for one simulation with (i) i.i.d. noise is shown in Figure \ref{fig:Ln-1d-noise}.
As illustrated by the right plot in the top panel, the computed empirical factor $\hat{\eta}_i$ are lower bounded by an $O(1)$ constant greater than 0.6.
Again we observe that the lower-boundedness of $\min_i \hat{\eta}_i$ by a constant $C_{\rm SK} = 0.5$ is always fulfilled throughout the SK iterations.
The plot shows the values of $\hat{\eta}^c_i = \rho_i \hat{\eta}_i $ on the outlier points (red cross), which are close to the function values of $p^{-1/2}(x_i)$,
and since $\rho_i$ are much smaller than 1, the values of $\hat{\eta}_i $ on these outlier points are much larger and outside the plot box. 
While data vectors have additive outlier noise, the SK iterations still converge exponentially fast,
and the r.h.s. discrepancy achieves less than \texttt{1e-3} in typically less than 10 iterations. 
With higher tolerance of \texttt{1e-2} and \texttt{1e-1}, similar results of $\hat{L} \rho_X f$ can be obtained. 
Compared to with clean data (Figure \ref{fig:Ln-1d-one-run}), the approximation of $\Delta_p f$ has larger $\infty$-norm error and $2$-norm error,
while the increase of $2$-norm error is smaller and the relative error remains less than 20\%.
These numerical results are consistent with the theoretical analysis in Section \ref{sec:theory-noise-outlier}, see Theorem \ref{thm:hatL-convergence-noise}.

The averaged errors (computed over 500 times) are from data with (ii) heteroskedastic noise.
The results are shown in Figure \ref{fig:Ln-error-1d-noise} in the same plotting format as in Figure \ref{fig:Ln-error-1d}.
After adding noise, the errors in Figure \ref{fig:Ln-error-1d-noise} are larger than those in Figure \ref{fig:Ln-error-1d}. 
The increase in the 2-norm error is less severe than the $\infty$-norm error.  
Compared with $\hat{L}^{\rm (DM)}$, the bi-stochastic graph Laplacian $\hat L^{(\rm SK)}$ has a smaller error in Figure \ref{fig:Ln-error-1d-noise}
and shows better noise robustness.
Overall, the performances of $\hat{L}^{\rm (DM)}$ and $\hat L^{(\rm SK)}$ 
on clean (Figures \ref{fig:Ln-1d-one-run} and \ref{fig:Ln-error-1d})
and i.i.d noise data (Figure \ref{fig:Ln-1d-noise}) are close,
and the advantage of $\hat L^{(\rm SK)}$ is more drastic when with heteroskedastic noise (Figure \ref{fig:Ln-error-1d-noise}). 
We will continue the comparison on spectral embedding in the next sub-section.

\subsubsection{Spectral embedding }\label{subsec:spec-embed}


While we postpone the full theoretical analysis of eigen-convergence to future work, we include numerical results on the eigenvectors 
(the spectral embedding of data samples)
and compare the bi-stochastic normalization v.s. traditional normalization
when data contain high dimensional outlier noise.

\begin{figure}[t]
\centering
\includegraphics[height=.24\linewidth]{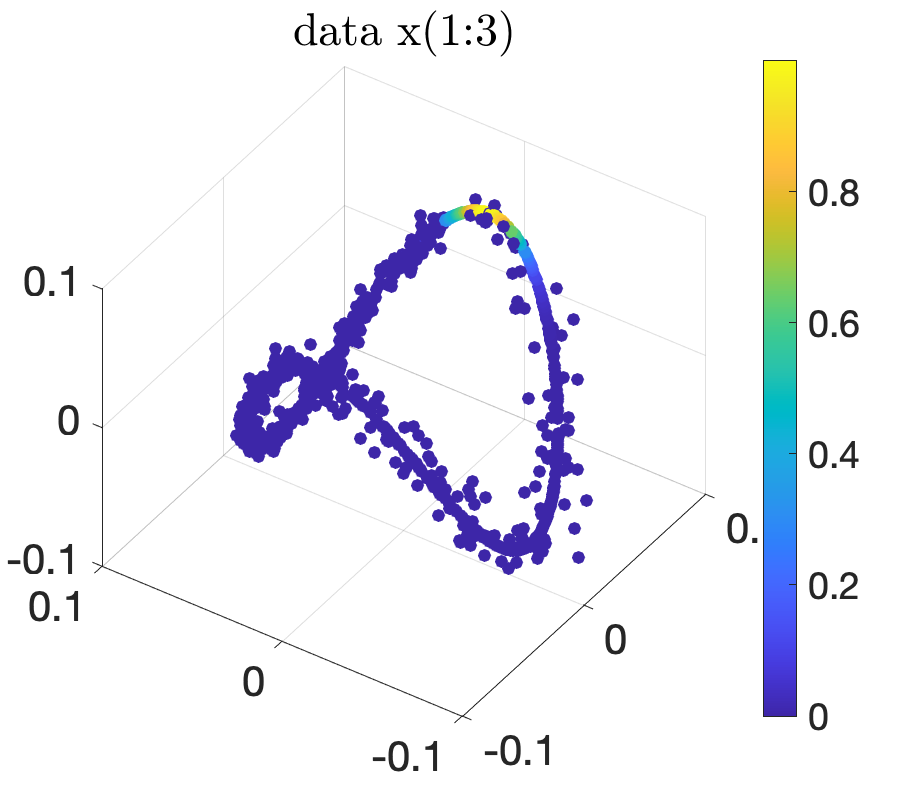}
\includegraphics[height=.24\linewidth]{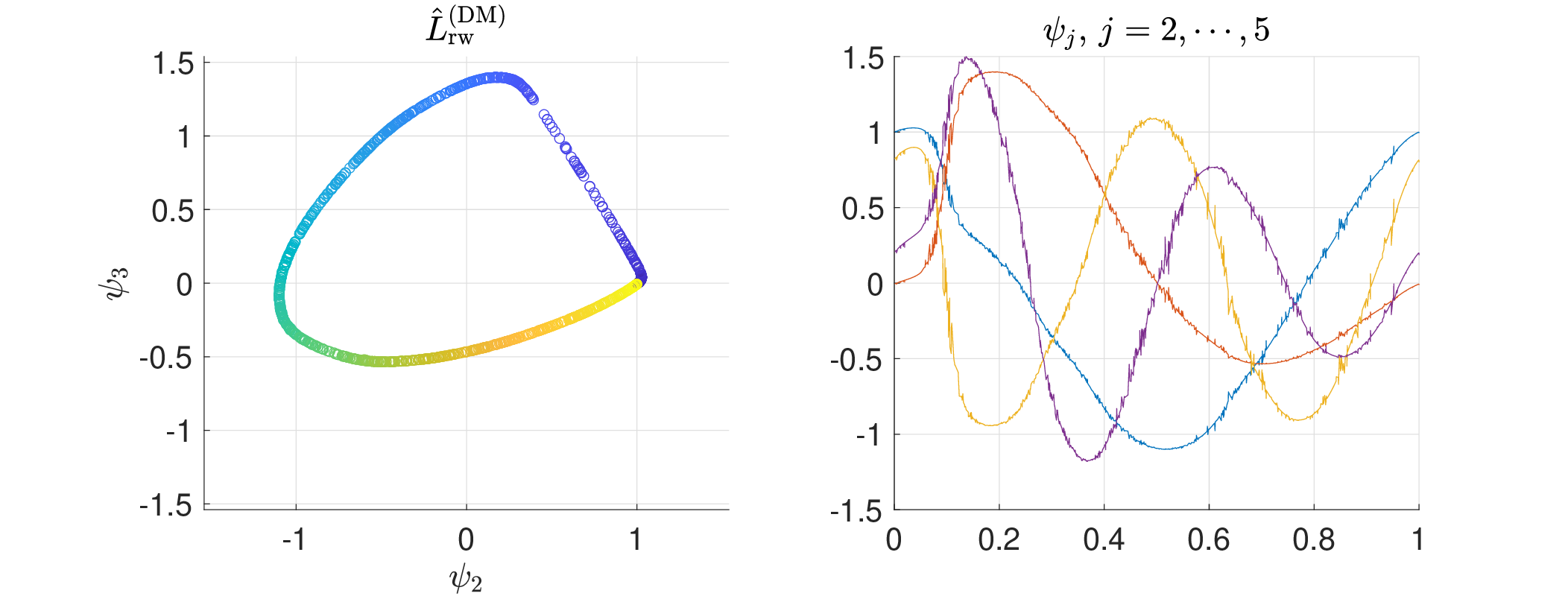}\\
\includegraphics[height=.24\linewidth]{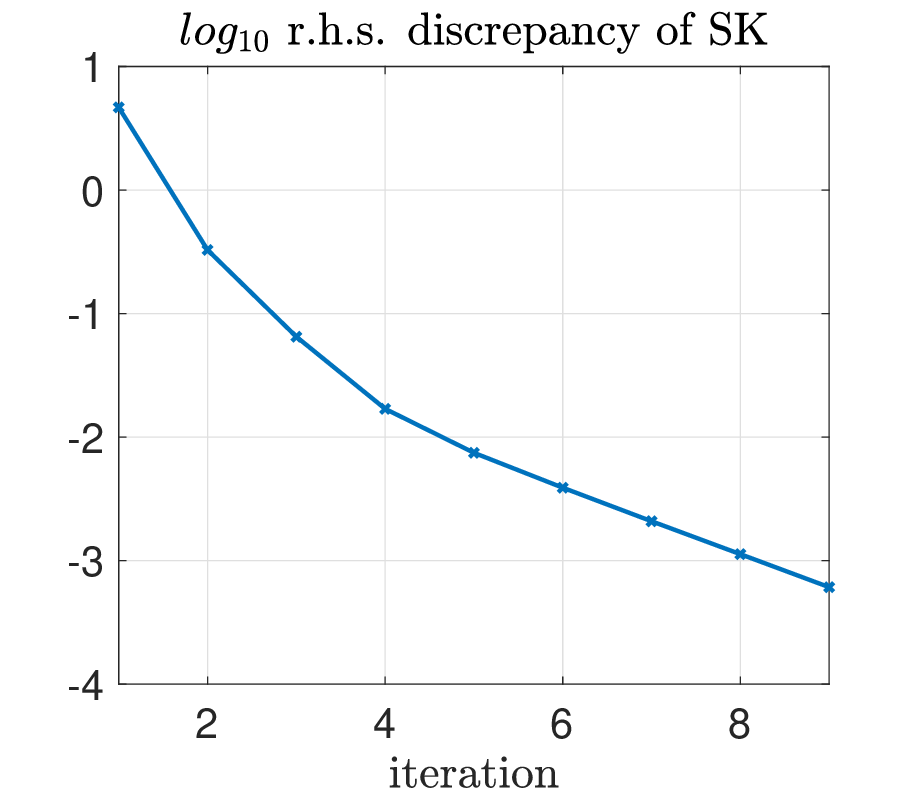}
\includegraphics[height=.24\linewidth]{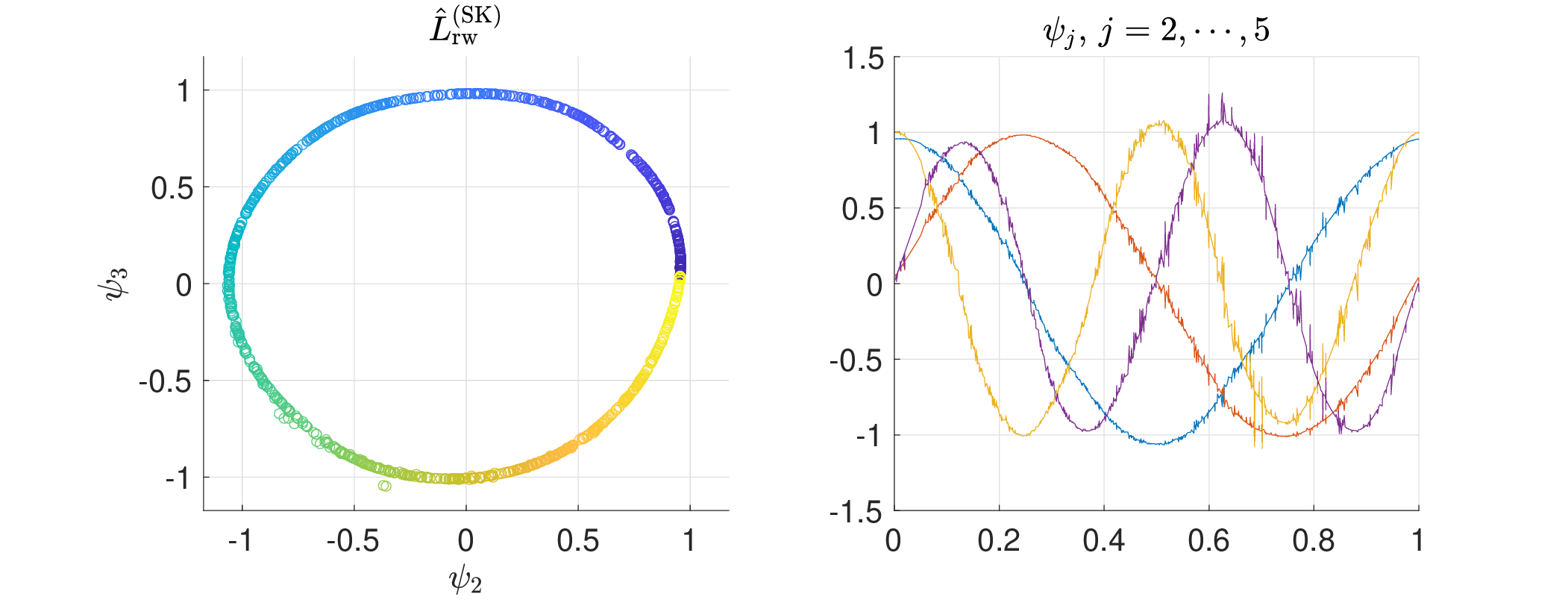}
\vspace{-5pt}
\caption{
\small
First few (non-constant) eigenvectors of $\hat{L}_{\rm rw}$ computed from $n=1000$ manifold data with heteroskedastic outlier noise added in $\R^m$, $m=2000$.
The Gaussian kernel affinity \eqref{eq:def-G-affinity} is computed with  $\epsilon=$ \texttt{5e-4}. 
Top panel: 
(Left) First 3 coordinates of data vectors colored by Gaussian kernel affinity values $G_{i_0 j}$  on $x_j$ (point $i_0$ is an in-lier). 
(Middle) Spectral embedding by the first two non-constant eigenvectors of $\hat{L}_{\rm rw}^{(\rm DM)}$, colored by the intrinsic coordinate of $x_i$.
(Right) First four  non-constant eigenvectors plotted against the intrinsic coordinate of $x_i$.
Bottom panel:
(Left) The convergence of Sinkhorn iterations;
(Middle-right)
Same plots as in the top panel for eigenvectors of $\hat{L}_{\rm rw}^{(\rm SK)}$.
In this simulation, 
the MSE errors (defined in Section \ref{subsec:spec-embed}) for DM is 
0.1234 and 0.1659
for the first and second pair of harmonics,
and those for SK is 
0.0030 and 0.0082.
}
\label{fig:embed-1d-noise-type1}
\end{figure}

\paragraph{Data simulation.}
We draw $n = 1000$ samples embedded in $\R^{m}$, $m=2000$,
where the clean manifold data (before adding the outlier noise) is uniformly distributed on $S^1$.
The noise model follows that in Section \ref{subsec:point-wise-with-noise} with some changed parameters:
\begin{itemize}
\item[(i)]  i.i.d. noise.
For the distribution of $b_i$, we set $p_i \equiv 0.95$.
For the distribution of $z_i$, we set $\sigma_{\rm out}^2 = 10^{-2}$, $\gamma_i \sim \rm {Unif}(0,3)$ i.i.d.
\item[(ii)]  Heteroskedastic noise.
Same as in Section \ref{subsec:point-wise-with-noise} except that $\sigma_{\rm out}^2 = 10^{-2}$. 
\end{itemize}

\paragraph{Method.}
We compute the eigenvectors associated with the smallest few eigenvalues of $L_{\rm rw} = I- D^{-1} W$
of the bi-stochastically normalized and ($\alpha=1/2$) Diffusion Map graph Laplacians respectively,
that is, setting $W$ to be $\hat{W} = D_{\hat{\eta}} W^0 D_{\hat{\eta}}$ as in Section \ref{sec:theory-noise-outlier}
and $\tilde{W}$ as in \eqref{eq:def-hatL-DM}.
Note that in the definitions of both  $\hat{L}^{(\rm DM)}$ and $\hat{L}^{(\rm SK)}$, any global constant multiplying to $W_{ij}$ will cancel out.
Thus, in practice, there is no need to know the constant $\epsilon^{-d/2}$ in \eqref{eq:def-W} nor $(4\pi)^{d/2}$ in \eqref{eq:def-g-gaussian},
and it suffices to compute the Gaussian affinity matrix $G$,
\begin{equation}\label{eq:def-G-affinity}
G_{ij}  = \exp \left\{ - \frac{ \|x_i - x_j\|^2 }{4 \epsilon}  \right\}.
\end{equation}
In computing bi-stochastically normalized graph Laplacian, the SK iterations is solved with $\varepsilon_{\rm SK} =$ \texttt{1e-3}.
Because we work with the matrix $G$ which differs from the matrix $K$ by a scalar multiplication, the constant factor needs to be included in the lower-bound constant $C_{\rm SK}$ as well.
In this experiment, we turn off the projection step in SK iterations, and empirically we observe that the iterations terminate after achieving the discrepancy, and the empirical scaling factor (after multiplying back the theoretical constant) observes the theoretical $O(1)$ lower bound.

\paragraph{Results.}
For data with heteroskedastic outlier noise, 
the results are shown in Figure \ref{fig:embed-1d-noise-type1}.
The first eigenvalue is zero and the corresponding eigenvector is the constant vector, thus only the non-constant eigenvectors are shown. 
Because the data density $p$ is uniform on $S^1$, the operator $\Delta_p$ is the Laplace-Beltrami operator, and the limiting eigenfunctions are sines and cosines. 
To compare the computed eigenvectors to the limiting eigenfunctions, 
we pair the first two (non-trivial) eigenvectors into a 2-dimensional eigen-space
and compute the mean-squared error (MSE) with the $n$-by-2 array formed by
$\{ \sin(2\pi t),  \cos(2\pi t) \}$ evaluated on $x_i^c$ 
after alignment (up to a 2-by-2 rotation of the 2-dimensional eigen-space and a scaler multiplication).
Similarly, we pair the next two eigenvectors and compute the MSE with $\{ \sin(4\pi t),  \cos(4\pi t) \}$ after alignment. 
The eigenvectors shown in the plots are after the alignment.

The mean MSE over 100 replicas are shown in Table \ref{tab:eigvect-mse},
including the cases with heteroskedastic noise and with i.i.d. noise.
The one-run result with i.i.d. noise is shown in Figure \ref{fig:embed-1d-noise-type2}.
Quantitatively, as shown in Table \ref{tab:eigvect-mse}, 
the performance gap between $\hat{L}_{\rm rw}^{(\rm SK)}$ and $\hat{L}_{\rm rw}^{(\rm DM)}$ is more significant under the  heteroskedastic noise.
With i.i.d. noise, the bi-stochastic normalization still shows an advantage that is statistically significant (the average MSEs in the 1st pair differ more than 1 standard deviation away, and in the 2nd pair more than 1.5 standard deviation away). 
The plots and the quantitative error indicate that the spectral embedding by $\hat{L}_{\rm rw}^{(\rm SK)}$ is more robust to outlier noise than with the traditional normalization.

\section{Discussion}\label{sec:discuss}

The analysis in the paper can be extended in several directions.
First, an $\infty$-norm consistency result of the empirical scaling factor of the bi-stochastic normalization would be interesting.
Such a result can lead to $\infty$-norm consistency of the point-wise convergence of the graph Laplacian,
which is suggested by the numerical experiments at least for clean manifold data.
The $\infty$-norm consistency of scaling factors will also help to further derive the eigen-convergence analysis based on Section \ref{subsec:discuss-eigenvector}.
For the eigen-convergence with rates, 
in addition to the variational approach as detailed in this work,
one can also try to combine our analysis with the spectral stability approach in \cite{wormell2021spectral}.

Next, the analysis of noise robustness can be extended beyond the technical assumption (A3), e.g., by considering more specific class of high-dimensional noise and relaxing the condition on Bernoulli random variable $b_i$ in the outlier noise model. 
Meanwhile, it would also be desirable to remove the lower-boundedness constraint  in the approximate matrix scaling problem.
The parameter $C_{\rm SK}$ is introduced for theoretical purposes in the current work,
and it appears to be  not needed in practice at least when there are enough data samples
and the convergence of SK iterations encounters no difficulty.
At last, it would be interesting to study the tightness of the noise-robustness analysis,
for example, by considering lower bounds of the error due to noise, 
for the bi-stochastic normalization as well as the traditional ones.

On the computation side, one may develop more efficient algorithms to solve the (approximate) matrix scaling problem, both in computational time and memory usage. A scalable algorithm to large sample size (graph size) would make bi-stochastically normalized graph Laplacian a practical tool for various applications. 
Finally, it would be helpful to examine the performance and explore the potential advantage of  bi-stochastically normalized graph Laplacian on real data and real noise scenario 
beyond the simulation experiments in this work.

\section{Proofs}\label{sec:proofs}

\subsection{Proofs in Section \ref{sec:approx-matrix-scaling-problem}}

\subsubsection{Proof of Lemma \ref{lemma:construct-bar-eta}}

We introduce a lemma  which characterizes the small $\epsilon$ expansion of  kernel integral on manifold applied to a test function $f$.
The argument was firstly derived as Lemma 8 in  \cite{coifman2006diffusion}, and used in several follow-up analysis of Diffusion Maps. 
Here we reproduce the version of Lemma A.5 in \cite{cheng2021convergence} and apply to when the kernel function is Gaussian.

\begin{lemma}[Lemma A.5 in \cite{cheng2021convergence}, {Lemma 8 in  \cite{coifman2006diffusion}}]
\label{lemma:h-integral-diffusionmap}
Under Assumption (A1), and let $g$ be the Gaussian kernel function as in \eqref{eq:def-g-gaussian}.
There is $\epsilon_0 >0$ a constant depending on $\calM$, 
such that when $ 0 < \epsilon < \epsilon_0$, 
for  any $f \in C^4({\calM})$,
\[
 \epsilon^{- {d}/{2}} \int_{\calM} g \left( \frac{ \| x-y \|^2}{\epsilon} \right) f(y) dV(y)
=   f(x) +  \epsilon ( \omega f + \Delta_{\calM}  f)(x) 
+ O (\epsilon^2) ,
\]
where $\omega(x)$ is determined by local derivatives of the extrinsic manifold coordinates at $x$.
In the $O(\epsilon^2)$ term, the constant involves up to the 4-th derivative of $f$ on ${\calM}$.
Specifically, if the residual term is denoted as $r_{f, \epsilon}(x)$, it satisfies 
$ \sup_{x \in {\calM}} |r_{f, \epsilon}(x)| \le C(f) \epsilon^2$, where
$C(f) = c({\calM}) (1+ \sum_{l=0}^4 \| D^{(l)} f \|_\infty)$.
\end{lemma}

The proof of Lemma \ref{lemma:h-integral-diffusionmap} is left to Appendix \ref{app:more-prooofs}. We now prove Lemma \ref{lemma:construct-bar-eta}.

 \begin{proof}[Proof of Lemma \ref{lemma:construct-bar-eta}]
 
 To prove (i), we construct the $C^4$ function $r$ in the following way.
 Let $\omega(x)$ be defined as in Lemma \ref{lemma:h-integral-diffusionmap},
 which is a $C^\infty$ function on $\calM$.  
 We define
 \begin{equation}\label{eq:def-r-q-epsilon}
 r := -\frac{1}{2} \left( \omega + \frac{\Delta (p^{1/2})}{p^{1/2}} \right), 
 \end{equation}
 and by that $\omega \in C^{\infty}(\calM)$ and $p \in C^6(\calM)$, we have that $r \in C^4(\calM)$. 
 By compactness and smoothness of $\calM$, $\| r \|_\infty$ is finite and is determined by $\omega$ and $p$. 
To prove the uniform boundedness of $q_\epsilon$,
observe that $\min_{x \in \calM} p^{-1/2}(x) = p_{\rm max}^{-1/2} > 0$.
This allows us to set $q_{\rm min} := 0.5 p_{\rm max}^{-1/2} > 0$ as follows:
As long as $1+\epsilon r(x) \ge 0.5$ uniformly for all $x$, we have $q_\epsilon (x) \ge 0.5 p^{-1/2}(x) \ge q_{\rm min}$.
The former can be fulfilled as long as $\epsilon \|r\|_\infty \le 0.5$, which can be guaranteed by letting  $\epsilon_0 = 0.5/ (1+\|r\|_\infty)$.
(Note that in the argument by making $\epsilon_0$ small, one can make $q_{\rm min} = c p_{\rm max}^{-1/2}$ for any $ 0<c< 1$, and here we take $c=0.5$ as an example.)
For the uniform upper bound of $q_\epsilon$, note that for any $x \in \calM$, 
$p^{-1/2}(x) \le p_{\rm min}^{-1/2}$ and $|1 + \epsilon r(x) | \le 1 +\epsilon \| r\|_\infty$, 
and thus $q_{\rm max} :=  p_{\rm min}^{-1/2}(1+ \epsilon_0 \| r\|_\infty) \le 1.5 p_{\rm min}^{-1/2} $  is a valid choice. 
This proves (i). 

 To prove (ii), we need to show that under the stated condition, setting $\bar{\eta}_i = q_\epsilon(x_i)$ satisfies that 
 \begin{equation}\label{eq:sum-W-bareta-1}
\bar{\eta}_i  \sum_{j=1}^n W_{ij}  \bar{\eta}_j  = 1 + e_i, 
\quad |e_i| \le \varepsilon_{\rm SK} = O \left( \epsilon^2, \sqrt{\frac{\log n }{n \epsilon^{d/2}}} \right),  
\quad \forall i = 1, \cdots, n.
 \end{equation}
For each $i$, the summation in \eqref{eq:sum-W-bareta-1} excluding the $j=i$ term is an independent sum across $j$ (conditioning on sample $x_i$). 
We will prove the $\varepsilon_{\rm SK}$-bound of $|e_i|$ for each $i$ by concentration argument and apply the union bound over $i=1,\cdots, n$.

For each $i$,  we have
\begin{align}
\bar{\eta}_i  \sum_{j=1}^n W_{ij}  \bar{\eta}_j
&= \bar{\eta}_i^2 W_{ii} 
	+ \bar{\eta}_i \sum_{j, j\neq i} W_{ij}  \bar{\eta}_j  \nonumber \\
&=  \frac{ (4\pi \epsilon)^{-d/2}}{n} q_\epsilon(x_i)^2 
	+  (1-\frac{1}{n}) q_\epsilon(x_i) \frac{1}{n-1}   \sum_{j, j \neq i}  
	\epsilon^{-d/2} g \left(  \frac{ \| x_i - x_j\|^2}{\epsilon} \right)q_\epsilon(x_j) \nonumber \\
&:= \textcircled{1} + 	 (1-\frac{1}{n})  \textcircled{2}. \label{eq:bar-eta-i-1}
\end{align}
For $\textcircled{1}$, by the uniform upper bound of $q_\epsilon$ in (i), 
$|  \textcircled{1} | \le q_{\rm max}^2 ( (4\pi \epsilon)^{d/2}  n)^{-1}$.
Since we have $\epsilon^{d/2} = \Omega(\log n/n)$, $(\epsilon^{d/2}n)^{-1} = o(1)$, and then
$(\epsilon^{d/2}  n)^{-1} = o( \sqrt{ \log n / \epsilon^{d/2}  n} )$. 
This means what we always have
\begin{equation}\label{eq:bound-1-bar-eta-1}
|  \textcircled{1} | = O( (\epsilon^{d/2}  n)^{-1} ) = o\left( \sqrt{\frac{\log n }{n \epsilon^{d/2}}} \right).
\end{equation}
For $\textcircled{2}$, 
\[
\textcircled{2} = q_\epsilon(x_i)   \frac{1}{n-1}   \sum_{j, j \neq i}  Y_j,
\quad Y_j:= \epsilon^{-d/2} g \left(  \frac{ \| x_i - x_j\|^2}{\epsilon} \right)q_\epsilon(x_j),
\]
and $\{Y_j, j\neq i\}$ are $(n-1)$ independent random variables conditioning on $x_i$. 
By definition, 
\begin{align*}
\E Y_j 
& =  \epsilon^{- {d}/{2}} \int_{\calM} g \left( \frac{ \| x-y \|^2}{\epsilon} \right) q_\epsilon(y) p(y) dV(y) \\
& = q_\epsilon p(x_i) + \epsilon ( w q_\epsilon p+ \Delta (q_\epsilon p ))(x_i) + O(\epsilon^2),
\end{align*}
where we apply Lemma \ref{lemma:h-integral-diffusionmap} with $f = q_\epsilon p $ which is $C^4$ on $\calM$,
and the constant in the $O(\epsilon^2)$ term is uniform for all $x$.
Recall that $q_\epsilon = p^{-1/2}(1+\epsilon r)$, 
we have
\begin{align*}
& q_\epsilon p + \epsilon( \omega  q_\epsilon p + \Delta (q_\epsilon p))
= p^{1/2} (1 + \epsilon r) + \epsilon( \omega p^{1/2} (1 + \epsilon r) +  \Delta( p^{1/2} (1 + \epsilon r)) )\\
&~~~
 = p^{1/2}  + \epsilon \left( p^{1/2} r +  \omega p^{1/2} +   \Delta( p^{1/2}) \right) + O(\epsilon^2),
\end{align*}
where the uniform constant in $O(\epsilon^2)$ is by the uniform boundedness of all involved functions on $\calM$. 
This gives that, omitting the evaluation at $x_i$ in the r.h.s. of the following equation,
\begin{align*}
q_\epsilon(x_i) \E Y_j 
& = p^{-1/2} (1+\epsilon r)  \left( p^{1/2}  + \epsilon \left( p^{1/2} r +  \omega p^{1/2} +   \Delta( p^{1/2}) \right) + O(\epsilon^2)  \right) \\
& = 
 1 + \epsilon \left( 2 r +  \omega  +  p^{-1/2} \Delta( p^{1/2}) \right) + O(\epsilon^2). 
\end{align*}
 By the definition of $r$ as  in \eqref{eq:def-r-q-epsilon},
$
2 r +  \omega  +  p^{-1/2} \Delta( p^{1/2}) = 0$,
and this proves that 
\begin{equation}\label{eq:bound-EY-circle2-1}
q_\epsilon(x_i) \E Y_j  
= \epsilon^{- {d}/{2}} \int_{\calM}  q_\epsilon(x_i) g \left( \frac{ \| x_i-y \|^2}{\epsilon} \right) q_\epsilon(y) p(y) dV(y) 
= 1 + O(\epsilon^2),
\end{equation}
where the constant in the $O(\epsilon^2)$ term is uniform for all $x_i$.

To bound the deviation of $\frac{1}{n-1}   \sum_{j, j \neq i}  Y_j$ from $\E Y_j$, 
note that $|Y_j| \le L = \Theta( \epsilon^{-d/2})$,
and variance of $Y_j$ is bounded by 
\[
\E Y_j^2 = \epsilon^{- {d}} \int_{\calM} g^2 \left( \frac{ \| x-y \|^2}{\epsilon} \right) q_\epsilon^2(y) p(y) dV(y)
\le \nu = \Theta(\epsilon^{-d/2}).
\]
Then classical Bernstein inequality gives that (by conditioning on $x_i$ first and then over the randomness of all $n$ samples)
when $n$ is sufficiently large, w.p.$> 1- 2n^{-10}$, 
\begin{equation}\label{eq:bound-conc-circle2-1}
\left| \frac{1}{n-1}   \sum_{j, j \neq i}  Y_j - \E Y_j \right|
\le \Theta \left(  \sqrt{ \frac{\log n}{ n } \nu} \right) 
= O \left(  \sqrt{ \frac{\log n}{ n  \epsilon^{d/2}} } \right),
\end{equation}
and we call this the good event $E_i$,  for each $i$. 
(The sub-Gaussian tail in Bernstein inequality requires $  \sqrt{\nu  {\log n}/{ n } } < C \nu / L$ for some absolute constant $C$,  
and this holds for large $n$ due to that $\epsilon^{d/2} = \Omega(\log n /n)$.)
Under $E_i$, \eqref{eq:bound-conc-circle2-1} together with \eqref{eq:bound-EY-circle2-1} implies that 
\[
\textcircled{2} = q_\epsilon(x_i)  \left( \E Y_j + O \left(  \sqrt{ \frac{\log n}{ n  \epsilon^{d/2}} } \right) \right)
= 1 + O \left(  \epsilon^2,  \sqrt{ \frac{\log n}{ n  \epsilon^{d/2}} } \right),
\]
and so is $(1-\frac{1}{n})\textcircled{2}$ because $O(n^{-1})$ is an higher order term. 

Back to \eqref{eq:bar-eta-i-1}, we have that under good event $E_i$ which happens w.p.$>1- 2 n^{-10}$,  
 \[
 \bar{\eta}_i  \sum_{j=1}^n W_{ij}  \bar{\eta}_j 
 =  o\left( \sqrt{\frac{\log n }{n \epsilon^{d/2}}} \right) + 1 + O \left(  \epsilon^2,  \sqrt{ \frac{\log n}{ n  \epsilon^{d/2}} } \right),
 \]
 by combining \eqref{eq:bound-1-bar-eta-1} and the bound of $(1-\frac{1}{n})\textcircled{2}$. 
 This proves \eqref{eq:sum-W-bareta-1} w.p.$> 1-2n^{-9}$ by a union bound. 
 \end{proof}

\subsubsection{Proof of Lemma \ref{lemma:eta1-eta2} with extension and Lemma \ref{lemma:C1-implies-C2}}

\begin{proof}[Proof of Lemma \ref{lemma:eta1-eta2}]
Let $A_i = D_{\eta_i} A D_{\eta_i}$, and denote $e_i=e(\eta_i)$, $i=1,2$. 
By Definition \ref{def:eps-approx-scaling}, $\| e_i\|_\infty \le \varepsilon_{\rm SK}$ for both $i=1,2$.
Recall that $\eta_2 = \eta_1 \odot ({\bf 1} + u)$, we have
\begin{align*}
A_2 {\bf 1} 
&= D_{{\bf 1} + u } A_1 D_{{\bf 1} + u } {\bf 1}  = (I + D_u) A_1 ({\bf 1} + u) \\
&
= A_1 ({\bf 1} + u)  + D_u A_1 ({\bf 1} + u)  \\
&
= {\bf 1} + e_1 + A_1 u  + D_u A_1 ({\bf 1} + u),
\end{align*}
and since  $A_2 {\bf 1}  = 1+ e_2$, this gives
\begin{equation}\label{eq:proof-e1-e2-1}
e_2 - e_1 =  A_1 u  + D_u A_1 ({\bf 1} + u).
\end{equation}
Multiply $u^T$ to both sides of \eqref{eq:proof-e1-e2-1}, and use the fact that $A_1  \succeq 0$  
(because $A  \succeq 0$ and $A_1 = D_{\eta_1} A  D_{\eta_1} $), we have
\begin{align}
u^T(e_2 - e_1) 
&\ge u^TD_u A_1 ({\bf 1} + u)  \nonumber \\ 
& = \sum_{i} u_i^2  (\eta_1)_i \sum_j  A_{ij} (\eta_1)_j (1+u_j) \nonumber \\
& = \sum_{i} u_i^2  \frac{(\eta_1)_i}{(\eta_2)_i}  \left( \sum_j  (\eta_2)_iA_{ij} (\eta_2)_j \right). 
\label{eq:proof-e1-e2-2}
\end{align}
Since $ \sum_j  (\eta_2)_iA_{ij} (\eta_2)_j =  \sum_j  (A_2)_{ij} = 1+ (e_2)_i $, 
and $|(e_2)_i | \le \| e_2\|_\infty \le \varepsilon_{\rm SK} < 0.1$,
we have 
\[
\sum_j  (\eta_2)_iA_{ij} (\eta_2)_j \ge 1-\varepsilon_{\rm SK} > 0.9, \quad i=1,\cdots,n.
\] 
Meanwhile, by the boundedness condition \eqref{eq:cond-C1-C2-bdd},
$ {(\eta_1)_i}/{(\eta_2)_i} \ge C_1/C_2$ for all $i$. Putting back to \eqref{eq:proof-e1-e2-2}, this gives that
\begin{equation}\label{eq:bound-uDuA1(1+u)}
u^T(e_2 - e_1) 
\ge  \frac{C_1}{C_2} 0.9  \sum_{i} u_i^2, 
\end{equation}
and then
\[
0.9 \frac{C_1}{C_2}  \| u \|_2^2 \le |u^T(e_2 - e_1) | \le \|u\|_2 ( \| e_2\|_2 + \|e_1\|_2 ).
\]
By that $ 0.9 \frac{C_1}{C_2} $ is strictly positive, this proves the lemma.
\end{proof}
\begin{lemma}[Extension of Lemma \ref{lemma:eta1-eta2} to non-PSD $A$]
\label{lemma:eta1-eta2-nonPSD}
Suppose $A$ is symmetric, 
$\eta_1, \eta_2 \in \R_+^n$ are two $\varepsilon_{\rm SK}$-approximate bi-stochastic scaling factor of $A$, $\varepsilon_{\rm SK} < 0.1$, 
and  there exist $C_1, C_2 > 0$ such that
\begin{equation}\label{eq:cond-C1-C2-bdd}
\min_i (\eta_1)_i \ge  C_1,
\quad
\max_i (\eta_1)_i ,\, \max_i (\eta_2)_i \le C_2.
\end{equation}
There is $\delta > 0$ and $\delta < 0.4 C_1/C_2^3$ such that $A + \delta I  \succeq 0$.
Let $D_{{\eta_i}} A D_{{\eta_i}} {\bf 1} = {\bf 1} + e(\eta_i)$ for $i=1,2$, and
define $u \in \R^n$ by $\eta_2 = \eta_1 \odot ({\bf 1} + u)$,  then
\[
\| u \|_2 \le 
\frac{1}{0.5 } 
\frac{C_2}{C_1} (\| e(\eta_1)\|_2 + \|e(\eta_2)\|_2 ).
\]
\end{lemma}

\begin{proof}[Proof of the Lemma \ref{lemma:eta1-eta2-nonPSD}]
The proof follows a similar argument as the proof of Lemma \ref{lemma:eta1-eta2}, with modification of certain steps. 
Specifically,  similarly as in  \eqref{eq:proof-e1-e2-1}, we have
\begin{equation*}
e_2 - e_1 =  D_{\eta_1} (A +\delta I)D_{\eta_1} u  - \delta  D_{\eta_1} D_{\eta_1} u  + D_u A_1 ({\bf 1} + u).
\end{equation*}
Then, using that $A +\delta I$ is PSD, multiplying $u^T$ to both sides gives  that
\begin{equation}\label{eq:uT(e2-e1)-2}
u^T (e_2 - e_1 ) 
 \ge - \delta \| u \odot {\eta_1} \|_2^2  + u^T D_u A_1 ({\bf 1} + u).
\end{equation}
The second term can be bounded in the same way as in the derivation from \eqref{eq:proof-e1-e2-2} to \eqref{eq:bound-uDuA1(1+u)}, 
which gives that 
\[
u^TD_u A_1 ({\bf 1} + u) 
= \sum_{i} u_i^2  \frac{(\eta_1)_i}{(\eta_2)_i}  \left( \sum_j  (\eta_2)_iA_{ij} (\eta_2)_j \right)
\ge  \frac{C_1}{C_2} 0.9  \| u \|_2^2.
\]
To bound the first term, we use that $\max_i (\eta_1)_i \le C_2$, which gives  that 
\[
\| u \odot {\eta_1} \|_2^2 \le C_2^2 \| u \|_2^2.
\]
Putting together back to \eqref{eq:uT(e2-e1)-2}, we have that 
\begin{align}
u^T (e_2 - e_1 ) 
& \ge - \delta C_2^2 \| u\|_2^2 + 0.9 \frac{C_1}{C_2} \|u\|_2^2 
 \ge  0.5 \frac{C_1}{C_2} \|u\|_2^2,
\end{align} 
where in the last inequality we used that $\delta < 0.4 C_1/C_2^3$.
By that $ | u^T (e_2 - e_1 )|  \le \| u \|_2 ( \| e_2\|_2 + \| e_1  \|_2)$ and $0.5 C_1/C_2 > 0$,
this finishes the proof. 
\end{proof}
\begin{proof}[Proof of Lemma \ref{lemma:C1-implies-C2}]
By definition $D_\eta A D_\eta = 1 + e(\eta)$,  $\|e(\eta)\|_\infty \le \varepsilon_{\rm SK}$, 
and  then 
\[
\sum_j \eta_i A_{ij} \eta_j = 1+ e(\eta)_i \le 1+ \varepsilon_{\rm SK}, \quad \forall i.
\]
Since $A_{ij } \ge 0$, $\eta_j \ge C_1 >0$, we have that $\sum_j  A_{ij} \eta_j \ge C_1 \sum_j  A_{ij} \ge C_1 C_3$. 
This gives that for any $i$,
\[
\eta_i (C_1 C_3) \le 1+\varepsilon_{\rm SK},
\]
which proves the lemma.
\end{proof}

\subsection{Proofs in Section \ref{sec:theory-clean-data}}

\subsubsection{Proof of Proposition \ref{prop:barLn-pointwise}}
\begin{proof}[Proof of Proposition \ref{prop:barLn-pointwise}]
By definition,  for each $i = 1,\cdots,n $, 
 \[
 -\left( \bar{L}_n \rho_X f \right)_i  
=  \sum_{j=1}^n \bar{W}_{ij}(f(x_j) -f(x_i))  
= (1-\frac{1}{n}) \textcircled{1},
\]
where 
\begin{align}\label{eq:pf-barL-circ1}
\textcircled{1}  = q_\epsilon(x_i ) \frac{1}{n-1}\sum_{j, j\neq i}  Y_j, 
\quad
Y_j : =   \epsilon^{-d/2} 
	g \left( \frac{ \| x_i - x_j \|^2}{\epsilon} \right)   q_\epsilon(x_j ) (f(x_j)-f(x_i)).
\end{align}
Conditioning on $x_i$, $Y_j$ are $(n-1)$ independent random variables. 
Define 
\[
G_\epsilon f (x) : = \epsilon^{- {d}/{2}} \int_{\calM} g \left( \frac{ \| x-y \|^2}{\epsilon} \right) f(y) dV(y),
\]
then 
\[
\E Y_j =  G_\epsilon (q_\epsilon f p )(x_i) -  f(x_i) G_\epsilon (q_\epsilon p )(x_i).
\]
Since $q_\epsilon$ is $C^4$ (by Lemma \ref{lemma:construct-bar-eta}), 
$p $ is $C^6$ (by (A2)),
and $f$ is $C^4$ by the assumption of the proposition, 
then $q_\epsilon f p$ and $q_\epsilon p$ are both $C^4$.
Applying Lemma \ref{lemma:h-integral-diffusionmap} then gives that  (omitting evaluating at $x_i$ in the notation)
 \begin{align*}
G_\epsilon (q_\epsilon f p )  -  f G_\epsilon (q_\epsilon p )
& = q_\epsilon f p + \epsilon( \omega q_\epsilon f p + \Delta(q_\epsilon f p)) + O(\epsilon^2) \\
& ~~~ 
-  f \left(  q_\epsilon p + \epsilon( \omega q_\epsilon p + \Delta(q_\epsilon p)) + O(\epsilon^2)    \right) \\
& = \epsilon( \Delta(q_\epsilon f p) - f \Delta(q_\epsilon  p) ) + O( \epsilon^2).
\end{align*}
Recall that $q_\epsilon = p^{-1/2} (1+ \epsilon r) $, 
then
\[
\Delta(q_\epsilon f p) - f \Delta(q_\epsilon  p)  
= q_\epsilon p \Delta f + 2 \nabla (q_\epsilon p) \cdot \nabla f
= p^{1/2 } \Delta f + 2 \nabla (p^{1/2}) \cdot \nabla f + O(\epsilon).
\]
This gives that
\[
\E Y_j = \epsilon ( p^{1/2 } \Delta f + 2 \nabla (p^{1/2}) \cdot \nabla f  )(x_i)  + O(\epsilon^2),
\]
and thus
\begin{equation}\label{eq:qeps-EYj-1}
q_\epsilon(x_i) \E Y_j 
=    \epsilon \left( \Delta f + 2 p^{-1/2} \nabla (p^{1/2}) \cdot \nabla f  \right)(x_i)   + O(\epsilon^2)
= \epsilon  \Delta_p f  (x_i)   + O(\epsilon^2)
\end{equation}
By uniform boundedness of all the involved functions, the constant in $O(\epsilon^2)$ term is uniform for all $x_i$.
To analyze the concentration of the summation $\frac{1}{n-1}\sum_{j, j\neq i}  Y_j$ at $\E Y_j$, 
we bound the magnitude and variance of $Y_j$.
By the exponential decay of the Gaussian kernel $g$,
it incurs $O(\epsilon^{10})$ error by restricting the summation to $x_j \in B_\delta(x_i) \cap \calM$,
$\delta \sim \sqrt{ \epsilon \log (1/\epsilon)  }$.
This is equivalent to replace $Y_j$ with
\[
\tilde{Y}_j: =  \epsilon^{-d/2} 
	g \left( \frac{ \| x_i - x_j \|^2}{\epsilon} \right)   q_\epsilon(x_j ) (f(x_j)-f(x_i)) {\bf 1}_{ \{  \|x_j - x_i \|_2 < \delta \}}.
\]
By that $f\in C^4(\calM)$ and thus Lipschitz continuous (with respect to geodesic distance on local $B_\delta(x_i)$ neighborhood) on the manifold, 
and that locally geodesic metric and ambient space Euclidean metric matches up to 3rd order so that geodesic distance is bounded by Euclidean distance up to a constant factor (see e.g. Lemma 6 in \cite{coifman2006diffusion}),
we have that
\begin{equation}\label{eq:truncate-fj-fi-bound}
| (f(x_j)-f(x_i)) {\bf 1}_{ \{  \|x_j - x_i \|_2 < \delta \}} | \le \Theta(\delta) = O(   \sqrt{ \epsilon \log (1/\epsilon)  }).
\end{equation}
This leads to that  $|\tilde{Y}_j| \le L = \Theta( \epsilon^{-d/2}\sqrt{ \epsilon \log (1/\epsilon) } )$.
Meanwhile, 
$\E {Y}_j^2 \le \nu = \Theta( \epsilon^{-d/2+1})$.
By Classical Bernstein, attempting at a deviation  $t= \sqrt{ 40 \nu \log n /n} $, we obtain $e^{- nt^2/(4\nu)} = n^{-10}$ tail probability as long as $t < 3\nu / L $.
This is satisfied if  $\sqrt{ \nu \log n /n}  < C \nu /L $, where $C := 3/\sqrt{40}$, and is equivalently
\begin{equation}\label{eq:condition-Bernstein}
\frac{ \epsilon^{d/2} }{\log (1/\epsilon) }> C' \frac{\log n  }{n }
\end{equation}
for some $O(1)$ constant $C'$, which is satisfied under the condition of the proposition that $\epsilon^{d/2+1} \gg \log n/n$.
This gives that with large $n$ and w.p.$>1-2 n^{-10}$,
\[
\frac{1}{n-1}\sum_{j, j\neq i}  Y_j = \E Y_j + O\left(  \sqrt{ \frac{\log n}{n  \epsilon^{d/2-1}}} \right),
\]
and we call this good event $E_i$.
Back to \eqref{eq:pf-barL-circ1}, this together with \eqref{eq:qeps-EYj-1} gives 
\[
\textcircled{1} = q_\epsilon(x_i) \E Y_j + O\left(  \sqrt{ \frac{\log n}{n  \epsilon^{d/2-1}}} \right)
= \epsilon  \Delta_p f  (x_i)   
+ O\left(\epsilon^2,  \sqrt{ \frac{\log n}{n  \epsilon^{d/2-1}}} \right).
\]
This proves that, under good event $E_i$,
\[
\left( - \frac{1}{\epsilon}\bar{L}_n \rho_X f \right)_i  = (1-\frac{1}{n}) \frac{ \textcircled{1}}{\epsilon}
=  \Delta_p f  (x_i)   
+ O\left(\epsilon,  \sqrt{ \frac{\log n}{n  \epsilon^{d/2+1}}} \right),
\]
noting that the $ { \textcircled{1}}/{ (n \epsilon)}$ term leads to an error term of $O(n^{-1})$ which is of higher order than $ \sqrt{ {\log n}/{(n  \epsilon^{d/2+1} )}} $ since $\epsilon = o(1)$. 
Combining the $n$ good events of $E_i$ for $i=1,\cdots,n$ proves \eqref{eq:eqn-ptwise-barL} with the declared high probability. 
\end{proof}

\subsubsection{Proof of Theorem \ref{thm:hatLn-2norm} with extension}
 \begin{proof}[Proof of Theorem \ref{thm:hatLn-2norm}]
 Under the condition of the theorem, Proposition \ref{prop:barLn-pointwise} holds and gives that under a good event $E_{(0)}$ which happens w.p. $> 1- 2n^{-9}$,
  \begin{equation}\label{eq:bound-barL-pt-2}
\frac{1}{\sqrt{n}}\| \left( - \frac{1}{\epsilon} \bar{L}_n \rho_X f  \right) -  \rho_X ( \Delta_p f )   \|_2 
= O \left( \epsilon, \sqrt{ \frac{\log n }{n \epsilon^{d/2+1}}}\right),
\end{equation}
by that $\| v\|_2 \le \sqrt{n} \| v \|_\infty$ for any $v \in \R^n$.
To finish the proof, it remains to bound 
\begin{equation}\label{eq:bound-barL-hatL-to-show}
\| ( \hat{L}_n -  \bar{L}_n ) \rho_X f  \|_2 \overset{?}{=} 
\sqrt{n} \epsilon  \cdot 
O\left(\epsilon, \sqrt{\frac{\log n \log (1/\epsilon) }{n \epsilon^{d/2+1}}} \right).
\end{equation}
Recall that 
\begin{align*}
-( \bar{L}_n \rho_X f  )_i & = \bar{\eta}_i \sum_{j=1}^n  {W}_{ij}   \bar{\eta}_j ( f(x_j) - f(x_i) ), \\
-( \hat{L}_n \rho_X f  )_i & = \hat{\eta}_i \sum_{j=1}^n  {W}_{ij}   \hat{\eta}_j ( f(x_j) - f(x_i) ),
\end{align*}
and thus 
 \begin{align}
( ( \hat{L}_n -  \bar{L}_n ) \rho_X f  )_i
& =  \bar{\eta}_i \sum_{j=1}^n  {W}_{ij}   \bar{\eta}_j ( f(x_j) - f(x_i) )
 	- \hat{\eta}_i \sum_{j=1}^n  {W}_{ij}   \hat{\eta}_j ( f(x_j) - f(x_i) )  \nonumber \\
& = 	(\bar{\eta}_i - \hat{\eta}_i ) \sum_{j=1}^n  {W}_{ij}   \bar{\eta}_j ( f(x_j) - f(x_i) )
	+ \hat{\eta}_i\sum_{j=1}^n  {W}_{ij} (  \bar{\eta}_j -  \hat{\eta}_j )( f(x_j) - f(x_i) ) \nonumber \\
& =: 	 \textcircled{1}_i +  \textcircled{2}_i, \label{eq:def-circle1-circle2}
 \end{align}
 where the vectors $\textcircled{1}, \textcircled{2} \in \R^n$ and $( \hat{L}_n -  \bar{L}_n ) \rho_X f  = \textcircled{1} + \textcircled{2} $.
In below, we bound $\| \textcircled{1}\|_2$ and $\| \textcircled{2}\|_2$  respectively.

To proceed, we use Lemma \ref{lemma:eta1-eta2} to bound the relative error vector of $\bar{\eta}$ and $\hat{\eta}$ in 2-norm.
Call the good event in Lemma \ref{lemma:construct-bar-eta} the event $E_{(1)}$, 
under which  $\bar{\eta}$ is a solution to the $(C_{\rm SK}, \varepsilon_{\rm SK})$-matrix scaling problem of $W$ with $(C_{\rm SK},\varepsilon_{\rm SK})$ as in \eqref{eq:C-eps-modified-SK} by that lemma. 
Meanwhile, $\hat{\eta}$ is also a solution to the $(C_{\rm SK}, \varepsilon_{\rm SK})$-problem guaranteed by the algorithm. 
This gives that 
\begin{equation}\label{eq:proof-C1}
\min_i \bar{\eta}_i, \, \min_i \hat{\eta}_i \ge C_1 = C_{\rm SK} > 0.
\end{equation}
To obtain $C_2$ needed by Lemma \ref{lemma:eta1-eta2}, we use the concentration of the degree $D_{ii}:=\sum_j W_{ij}$ uniformly for all $i$, which is given in the following lemma. 
The proof can be directly derived from Lemma 6.1 in \cite{cheng2021eigen}, and is included in Appendix \ref{app:more-prooofs} for completeness.

\begin{lemma}[Concentration of degree of $W$]
\label{lemma:degree-conc-W}
Under Assumptions (A1) (A2), 
suppose as $n \to \infty$, $\epsilon \to 0+ $, $\epsilon^{d/2} = \Omega( {\log n}/{n} ) $,
then when $n$ is large enough, w.p. $> 1- 2 n^{-9}$, 
\begin{equation}\label{eq:degree-D-concen}
 \sum_{j=1}^n W_{ij}
= {p} (x_i) +  O \left( \epsilon, \sqrt{ \frac{\log n}{n \epsilon^{d/2}} }\right),
\quad 
 i = 1, \cdots, n.
\end{equation}
where the constant in big-O is uniform for all $i$ and depends on $p$ and $\calM$. 
\end{lemma}

Under the condition of the theorem, as $n$ increases, $\epsilon = o(1)$ and also $\epsilon \gg (\log n /n)^{1/(d/2+1)}$,
thus for large $n$,
the $O( \epsilon, \sqrt{ \log n/ (n \epsilon^{d/2}) })$ residual term in \eqref{eq:degree-D-concen}
 is $o(1)$ and suppose it is less than $0.1 p_{\rm min}$. 
Then under the good event in Lemma \ref{lemma:degree-conc-W}, called $E_{(2)}$ which happens w.p. $> 1-2n^{-9}$,
we have that 
\begin{equation}\label{eq:proof-C3}
D_{ii} = \sum_{j} W_{ij} \ge p(x_i) - 0.1 p_{\rm min} \ge 0.9 p_{\rm min} = C_3 > 0,
\quad i = 1, \cdots,n,
\end{equation}
that is, $0.9 p_{\rm min}$ serves as the constant $C_3$ needed by Lemma \ref{lemma:C1-implies-C2} for $A=W$.
Combined  with \eqref{eq:proof-C1}, 
Lemma \ref{lemma:C1-implies-C2} gives that, under $E_{(1)} \cap E_{(2)}$, we have (using $\varepsilon_{\rm SK} < 0.1$)
\begin{equation}\label{eq:proof-C2}
\max_i \bar{\eta}_i, \, \max_i \hat{\eta}_i \le  \frac{1+\varepsilon_{\rm SK}}{C_1 C_3} \le \frac{1.1}{  C_{\rm SK}  0.9 p_{\rm min }}
=   C_2, 
\end{equation}
and $C_2$ is an $O(1)$ constant depending on $(\calM, p)$ only. 
(Note that $\max_i \bar{\eta}_i \le q_{\rm max}$ which may be smaller than $C_2$ and improve the constant. 
In below, only $\max_i \hat{\eta}_i \le C_2$ is used in the rest of the proof.)

Now by Lemma \ref{lemma:eta1-eta2} we can bound both 
$(\bar{\eta} \oslash \hat{\eta} - {\bf 1})$
and $(\hat{\eta} \oslash \bar{\eta} - {\bf 1})$
and use them to bound $\| \textcircled{1}\|_2$ and $\| \textcircled{2}\|_2$. It turns out that using one of the ratio bound suffices. 
Specifically, define $\bar{u} \in \R^n$ such that 
\begin{equation}\label{eq:def-baru}
\hat{\eta} = \bar{\eta} \odot  (1 + \bar{u}),
\end{equation}
By \eqref{eq:proof-C1} and \eqref{eq:proof-C2}, applying Lemma \ref{lemma:eta1-eta2} with $ \varepsilon_{\rm SK} < 0.1$, we have that 
\begin{equation}\label{eq:bound-baru-2norm}
\| \bar{u} \|_2 \le \frac{2}{0.9} \frac{C_2}{C_1}  \sqrt{n} \varepsilon_{\rm SK}.
\end{equation}
We are ready to bound $\| \textcircled{1}\|_2$ and $\| \textcircled{2}\|_2$.

\vspace{5pt}
\noindent \underline{Bound of $\| \textcircled{1}\|_2$}: 
 By the definition of $ \textcircled{1}$ as in \eqref{eq:def-circle1-circle2} and $\bar{u}$ as in \eqref{eq:def-baru},
\begin{align}
 \textcircled{1}_i 
& = -  \bar{u}_i  \bar{\eta}_i \sum_{j=1}^n  {W}_{ij}   \bar{\eta}_j ( f(x_j) - f(x_i) ) 
	\label{eq:pf-circle1-2}
\end{align}
We claim that
\begin{equation}\label{eq:bound-sum-barWij-fj-fi}
\sum_{j} {W}_{ij} \bar{\eta}_i   \bar{\eta}_j  |(f(x_j)-f(x_i))|  
\le \Theta(  \sqrt{ \epsilon \log (1/\epsilon)  }),
\quad \forall i =1, \cdots, n,
\end{equation}
the proof of which is given later in beneath. 
If true, \eqref{eq:pf-circle1-2}  gives that
 \[
| \textcircled{1}_i| 
\le 
|  \bar{u}_i | \Theta( \sqrt{ \epsilon \log (1/\epsilon)  } ),
\]
 and thus, together with \eqref{eq:bound-baru-2norm}, we have
 \begin{equation}\label{eq:bound-circle1}
 \|\textcircled{1} \|_2 
 \le \Theta( \sqrt{ \epsilon \log (1/\epsilon)  } ) \|  \bar{u} \|_2
 \le \Theta( \sqrt{ \epsilon \log (1/\epsilon)  } )   \sqrt{n} \varepsilon_{\rm SK}
 \end{equation}

\vspace{5pt}
\noindent
\underline{
Proof of Claim \eqref{eq:bound-sum-barWij-fj-fi}}:
By definition,
\begin{align}
 \sum_{j} {W}_{ij} \bar{\eta}_i   \bar{\eta}_j  |(f(x_j)-f(x_i))|  
& =  \frac{1}{n}    q_\epsilon(x_i ) \sum_{j=1}^n
  \epsilon^{-d/2} 
	g \left( \frac{ \| x_i - x_j \|^2}{\epsilon} \right)   q_\epsilon(x_j ) |f(x_j)-f(x_i)|. \nonumber
\end{align}
We are to use the same truncation argument 
as in the proof of Proposition \ref{prop:barLn-pointwise}
of the kernel on Euclidean ball $B_{\delta}(x_i)$
where 
\begin{equation*}
\delta \sim \sqrt{ \epsilon \log (1/\epsilon)  }.
\end{equation*}
To bound the contribution of the summation for $x_j$ outside $B_{\delta}(x_i)$,
we use that both $f(x_j)$ and $q_\epsilon(x_j)$ are uniformly bounded by $O(1)$ constant for all $j$ (by Lemma \ref{lemma:construct-bar-eta}(i), $ q_\epsilon(x_i ) $ are positive and bounded by $O(1)$ constant $q_{\rm max } $ for all $i$),
and thus the summation outside $B_{\delta}(x_i)$ can be bounded by $O(  \epsilon^{10})$.
Then, the kernel truncation gives that 
\begin{align}
& \frac{1}{n}    q_\epsilon(x_i ) \sum_{j=1}^n
  \epsilon^{-d/2} 
	g \left( \frac{ \| x_i - x_j \|^2}{\epsilon} \right)   q_\epsilon(x_j ) |f(x_j)-f(x_i)| \nonumber \\
& =\frac{1}{n}    q_\epsilon(x_i ) \sum_{j=1}^n
  \epsilon^{-d/2} 
	g \left( \frac{ \| x_i - x_j \|^2}{\epsilon} \right)   q_\epsilon(x_j ) |f(x_j)-f(x_i)| {\bf 1}_{ \{  \|x_j - x_i \|_2 < \delta \}}
	+ O(  \epsilon^{10}). \label{eq:pf-circle1-1}
\end{align}
Next, same as before, by \eqref{eq:truncate-fj-fi-bound}, 
$| (f(x_j)-f(x_i)) {\bf 1}_{ \{  \|x_j - x_i \|_2 < \delta \}} | \le \Theta(\delta)$.
and then
\begin{align}
&
\frac{1}{n}    q_\epsilon(x_i ) \sum_{j=1}^n
  \epsilon^{-d/2} 
	g \left( \frac{ \| x_i - x_j \|^2}{\epsilon} \right)   q_\epsilon(x_j ) |f(x_j)-f(x_i)| {\bf 1}_{ \{  \|x_j - x_i \|_2 < \delta \}}
	 \nonumber \\
& \le
\frac{1}{n}    q_\epsilon(x_i ) \sum_{j=1}^n
  \epsilon^{-d/2} 
	g \left( \frac{ \| x_i - x_j \|^2}{\epsilon} \right)   q_\epsilon(x_j ) 
	{\bf 1}_{ \{  \|x_j - x_i \|_2 < \delta \}}  \Theta(\delta)  \nonumber \\
& =	 \Theta(\delta)  
   \bar{\eta}_i \sum_{j=1}^n W_{ij}   \bar{\eta}_j 
	{\bf 1}_{ \{  \|x_j - x_i \|_2 < \delta \}}  \nonumber \\
& \le 	 \Theta(\delta)  
   \bar{\eta}_i \sum_{j=1}^n W_{ij}   \bar{\eta}_j   \nonumber \\
&   = 	 \Theta(\delta)  (1 + e(\bar{\eta})_i )
   \le \Theta(\delta)  (1 + \varepsilon_{\rm SK}) = \Theta( \sqrt{ \epsilon \log (1/\epsilon)  }).
   \label{eq:bound-row-sum-etac-Wc-1}
\end{align}
Putting back to \eqref{eq:pf-circle1-1}, we have that
\begin{align}
& 
 \frac{1}{n}    q_\epsilon(x_i ) \sum_{j=1}^n
  \epsilon^{-d/2} 
	g \left( \frac{ \| x_i - x_j \|^2}{\epsilon} \right)   q_\epsilon(x_j ) |f(x_j)-f(x_i)| 
 \nonumber \\
& \le\Theta( \sqrt{ \epsilon \log (1/\epsilon)  }) + O(  \epsilon^{10}) 
= \Theta( \sqrt{ \epsilon \log (1/\epsilon)  } ).
\end{align}
The argument holds for each $i$ and  this proves \eqref{eq:bound-sum-barWij-fj-fi}.

 \vspace{5pt}
 \noindent \underline{Bound of $\| \textcircled{2}\|_2$}:  
 By definition of $ \textcircled{2}$ as in \eqref{eq:def-circle1-circle2} and $\bar{u}$ as in \eqref{eq:def-baru},
\begin{align}
 \textcircled{2}_i
 & = - \bar{\eta}_i  (1+\bar{u}_i) \sum_{j=1}^n  {W}_{ij}  \bar{\eta}_j \bar{u}_j ( f(x_j) - f(x_i) ) \nonumber \\
& = -  (1+\bar{u}_i) 
 \sum_{j=1}^n \bar{\eta}_i  W_{ij} \bar{\eta}_j    (f(x_j)-f(x_i))  \bar{u}_j.
\end{align}
Note that by the uniform lower-bound of $\bar{\eta}_i$  \eqref{eq:proof-C1} 
and uniform upper-bound of $\hat{\eta}_i$ in \eqref{eq:proof-C2},
\[
| 1+\bar{u}_i | = \frac{\hat{\eta}_i}{\bar{\eta}_i} \le \frac{C_2}{C_1}, 
\quad \forall i,
\]
and thus
\begin{equation}
| \textcircled{2}_i | \le \frac{C_2}{C_1} \textcircled{3}_i,
\quad
 \textcircled{3}_i : = 
 \sum_{j=1}^n \bar{\eta}_i  W_{ij} \bar{\eta}_j    |f(x_j)-f(x_i)|   | \bar{u}_j |
    =  (A v)_i,
\end{equation}
where $A$ is an $n$-by-$n$ matrix  and $v \in \R^n$ are defined as 
\[
A_{ij} =   \bar{\eta}_i    W_{ij} \bar{\eta}_j   |f(x_j)-f(x_i)| ,
\quad
v_j = | \bar{u}_j|.
\]
By definition, we have that 
\begin{equation}\label{eq:pf-circle2-1}
\| \textcircled{2}  \|_2 \le \frac{C_2}{C_1} \| \textcircled{3}  \|_2, \quad \textcircled{3}   = Av, 
\quad \|v\|_2 = \|   \bar{u}\|_2.
\end{equation}
Because $A$ is real-symmetric with non-negative entries, 
the operator norm $\|A\|_2$ is bounded by the Perron-Frobenius eigenvalue $\rho$ of $A$, 
which is furtherly bounded by the maximum row-sum of $A$. 
Thus,
\[
\| A\|_2 \le 
\max_{i} \sum_j A_{ij}
= \max_{i}  \sum_{j} {W}_{ij} \bar{\eta}_i   \bar{\eta}_j  |(f(x_j)-f(x_i))|  
\le \Theta(  \sqrt{ \epsilon \log (1/\epsilon)  }),
\]
where the last inequality is by  \eqref{eq:bound-sum-barWij-fj-fi}.
Combined with \eqref{eq:pf-circle2-1} and  \eqref{eq:bound-baru-2norm}, we have
\begin{equation}\label{eq:bound-circle2}
\| \textcircled{2} \|_2  
\le \frac{C_2}{C_1} \| \textcircled{3} \|_2  
\le  \frac{C_2}{C_1}  \|A\|_2 \|v\|_2 
=   \frac{C_2}{C_1}  \|A\|_2 \|\bar{u} \|_2 
\le \Theta( \sqrt{ \epsilon \log (1/\epsilon)  } )\sqrt{n} \varepsilon_{\rm SK}.
\end{equation}

Putting together \eqref{eq:bound-circle1} and \eqref{eq:bound-circle2} and back to \eqref{eq:def-circle1-circle2},
we have that 
\begin{equation}\label{eq:bound-circ1+circ2}
\| ( \hat{L}_n -  \bar{L}_n ) \rho_X f  \|_2
\le  \| \textcircled{1} \|_2 + \| \textcircled{2}  \|_2
= \Theta( \sqrt{ \epsilon \log (1/\epsilon)  } )\sqrt{n} \varepsilon_{\rm SK}.
\end{equation}
Recall that 
$\varepsilon_{\rm SK} =  O \left( \epsilon^2, \sqrt{\frac{\log n }{n \epsilon^{d/2}}} \right)$, 
then 
\[
\Theta( \sqrt{ \epsilon \log (1/\epsilon)  } )\sqrt{n} \varepsilon_{\rm SK}
= \sqrt{n} \epsilon    \cdot O \left( \epsilon^{3/2} \sqrt{\log (1/\epsilon)  } ,  \sqrt{\frac{\log n \log (1/\epsilon) }{n \epsilon^{d/2+1}}} \right).
\]
Because $ \epsilon^{3/2} \sqrt{\log (1/\epsilon)  } = O(\epsilon)$,
this proves \eqref{eq:bound-barL-hatL-to-show} under the intersection of the good events  $E_{(1)}$ and $E_{(2)}$.

Together with \eqref{eq:bound-barL-pt-2}, 
we have shown that the claim of the theorem holds under 
the intersection of the good events  $E_{(0)}$, $E_{(1)}$ and $E_{(2)}$,
which happens w.p.$> 1- 6n^{-9}$ for large $n$. 
 \end{proof}
  \begin{proof}[Proof of extension of Theorem \ref{thm:hatLn-2norm} ]
  We show that the proof extends if one replaces $W$ with $W^0$ in the computation of bi-stochastically normalized graph Laplacian.
  
  First, Lemma \ref{lemma:construct-bar-eta} extends to $W^0$, as shown in Remark \ref{rk:W-diagonal-zero}.
  To apply the extension of Lemma \ref{lemma:eta1-eta2}, i.e. Lemma \ref{lemma:eta1-eta2-nonPSD}, to $W^0$,
we verify that letting $A = W^0$ satisfies the needed conditions:

For the two additional conditions as stated in Remark \ref{rk:hateta-W0},
condition (i) is satisfied by setting $\delta$ to be the diagonal entries of $W$, 
because $\delta = (4\pi \epsilon)^{-d/2}/n $ is $o(1)$ under the condition that $\epsilon^{d/2} = \Omega(\log n/n)$. 
Condition (ii) requires the uniform upper-boundedness of both $\hat{\eta}$ and $\bar{\eta}$.
The uniform upper and lower boundedness of $\bar{\eta}$ follows by that of $q_\epsilon$ as before.
To show the uniform upper boundedness of $\hat{\eta}$, 
it suffices to have $\sum_{j} W^0_{ij}$ for all $i$ lower-bounded by $C_3$ and then one can apply Lemma \ref{lemma:degree-conc-W}.
The uniform lower-boundedness of $\sum_{j} W^0_{ij}$
is guaranteed by extending Lemma \ref{lemma:degree-conc-W} to $W^0$,
due to that the diagonal entry $W_{ii} = (4 \pi \epsilon)^{-d/2}/n = o(1)$ under the asymptotic condition of $\epsilon$. 

As a result, Lemma \ref{lemma:eta1-eta2-nonPSD} applies to $A = W^0$ to give the bound of the 2-norm of $\bar{u}$ as
$
\| \bar{u}\|_2 \le \frac{2 C_2 }{0.5 C_1} \sqrt{n }\varepsilon_{\rm SK}
$
which is the counterpart of  \eqref{eq:bound-baru-2norm}.
The rest of the proof of the theorem is by comparing the vector $\hat{L}_n \rho_X f$ to $\bar{L}_n \rho_X f$,
which extends because the diagonal entries of $W$ are not involved in the computation of both vectors. 
\end{proof}

\subsection{Proofs in Section \ref{sec:theory-noise-outlier}}

\subsubsection{Proofs of Lemma \ref{lemma:uniform-H} and \ref{lemma:bareta-noise}}
\begin{proof}[Proof of Lemma \ref{lemma:uniform-H}]
Under Assumption (A3), in the joint limit of $n$ and $\epsilon$, 
$ \varepsilon_z/ \epsilon  = o(1)$. 
We have that for large enough $n$ such that $\varepsilon_z/ \epsilon < 0.1$, 
\[
\frac{|r_{ij}|}{\epsilon} \le \frac{\varepsilon_z}{\epsilon} < 0.1, \quad \forall i\neq j,
\]
and then by that $|e^{- x/4} -1 | \le |x|$ for $|x| < 0.1$, we have that 
\begin{equation}\label{eq:uniform-bound-Hij}
\sup_{i\neq j} |  e^{ -\frac{r_{ij}}{4\epsilon}} -1| \le  \frac{\varepsilon_z}{\epsilon}.
\end{equation}
This shows the uniform bound of $|H_{ij}|$ in \eqref{eq:W0-W'-relation}
by the definition of $H$ as in \eqref{eq:def-Hij}.
\end{proof}
\begin{proof}[Proof of Lemma \ref{lemma:bareta-noise}]
For any $i=1,\cdots,n$,
\begin{align}
\bar{\eta}_i \sum_{j=1}^n W^0_{ij} \bar{\eta}_j 
&= \sum_{j \neq i} \bar{\eta}_i^c \rho_i^{-1} W^0_{ij}  \rho_j^{-1} \bar{\eta}_j^c \nonumber \\
& = \sum_{j \neq i} \bar{\eta}_i^c   W_{ij}^{c,0} (1+H_{ij})  \bar{\eta}_j^{c} \nonumber \\
& =  \sum_{j \neq i} \bar{\eta}_i^c   W_{ij}^{c,0}  \bar{\eta}_j^{c} 
	 + \sum_{j \neq i} \bar{\eta}_i^c   W_{ij}^{c,0} \bar{\eta}_j^{c}  H_{ij}.
	 \label{eq:bound-sum-pf-etac-0}
\end{align}
By Lemma \ref{lemma:uniform-H}, for large enough $n$ and under the good event $E_{(z)}$,
$\sup_{ 1 \le i,j \le n}|H_{ij}| \le {\varepsilon_z}/{\epsilon}$.
Then, by that $W_{ij}^{c,0}$ and  $\bar{\eta}_i^{c}$ are non-negative,
the second term in \eqref{eq:bound-sum-pf-etac-0} can be bounded by
\[
\left| \sum_{j \neq i} \bar{\eta}_i^c   W_{ij}^{c,0} \bar{\eta}_j^{c}  H_{ij}  \right|
\le \sum_{j \neq i} \bar{\eta}_i^c   W_{ij}^{c,0} \bar{\eta}_j^{c} |  H_{ij}  |
\le  \left( \sum_{j \neq i} \bar{\eta}_i^c   W_{ij}^{c,0} \bar{\eta}_j^{c} \right) \frac{\varepsilon_z}{\epsilon}.
\]
This gives that, for any $i$,
\begin{equation}\label{eq:bareta-W0-1}
\bar{\eta}_i \sum_{j=1}^n W^0_{ij} \bar{\eta}_j 
= \left( \sum_{j \neq i} \bar{\eta}_i^c   W_{ij}^{c,0}  \bar{\eta}_j^{c} \right) ( 1 +  O(\frac{\varepsilon_z}{\epsilon})),
\end{equation}
where the constant in big-O is an absolute one.

By Lemma \ref{lemma:construct-bar-eta}, 
under the good event $E_{(1)}$ which happens w.p. $>1-2 n^{-9}$ for large $n$, for all $i$,
\begin{equation}\label{eq:bound-row-sum-E1-Wcij}
 \sum_{j = 1}^n 
 \bar{\eta}_i^c   W_{ij}^c  \bar{\eta}_j^c
 = 1 + \bar{e}^c_i,
 \quad 
 |\bar{e}^c_i| \le \varepsilon^c = O \left( \epsilon^2, \sqrt{\frac{\log n }{n \epsilon^{d/2}}} \right),
\end{equation}
and in addition, 
\begin{equation}\label{eq:uniform-bound-baretac}
0 < q_{\rm min} \le \bar{\eta}_i^c \le q_{\rm max}, \quad i=1,\cdots, n.
\end{equation}
Since 
\begin{equation}\label{eq:Wcii}
W_{ii}^c = \frac{1}{n}
\epsilon^{-d/2} g \left(  0 \right) = \frac{\epsilon^{-d/2}}{n (4\pi)^{d/2}}, 
\end{equation}
we have that 
\[
0 \le (\bar{\eta}_i^c)^2 W_{ii}^c \le \frac{q_{\rm max}^2}{(4\pi)^{d/2}} \frac{\epsilon^{-d/2}}{n } 
=  \Theta( \frac{1}{n  \epsilon^{d/2}}  )
\]
which is of higher order than $\varepsilon^c$. As a result, 
we have that for all $i$,
\begin{equation}\label{eq:bound-sum-pf-etac-1}
 \sum_{j \neq i} \bar{\eta}_i^c   W_{ij}^c \bar{\eta}_j^c 
 = 
( D_{  \bar{\eta}^c   } W^c D_{  \bar{\eta}^c   } {\bf 1 } )_i  
 	- (\bar{\eta}_i^c)^2 W_{ii}^c
= 1 + e_i,
\quad e_i = \bar{e}^c_i + \Theta( \frac{1}{n  \epsilon^{d/2}}  ) = O \left( \epsilon^2, \sqrt{\frac{\log n }{n \epsilon^{d/2}}} \right),
\end{equation}
and the constant in big-O is uniform for all $i$.
Combining \eqref{eq:bareta-W0-1} and \eqref{eq:bound-sum-pf-etac-1}, we have that
\[
\bar{\eta}_i \sum_{j=1}^n W^0_{ij} \bar{\eta}_j 
= \left(1+ e_i \right) ( 1 +  O(\frac{\varepsilon_z}{\epsilon}))
= 1 + O \left( \epsilon^2, \sqrt{\frac{\log n }{n \epsilon^{d/2}}},  \frac{\varepsilon_z}{\epsilon} \right),
\quad \forall i = 1,\cdots, n,
\]
with uniform constant in big-O. This proves \eqref{eq:bareta-noise-existence} under the intersection of $E_{(1)}$ and $E_{(z)}$ in (A3). 
\end{proof}

\subsubsection{Proofs  in Step 2.}

\begin{proof}[Proof of Lemma \ref{lemma:hatetac-upper}]

We first make the following claim:

There is $O(1)$ constant $C_3$ s.t. for large $n$, under a good event $E_{(2)}$ which happens w.p. $>1-2n^{-9}$,
\begin{equation}\label{eq:Wc0-C3-claim}
\sum_{j} W^{c,0}_{ij}\rho_j \ge C_3 > 0,
\quad i=1,\cdots,n.
\end{equation}

Suppose the claim holds, we prove the lemma.
Because $\hat{\eta}$ is an $\varepsilon_{\rm SK}$-approximate scaling factor of $W^0$, 
for all $i$,
\[
\sum_{j \neq i} \hat{\eta}_i W^0_{ij}  \hat{\eta}_j 
= \sum_{j \neq i} \hat{\eta}_i W_{ij}'(1+ H_{ij}) \hat{\eta}_j 
= 1 + \hat{e}_i, 
\quad \| \hat{e} \|_\infty \le \varepsilon_{\rm SK}.
\]
Under the good event $E_{(z)}$ of (A3), by that $W_{ij}'$ and $\hat{\eta}_i$ are non-negative, similarly as in \eqref{eq:bareta-W0-1} we have that 
\begin{equation}
1 + \hat{e}_i  
= \left( \sum_{j \neq i} \hat{\eta}_i W_{ij}'  \hat{\eta}_j  \right)
( 1 +  O( \frac{\varepsilon_z}{\epsilon})),
\end{equation}
where the constant in big-O is an absolute one.
Because $|\hat{e}_i| \le \varepsilon_{\rm SK}$ which is $o(1)$ under the condition of the lemma, 
and  ${\varepsilon_z}/{\epsilon}$ is also $o(1)$ under (A3),
we have that for large $n$,
\begin{equation}\label{eq:sum-hateta-W'-2}
\sum_{j \neq i} \hat{\eta}_i W_{ij}' \hat{\eta}_j  \le 1.1,  \quad i=1, \cdots, n.
\end{equation}
Meanwhile, by \eqref{eq:def-hateta-noise}  and \eqref{eq:def-W'},
\[
\sum_{j \neq i} \hat{\eta}_i W_{ij}' \hat{\eta}_j 
=  \sum_{j \neq i} \hat{\eta}_i^c \frac{1}{\rho_i} W_{ij}'  \frac{1}{\rho_j}  \hat{\eta}_j^c 
=  \hat{\eta}_i^c  \sum_{j \neq i}  W^{c,0}_{ij} \rho_j \hat{\eta}_j 
\ge \hat{\eta}_i^c  C_1 C_3,
\]
where the last inequality is by that $\min_i \hat{\eta}_i \ge C_1 = C_{\rm SK}$  (by the requirement of the $(C_{\rm SK}, \varepsilon_{\rm SK})$-problem)
and  the claim \eqref{eq:Wc0-C3-claim}.
Together with \eqref{eq:sum-hateta-W'-2}, this gives that 
\[
\hat{\eta}_i^c  C_1 C_3 \le 1.1, \quad i=1,\cdots,n, 
\]
and this proves the lemma with $C_2 = 1.1/(C_1C_3)$.

It remains to prove the claim \eqref{eq:Wc0-C3-claim}.
Observe that for each $i$,
\begin{align*}
d_i 
:= \sum_{j\neq i} W^{c,0}_{ij}\rho_j 
& = \sum_{j \neq i, \, b_j =0 } W^{c,0}_{ij} 
    + \sum_{j \neq i, \, b_j =1 } W^{c,0}_{ij}  \rho_j \\
& \ge \sum_{j\neq i, \, b_j =0 } W^{c,0}_{ij} 
 = \sum_{j \neq i} W^{c,0}_{ij} (1-b_j) =: d_i',
\end{align*}
which is an independent sum of $(n-1)$ many r.v. conditioning on $x_i^c$, namely, 
\begin{align*}
d_i'
& = \frac{1}{n} \sum_{j \neq i}  \epsilon^{-d/2} g \left(  \frac{ \| x_i^c - x_j^c\|^2}{\epsilon} \right) (1-b_j) \\
& = (1-\frac{1}{n}) \frac{1}{n-1} \sum_{j \neq i} Y_j
\quad Y_j :=  \epsilon^{-d/2} g \left(  \frac{ \| x_i^c - x_j^c\|^2}{\epsilon} \right) (1-b_j)
\end{align*}
due to that $(x_i^c, b_i)$ are i.i.d. across $i$ by Assumption (A3). 
Using the same argument as in the proof of degree concentration in Lemma \ref{lemma:degree-conc-W},
one can verify that  $|Y_j| \le L_Y = \Theta( \epsilon^{-d/2})$,
$\text{Var}(Y_j) \le \nu_Y = \Theta( \epsilon^{-d/2} )$, 
and by (A3)(i), 
\begin{align*}
\E [Y_j | x_i^c] 
& = \E [(1- p_j) \epsilon^{-d/2} g \left(  \frac{ \| x_i^c - x_j^c\|^2}{\epsilon} \right) | x_i^c ] \\
& \ge  (1- p_{\rm out})   \E [  \epsilon^{-d/2} g \left(  \frac{ \| x_i^c - x_j^c\|^2}{\epsilon} \right)  | x_i^c  ] \\
& =  (1- p_{\rm out}) (  p^c( x_i^c  ) + O(\epsilon) ) 
\ge (1- p_{\rm out})   p_{\rm min} + O(\epsilon) ,
\end{align*}
where all the constants in big-$\Theta$ and big-$O$ are uniform for all $i$.
This gives that for each $i$, under a good event happenning w.p. $> 1-2n^{-10}$,
\begin{align*}
d_i'
&= \E [Y_j | x_i^c] + O(\sqrt{ \frac{\log n}{n \epsilon^{d/2}} }  ) 
 \ge (1- p_{\rm out})   p_{\rm min} + O(\epsilon , \sqrt{ \frac{\log n}{n \epsilon^{d/2}} }  ). 
\end{align*}
Let $E_{(2)}$ be the intersection of the $n$ good events (over the randomness of $(x_i^c, b_i)$ for all $i$),
under which we have
\[
d_i \ge d_i' \ge (1- p_{\rm out})   p_{\rm min} + O(\epsilon , \sqrt{ \frac{\log n}{n \epsilon^{d/2}} }  ),
\quad i = 1, \cdots, n.
\]
Since the constant $(1- p_{\rm out})   p_{\rm min}  > 0$, 
and under the condition of the lemma $O \left( \epsilon, \sqrt{ \frac{\log n}{n \epsilon^{d/2}} }\right) = o(1)$, 
we have that, for large $n$ and under $E_{(2)}$,
$d_i$ is uniformly bounded from below by an $O(1)$ constant $C_3 > 0$.
\end{proof}
\begin{proof}[Proof of Lemma \ref{lemma:hatetac-baretac}]
Under $E_{(1)}$ in Lemma \ref{lemma:bareta-noise}, \eqref{eq:hateta-bareta-epsSK} holds. 
By construction and \eqref{eq:def-u-hateta-bareta}, 
similarly as in the derivation of \eqref{eq:proof-e1-e2-1} in the proof of Lemma \ref{lemma:eta1-eta2}, 
 we have
 \[
 e(\hat{\eta}) - e(\bar{\eta})
 = D_{\bar{\eta}} W^0 D_{\bar{\eta}}  u +  D_u D_{\bar{\eta}} W^0 D_{\bar{\eta}}  ({\bf 1}+u),
 \]
and multiplying $u^T$ to both sides of it gives 
\begin{equation}\label{eq:proof-e1-e2-zerodiag-1}
u^T(  e(\hat{\eta}) - e(\bar{\eta})) 
=  u^T A u  + u^T D_u  A ({\bf 1}+u), 
\end{equation}
where
\[
 A:= D_{\bar{\eta}} W^0 D_{\bar{\eta}},
\quad  A {\bf 1} = {\bf 1}  + e(\bar{\eta}).
\]

For the 2nd term in the r.h.s. of \eqref{eq:proof-e1-e2-zerodiag-1}, since $ \hat{\eta}_i = \bar{\eta}_i (1+u_i)$,
\begin{align}
u^TD_u A ({\bf 1} + u) 
= \sum_{i=1}^n u_i^2  \sum_{j=1}^n \bar{\eta}_i W^0_{ij} \hat{\eta}_j  
 = \sum_{i=1}^n  \frac{u_i^2 }{1+u_i} \left( \sum_{j=1}^n \hat{\eta}_i W^0_{ij} \hat{\eta}_j \right).
 \label{eq:2nd-term-1}
\end{align}
Note that by the definitions \eqref{eq:def-bareta-noise} and \eqref{eq:def-hateta-noise},
\begin{equation}\label{eq:bound-1+ui-noise}
1+u_i = \frac{\hat{\eta}_i}{\bar{\eta}_i} =  \frac{\hat{\eta}_i^c}{\bar{\eta}_i^c}  \le \frac{C_2}{q_{\rm min}}, \quad \forall i,
\end{equation}
where the last inequality is by Lemma \ref{lemma:hatetac-upper} (under $E_{(2)} \cap E_{(z)}$)
 and that $\bar{\eta}_i^c \ge  q_{\rm min}$ as in \eqref{eq:uniform-bound-baretac}.
Meanwhile, 
\[ 
\sum_{j=1}^n \hat{\eta}_i W^0_{ij} \hat{\eta}_j  = 1+ e(\hat{\eta})_i \ge 1- \varepsilon_{\rm SK} > 0.9,
\]
and then \eqref{eq:2nd-term-1} continues as
\begin{equation}\label{eq:2nd-term-2}
u^TD_u A ({\bf 1} + u) 
\ge \frac{q_{\rm min}}{C_2} \sum_{i=1}^n  u_i^2  \left( \sum_{j=1}^n \hat{\eta}_i W^0_{ij} \hat{\eta}_j \right)
\ge 0.9 \frac{q_{\rm min}}{C_2} \| u \|_2^2.
\end{equation}

For the 1st term in the r.h.s. of \eqref{eq:proof-e1-e2-zerodiag-1}, 
though $A $ may not be PSD, by \eqref{eq:W0-and-W'} we have
\[
 A = D_{\bar{\eta}} (W' + W' \odot H) D_{\bar{\eta}}  
= D_{\bar{\eta}} W' D_{\bar{\eta}}   +D_{\bar{\eta}} W'  D_{\bar{\eta}}   \odot H.
\]
Recall that by definition  \eqref{eq:barW'-2},
$\bar{W}' = D_{\bar{\eta}} W' D_{\bar{\eta}} = D_{\bar{\eta}^c} W^{c,0} D_{\bar{\eta}^c}$,
and then
\[
A= \bar{W}'    + \bar{W}'   \odot H.
\]
This gives that 
\[
u^T A u = u^T \bar{W}'  u + u^T (\bar{W}'   \odot H)u =: \textcircled{1} +  \textcircled{2}.
\]
To bound $|\textcircled{2}|$,  
by Lemma \ref{lemma:uniform-H}, with large $n$ and under $E_{(z)}$, we have \eqref{eq:W0-W'-relation}, and then 
\begin{equation}\label{eq:circle2-1}
|\textcircled{2}| \le  \sum_{i,j} |u_i| \bar{W}'_{ij} |H_{ij}| | u_j |
\le \frac{\varepsilon_z}{\epsilon}  \left( \sum_{i,j} |u_i| \bar{W}'_{ij}  |u_j | \right).
\end{equation}
Note that as shown in \eqref{eq:bound-sum-pf-etac-1}, under $E_{(1)}$  of Lemma \ref{lemma:bareta-noise}, for all $i$,
\begin{equation}
\sum_{j=1}^n \bar{W}'_{ij} = \sum_{j \neq i} \bar{\eta}_i^c   W_{ij}^c \bar{\eta}_j^c 
\le 1 + O \left( \epsilon^2, \sqrt{\frac{\log n }{n \epsilon^{d/2}}} \right).
\end{equation}
Because $\bar{W}'$ is a symmetric matrix having non-negative entries, 
its operator norm is bounded by its Perron-Frobenius eigenvalue which is furtherly bounded by the maximum row-sum 
which is $1 + O \left( \epsilon^2, \sqrt{\frac{\log n }{n \epsilon^{d/2}}} \right)$.
Under the condition of the lemma, $O \left( \epsilon^2, \sqrt{\frac{\log n }{n \epsilon^{d/2}}} \right) =o(1)$,
and thus for large enough $n$, 
\[
\| \bar{W}' \|_2 \le 1 + O \left( \epsilon^2, \sqrt{\frac{\log n }{n \epsilon^{d/2}}} \right) \le 1.1.
\]
Back to \eqref{eq:circle2-1}, this gives that 
\[
|\textcircled{2}| 
\le \frac{\varepsilon_z}{\epsilon}  \| \bar{W}' \|_2 \| u\|_2^2 
\le 1.1 \frac{\varepsilon_z}{\epsilon} \| u\|_2^2 .
\]
As for  $\textcircled{1}$, by \eqref{eq:def-Wc0},
\[
\textcircled{1} 
= u^T \bar{W}'  u
= u^T D_{\bar{\eta}^c} W^{c,0} D_{\bar{\eta}^c} u
= u^T D_{\bar{\eta}^c} (W^{c} - \beta_n I) D_{\bar{\eta}^c} u
\ge - \beta_n \| u\odot \bar{\eta}^c \|_2^2,
\]
where the last inequality is by that $W^{c}  \succeq 0$.
Because $\max_i \bar{\eta}^c_i \le  q_{\rm max}$ by \eqref{eq:uniform-bound-baretac}, 
$\| u\odot \bar{\eta}^c \|_2 \le q_{\rm max}   \|u\|_2$, and then
\[
\textcircled{1}  \ge - \beta_n q_{\rm max}^2  \|u\|_2^2.
\]
Putting together, we have that 
\[
u^T A u = \textcircled{1} +  \textcircled{2}
\ge - \beta_n q_{\rm max}^2  \|u\|_2^2 - 1.1 \frac{\varepsilon_z}{\epsilon} \| u\|_2^2.
\]
Because  $\beta_n  = \Theta( \frac{\epsilon^{-d/2}}{n } )$ and $\frac{\varepsilon_z}{\epsilon}$ are both $o(1)$ under the condition of the lemma,
we have  that
\begin{equation}\label{eq:1st-term-2}
u^T A u 
\ge - \delta_n \| u\|_2^2, \quad \delta_n 
= \Theta \left( \frac{\epsilon^{-d/2}}{n } , \frac{\varepsilon_z}{\epsilon} \right) = o(1). 
\end{equation}

Inserting \eqref{eq:1st-term-2} and \eqref{eq:2nd-term-2} to \eqref{eq:proof-e1-e2-zerodiag-1}, we have that 
\[
u^T(  e(\hat{\eta}) - e(\bar{\eta})) 
\ge   (0.9 \frac{q_{\rm min}}{C_2}  - \delta_n )  \|u\|_2^2.
\]
Because $ \delta_n = o(1)$, with large $n$,  $ 0.9 \frac{q_{\rm min}}{C_2}  - \delta_n  > 0.8 \frac{q_{\rm min}}{C_2} $,
and then we have
\[
0.8 \frac{q_{\rm min}}{C_2}   \|u\|_2 \le  \|  e(\hat{\eta}) - e(\bar{\eta}) \|_2 \le \|  e(\hat{\eta}) \|_2 + \| e(\bar{\eta}) \|_2,
\]
which proves the lemma combined with that $\| e(\hat{\eta})\|_\infty, \| e(\bar{\eta})\|_\infty \le \varepsilon_{\rm SK}$
and that $\| v\|_2 \le \sqrt{n} \|v\|_\infty$ for any $v \in \R^n$.
\end{proof}

\begin{proof}[Proof of Proposition \ref{prop:hatL'-convergence}]
By definitions of $\hat{L}'$ and $ \bar{L}' $ in \eqref{eq:def-three-Ls-noise},
similarly as in \eqref{eq:def-circle1-circle2}, we have
 \begin{align*}
( ( \hat{L}' -  \bar{L}' ) \rho_X f  )_i
& =  \bar{\eta}_i \sum_{j=1}^n  {W}_{ij}'   \bar{\eta}_j ( f(x_j) - f(x_i) )
 	- \hat{\eta}_i \sum_{j=1}^n  {W}_{ij}'   \hat{\eta}_j ( f(x_j) - f(x_i) )  = 	 \textcircled{1}_i +  \textcircled{2}_i,
 \end{align*}
where
\begin{align*}
\textcircled{1}_i 
& =  (\bar{\eta}_i - \hat{\eta}_i ) \sum_{j=1}^n  {W}_{ij}'   \bar{\eta}_j ( f(x_j) - f(x_i) ) 
 =  - u_i \bar{\eta}_i  \sum_{j=1}^n  {W}_{ij}'   \bar{\eta}_j ( f(x_j) - f(x_i) ),  \\ 
\textcircled{2}_i 
& = \hat{\eta}_i\sum_{j=1}^n  {W}_{ij}' (  \bar{\eta}_j -  \hat{\eta}_j )( f(x_j) - f(x_i) )
= -(1+u_i) \sum_{j=1}^n  \bar{\eta}_i {W}_{ij}'\bar{\eta}_j   ( f(x_j) - f(x_i) ) u_j.
\end{align*}
By \eqref{eq:def-W'} and definitions of $\hat{\eta}^c$ and $\bar{\eta}^c$, we equivalently have
\begin{align*}
\textcircled{1}_i 
& =  - u_i  \sum_{j=1}^n  \bar{\eta}^c_i  {W}^{c,0}_{ij}  \bar{\eta}^c_j ( f(x_j) - f(x_i) ),  \\ 
\textcircled{2}_i 
& = -(1+u_i) \sum_{j=1}^n  \bar{\eta}^c_i {W}^{c,0}_{ij} \bar{\eta}^c_j   ( f(x_j) - f(x_i) ) u_j.
\end{align*}
Applying Lemma \ref{lemma:hatetac-baretac} gives that $\| u\|_2 \le  \Theta( \sqrt{n} \varepsilon_{\rm SK})$, 
and we also have $O(1)$ uniform entry-wise upper bound of $(1+u_i)$ as in \eqref{eq:bound-1+ui-noise}.
Note that the claim \eqref{eq:bound-sum-barWij-fj-fi} holds by replacing 
$W_{ij}$ with ${W}^{c,0}_{ij} $
and $\bar{\eta}_i$ with $\bar{\eta}^c_i$,
because the diagonal entry, i.e, when $j=i$, is not involved in the summation,
and also we only need \eqref{eq:bound-row-sum-E1-Wcij} (which holds under $E_{(1)}$) 
to prove \eqref{eq:bound-row-sum-etac-Wc-1}.
The rest of the proof is the same as the  derivation of \eqref{eq:bound-circ1+circ2} in the proof of  Theorem \ref{thm:hatLn-2norm},
which gives that
\[
\| ( \hat{L}' -  \bar{L}' ) \rho_X f  \|_2
\le  \| \textcircled{1} \|_2 + \| \textcircled{2}  \|_2
= \Theta( \sqrt{ \epsilon \log (1/\epsilon)  } )\sqrt{n} \varepsilon_{\rm SK}
\]
under the needed good events as stated in the proposition at large $n$. 
\end{proof}

\subsubsection{Proofs in Step 3. and Theorem \ref{thm:hatL-convergence-noise}}

\begin{proof}[Proof of Proposition \ref{prop:hatL-convergence}]
By definitions of $\hat{L}'$ and $ \hat{L} $ in \eqref{eq:def-three-Ls-noise}, we have
 \begin{align*}
\textcircled{3}_i
& := ( ( \hat{L}' -  \hat{L} ) \rho_X f  )_i \\
& =   \sum_{j=1}^n  \hat{\eta}_i {W}^0_{ij}  \hat{\eta}_j  ( f(x_j) - f(x_i) )
 	-  \sum_{j=1}^n  \hat{\eta}_i  {W}_{ij}'  \hat{\eta}_j  ( f(x_j) - f(x_i) )  \\
& =  \sum_{j=1}^n  \hat{\eta}_i {W}_{ij}' H_{ij}  \hat{\eta}_j  ( f(x_j) - f(x_i) ),
 \end{align*}
where the last equality is by \eqref{eq:W0-and-W'}.
Furthermore, by \eqref{eq:def-u-hateta-bareta},  
\[
\textcircled{3}_i
=  \sum_{j=1}^n (1+u_i) \bar{\eta}_i {W}_{ij}' H_{ij}  \bar{\eta}_j  (1+u_j)( f(x_j) - f(x_i) ).
\]
By \eqref{eq:bound-1+ui-noise} which holds by Lemma \ref{lemma:hatetac-upper} (under $E_{(2)} \cap E_{(z)}$)
\begin{equation}
1+u_i   \le \frac{C_2}{q_{\rm min}}, \quad \forall i,
\end{equation}
and then, together with the uniform upperbound of $|H_{ij}|$ in  \eqref{eq:W0-W'-relation} which holds with large $n$ and under $E_{(z)}$ by Lemma \ref{lemma:uniform-H},
\[
| \textcircled{3}_i |
\le (\frac{C_2}{q_{\rm min}} )^2 \sum_{j=1}^n   \bar{\eta}_i {W}_{ij}'  \bar{\eta}_j | f(x_j) - f(x_i) | | H_{ij} |
\le \Theta (\frac{\varepsilon_z}{\epsilon}) \sum_{j=1}^n   \bar{\eta}_i {W}_{ij}'  \bar{\eta}_j | f(x_j) - f(x_i) |.
\]
At last, by \eqref{eq:def-W'} and the definition of $\bar{\eta}^c$ as in \eqref{eq:def-bareta-noise},
\[
\sum_{j=1}^n   \bar{\eta}_i {W}_{ij}'  \bar{\eta}_j | f(x_j) - f(x_i) |
= \sum_{j=1}^n   \bar{\eta}_i^c {W}^{c,0}_{ij}  \bar{\eta}_j^c | f(x_j) - f(x_i) |
\le \Theta(  \sqrt{ \epsilon \log (1/\epsilon)  }), \quad \forall i,
\]
where the last inequality is by claim \eqref{eq:bound-sum-barWij-fj-fi}
 replacing $W_{ij}$ with ${W}^{c,0}_{ij} $
and $\bar{\eta}_i$ with $\bar{\eta}^c_i$,
as has been used in the proof of Proposition \ref{prop:hatL'-convergence} (holds under $E_{(1)}$).
This shows that 
$\| \textcircled{3} \|_\infty \le \Theta (\frac{\varepsilon_z}{\epsilon})  \Theta(  \sqrt{ \epsilon \log (1/\epsilon)  })$
which proves the proposition, and all constants in big-$\Theta$ depend on $\calM, p$ and $f$.
\end{proof}
\begin{proof}[Proof of Theorem \ref{thm:hatL-convergence-noise}]
Under the condition of the theorem, 
for large $n$ and under the intersection of good events 
$E_{(z)}$ in (A3), 
 $E_{(0)}$ in Proposition \ref{prop:barL'-convergence},
$E_{(1)}$ in Lemma \ref{lemma:bareta-noise}, 
and 
$E_{(2)}$ in Lemma \ref{lemma:hatetac-upper}
which holds with the high probability as stated in the theorem,
Propositions \ref{prop:barL'-convergence}, \ref{prop:hatL'-convergence} and \ref{prop:hatL-convergence} hold.
By the setting of $\varepsilon_{\rm SK}$ as in \eqref{eq:def-epsSK-noise}, one can verify that 
\[
O \left(  \sqrt{\frac{  \log (1/\epsilon)}{\epsilon } } \right)  \varepsilon_{\rm SK}
=  O \left( \epsilon^{3/2}  \sqrt{{  \log (1/\epsilon)}}, 
       \sqrt{\frac{\log n  \log (1/\epsilon) }{n  \epsilon^{d/2+1}}} , 
       \frac{\varepsilon_z}{\epsilon}  \sqrt{\frac{  \log (1/\epsilon)}{\epsilon } }\right),
\]
which is bounded by the r.h.s. of \eqref{eq:bound-thm-hatL-noise}.
The theorem then follows by the three propositions and  that $\| v\|_2 \le \sqrt{n} \|v\|_\infty$ for any $v \in \R^n$.
\end{proof}

\section*{Acknowledgement}
The work is supported by NSF DMS-2007040. The two authors acknowledge support by NIH grant R01GM131642. X.C. is also partially supported by NSF (DMS-1818945, DMS-1820827, DMS-2134037)  and the Simons Foundation.

\small
\bibliographystyle{plain}
\bibliography{kernel}

\appendix

\setcounter{figure}{0} \renewcommand{\thefigure}{A.\arabic{figure}}
\setcounter{table}{0} \renewcommand{\thetable}{A.\arabic{table}}
\setcounter{equation}{0} \renewcommand{\theequation}{A.\arabic{equation}}
\setcounter{remark}{0} \renewcommand{\theremark}{A.\arabic{remark}}

\section{Experimental details}\label{app:exp-detail}

\begin{figure}
\centering
\includegraphics[height=.24\linewidth]{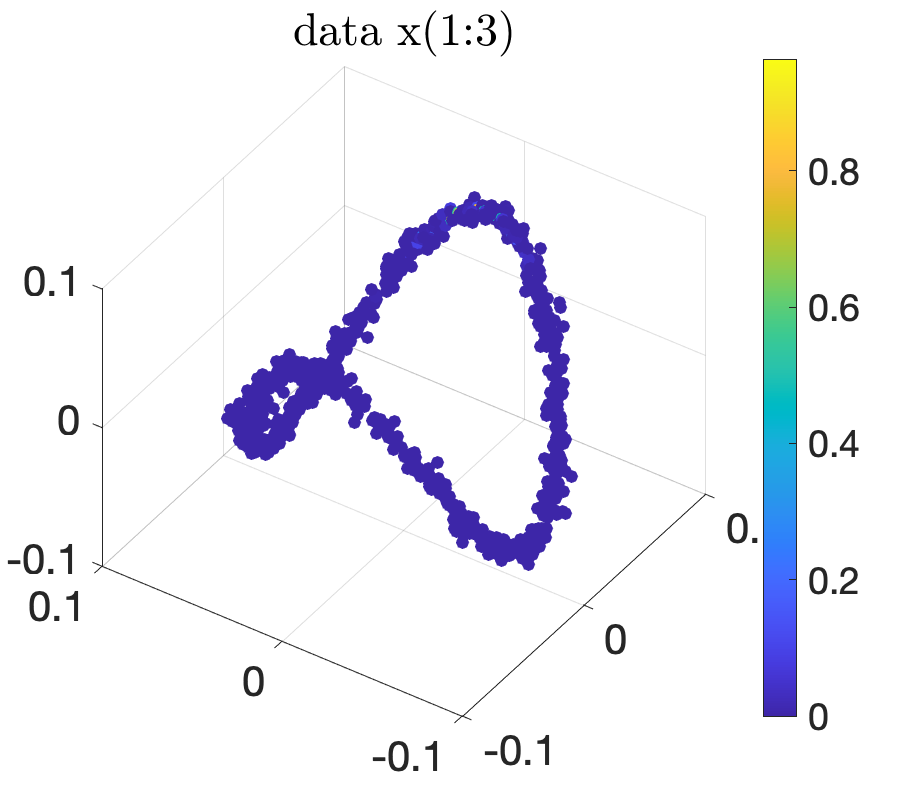}
\includegraphics[height=.24\linewidth]{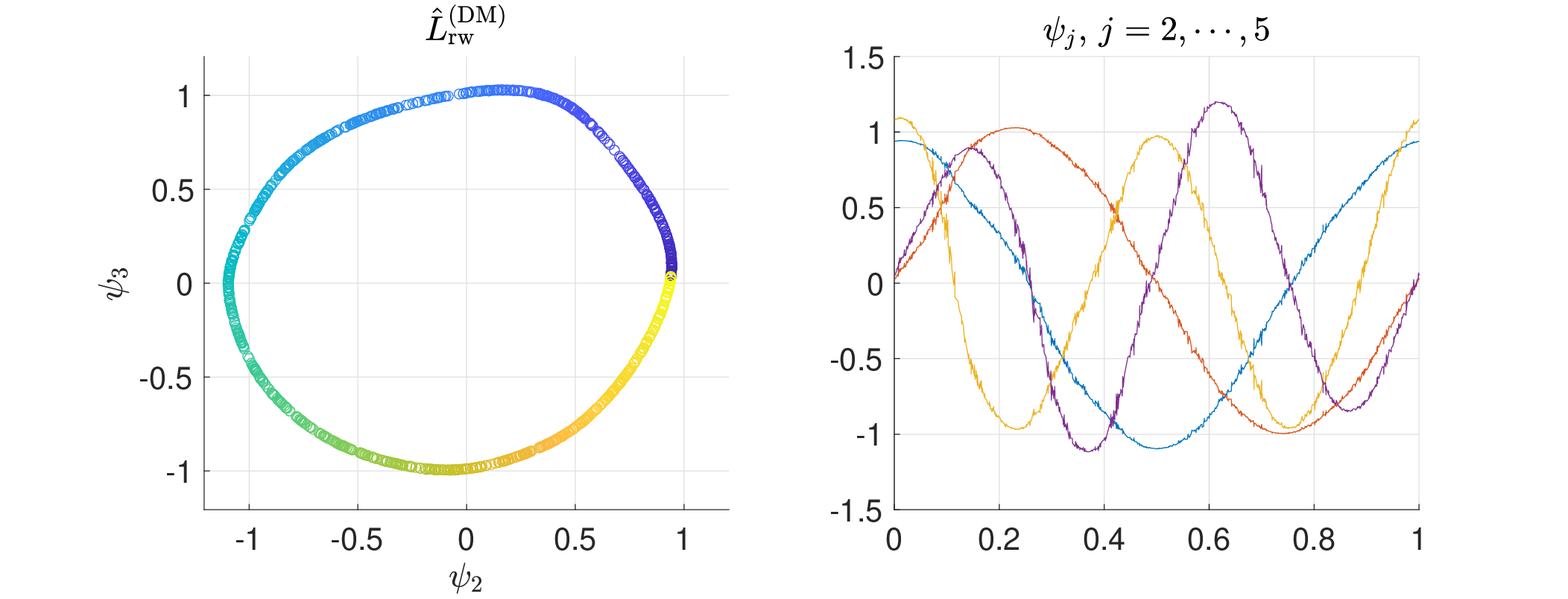}\\
\includegraphics[height=.24\linewidth]{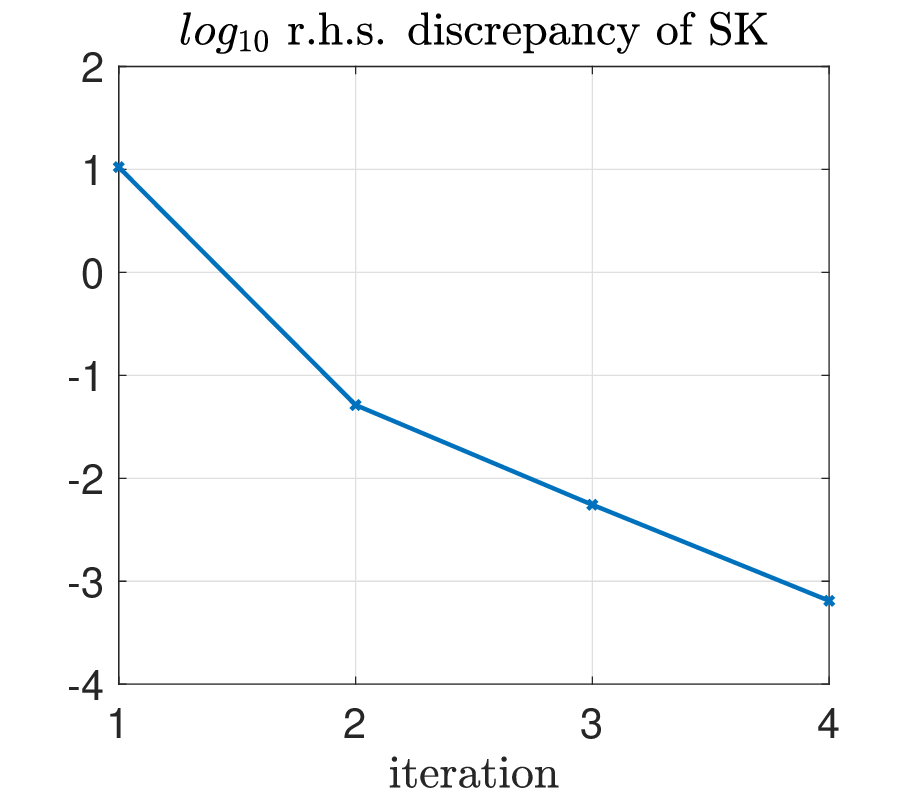}
\includegraphics[height=.24\linewidth]{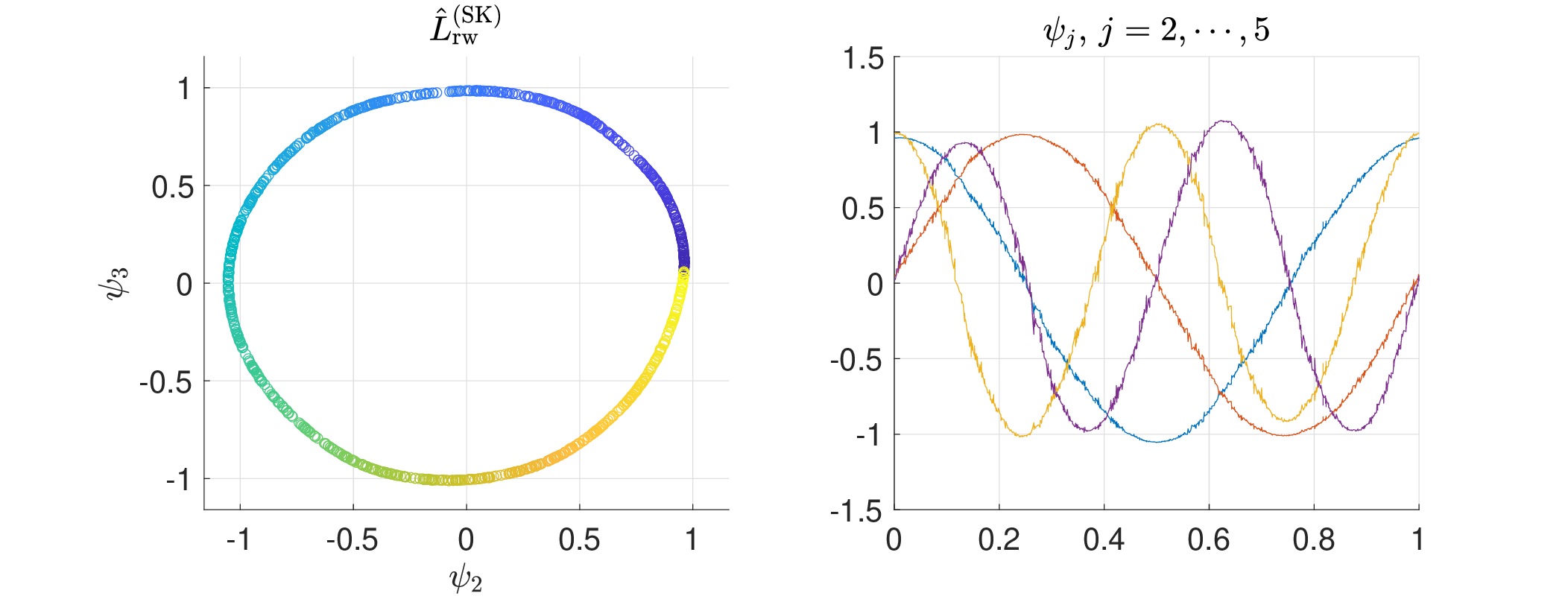}
\vspace{-5pt}
\caption{
\small
Same plots as in Figure \ref{fig:embed-1d-noise-type1} for 1D manifold data with i.i.d. noise (non-heteroskedastic) in  $\R^m$, $m=2000$.
In this simulation, the MSE errors for DM is  
0.0066 and 0.0295
for the first and second pair of harmonics,
and those for SK is 
0.0026 and 0.0069.
}
\label{fig:embed-1d-noise-type2}
\end{figure}

\subsection{Simulated 1D manifold data}

\subsubsection{Clean manifold data and test function}

For the intrinsic coordinate $t \in [0,1]$, the extrinsic coordinates of data samples in $\R^4$ are computed by 
\[
x(t) = \frac{1}{2\pi \sqrt{5} } \left( 
                       cos(2 \pi  t), 
                       sin( 2 \pi t), 
                       \frac{2}{\omega_\calM} \cos (\omega_\calM 2 \pi   t), 
                       \frac{2}{\omega_\calM} \sin (\omega_\calM 2 \pi   t) \right),
\quad \omega_\calM =2.
\]                       
 The data density $p$ on 1D manifold $S^1$ has the expression
\[
p(t) = 1- 0.6 \sin( 6 \pi t),
\]
and the test function $f$ has the expression
\[
f(t) = \sin(2\pi (t + 0.05)).
\]
The manifold data and the functions $p$, $f$, and $\Delta_p f$ are shown in Figure \ref{fig:data-1d}.


\subsection{Computation of $ \hat{L}^{(\rm DM)}$}
Following \cite{coifman2006diffusion}, the $\alpha= {1}/{2}$ normalized graph Laplacian  is constructed as
\begin{equation}\label{eq:def-hatL-DM}
 \hat{L}^{(\rm DM)} = D( \tilde{W}) - \tilde{W},
 \quad
 \tilde{W}_{ij} := \frac{W^0_{ij}}{\sqrt{D(W^0)_i}\sqrt{D(W^0)_j}},
\end{equation}
and $D(W^0)$ is the degree matrix of $W^0$, and $W^0$ equals $W$ as in \eqref{eq:def-W} with zero-ing out the diagonal entries. 

For clean manifold data, under the proper setting of bandwidth parameter $\epsilon$, namely $\epsilon$ is at least larger than the connectivity regime for large $m$ (see Remark \ref{rk:connectivity-regime}), the degree entries $D(W^0)_i$ for all $i$ are strictly positive with high probability, and thus \eqref{eq:def-hatL-DM} is well-defined. 
For data with outlier noise, in the setting considered in this paper, 
the degree matrix of $W^0$  again has strictly positive diagonal entries with high probability: this can be verified by the claim \eqref{eq:Wc0-C3-claim}, the definition of $\rho_i$,  and that $W^{0}_{ij} = \rho_i W^{0,c}_{ij} \rho_j (1 + H_{ij})$ where $H_{ij}$ is uniformly $o(1)$ (Lemma \ref{lemma:uniform-H}).

\begin{figure}
\centering{
\includegraphics[height=.25\linewidth]{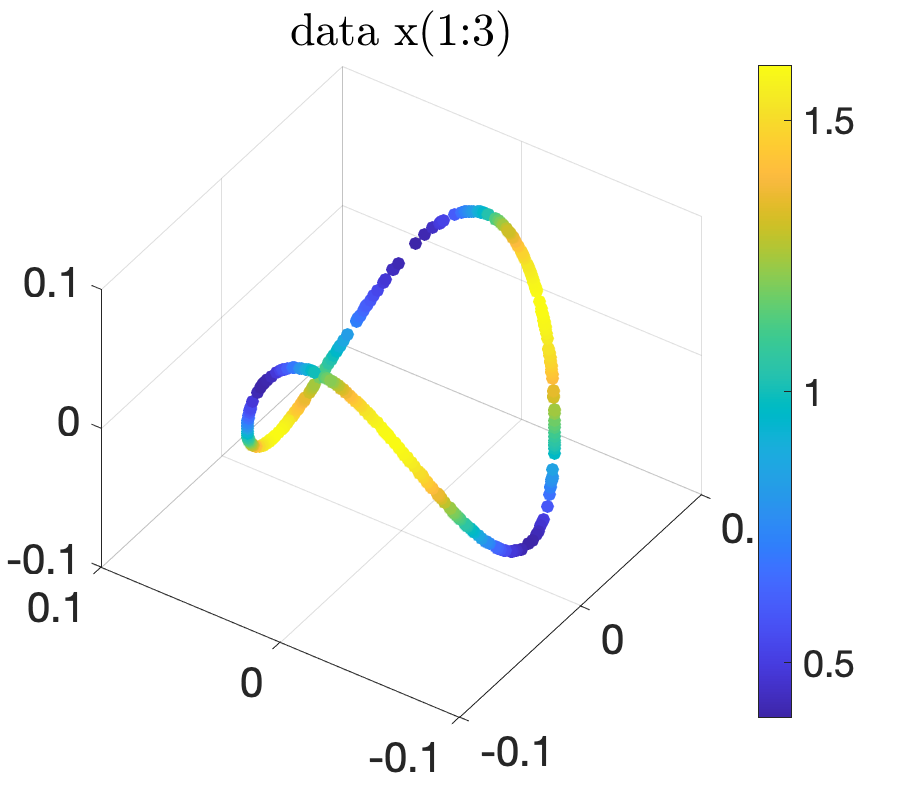} 
\includegraphics[height=.25\linewidth]{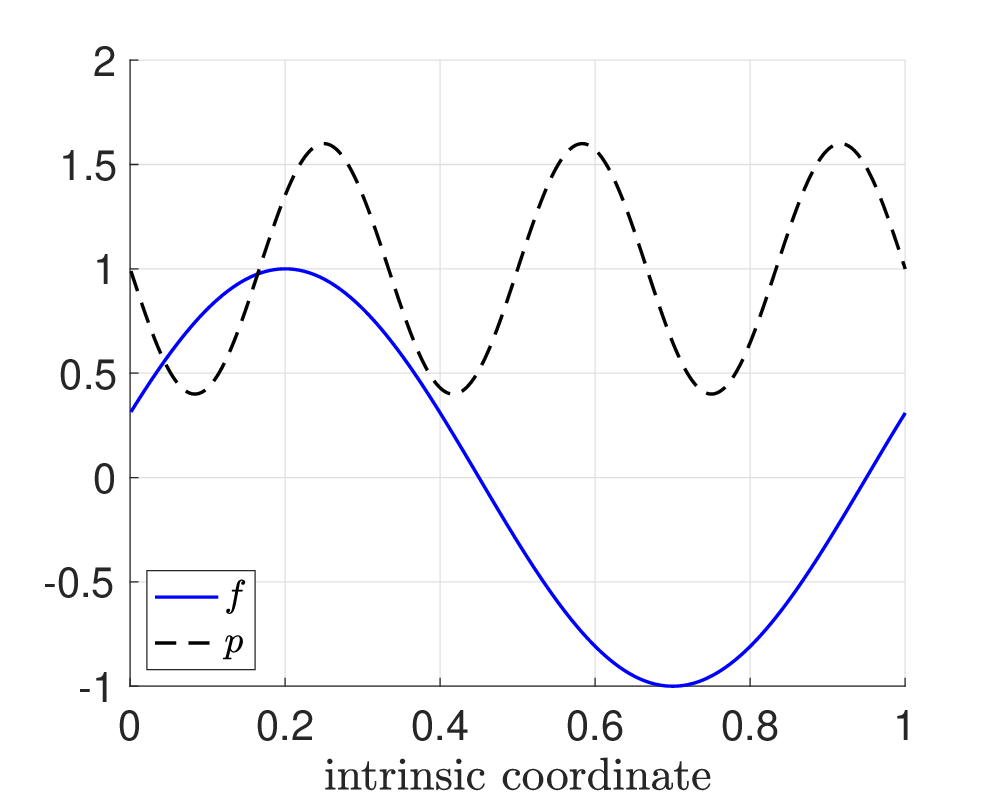}
\includegraphics[height=.25\linewidth]{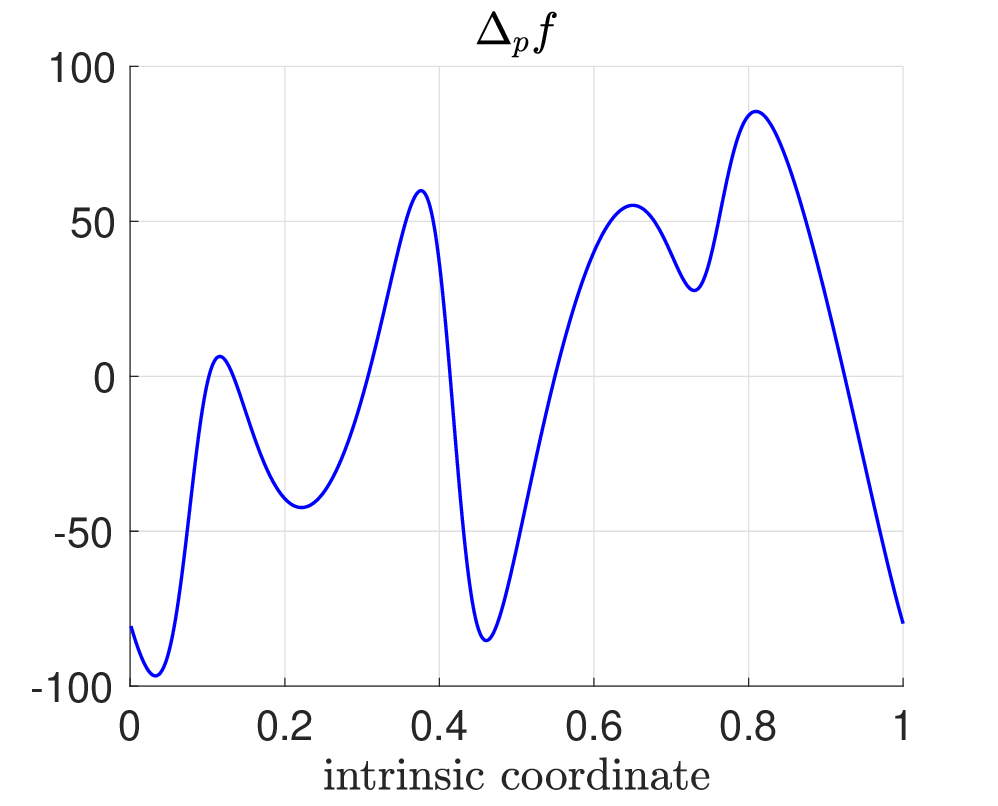}
}\vspace{-5pt}
\caption{
\small
Data in $\R^4$ lying on a 1D closed curve of length 1.
(Left) First 3 coordinates of 500 samples with color indicating the density function $p$.
(Middle) The density function $p$ and a function $f$,
and (Right) $\Delta_p f$, all plotted v.s. the intrinsic coordinate (the arclength) on $[0,1]$. 
}
\label{fig:data-1d}
\end{figure}

\section{Additional proofs}\label{app:more-prooofs}

\subsection{Lemma \ref{lemma:h-integral-diffusionmap}}

\begin{proof}[Proof of Lemma \ref{lemma:h-integral-diffusionmap}]
Lemma A.5 in \cite{cheng2021convergence} was derived for general kernel function $h(\xi)$ satisfying certain  $C^2$ regularity and sub-exponential decay condition. 
The Gaussian kernel $g$ as in \eqref{eq:def-g-gaussian} satisfies the needed conditions.
The lemma directly follows Lemma A.5 in \cite{cheng2021convergence} by letting $h = g$:
The two moments $m_0 [ h ]$ and $m_2 [ h ]$  in the statement of Lemma A.5 in \eqref{eq:def-g-gaussian} 
are defined for kernel function $h: \R_+ \to \R$ as
 \begin{equation}\label{eq:def-m0-m2}
m_0 [ h ] := \int_{\R^d} h( \| u \|^2) du,
 \quad
  m_2 [ h ] :=  \frac{1}{d} \int_{\R^d} \| u \|^2 h( \| u \|^2) du,
  \end{equation}
and the moments are finite when $h$ decays sub-exponentially. 
For Gaussian kernel $g$,  $ m_0[g] = 1$, $ m_2[g] = 2$.
Plugging in the statement of Lemma A.5 in \cite{cheng2021convergence} 
 the values of  $m_0 [ g ]$ and $m_2 [ g ]$ proves Lemma \ref{lemma:h-integral-diffusionmap}
\end{proof}

\[
 \epsilon^{- {d}/{2}} \int_{\calM} g \left( \frac{ \| x-y \|^2}{\epsilon} \right) f(y) dV(y)
=   f(x) +  \epsilon ( \omega f + \Delta_{\calM}  f)(x) 
+ O (\epsilon^2) ,
\]

\begin{equation}
 \sum_{j=1}^n W_{ij}
= {p} (x_i) +  O \left( \epsilon, \sqrt{ \frac{\log n}{n \epsilon^{d/2}} }\right),
\quad 
 i = 1, \cdots, n.
\end{equation}

\begin{proof}[Proof of Lemma \ref{lemma:degree-conc-W}]
The proof uses the same argument as in Lemma 6.1(1) of \cite{cheng2021eigen} and is included for completeness.
We show that \eqref{eq:degree-D-concen} holds for each $i$ under a good event and then apply the union bound to the $n$ good events. 
For fixed $i$, note that 
\[
\sum_{j=1}^n W_{ij} = \frac{\epsilon^{-d/2}}{n}   g( 0)
			+ \frac{1}{n} \sum_{j\neq i} \epsilon^{-d/2} g \left(  \frac{ \| x_i - x_j\|^2}{\epsilon} \right)
			=: \textcircled{1} + (1-\frac{1}{n}) \textcircled{2}.
\]
We have that $\textcircled{1}  = O(n^{-1} \epsilon^{-d/2})$, 
and $\textcircled{2}$ consists of  an independent sum conditioning on $x_i$ and over the randomness of the $n-1$ r.v. $x_j$, namely,
$\textcircled{2} = \frac{1}{n-1}\sum_{j \neq i} Y_j$,
and 
\[
Y_j = \epsilon^{-d/2} g \left(  \frac{ \| x_i - x_j\|^2}{\epsilon} \right), \quad j \neq i.
\]
Applying Lemma \ref{lemma:h-integral-diffusionmap} with $f=p$, which has $C^6$ regularity by Assumption (A2), gives that 
\[
\E [Y_j | x_i] =  \epsilon^{- {d}/{2}} \int_{\calM} g \left( \frac{ \| x_i -y \|^2}{\epsilon} \right) p(y) dV(y)
=  p(x_i) +  O( \epsilon ),
\]
and  the constant in big-$O$ depends on ($\calM, p)$  and is uniform for all $ i$.
Meanwhile, we have boundedness of the r.v.  $|Y_j| \le L_Y = \Theta(\epsilon^{-d/2})$, 
and variance of $Y_j$ is bounded by 
\begin{align}
\E [ Y_j^2 |x_i ]
& = \epsilon^{-d} \int_\calM g \left(  \frac{ \| x_i - y\|^2}{\epsilon} \right)^2 p(y) dV(y)  \nonumber \\
& = (4\pi)^{-d}  \epsilon^{-d} \int_\calM e^{- \frac{ \| x_i - y\|^2}{2 \epsilon}  } p(y) dV(y) 
\le \nu_Y = \Theta( \epsilon^{-d/2}).
\end{align}
The constants in the big-$\Theta$ notation of $L_Y$ and $\nu_Y$ depend on ($\calM, p)$  and not on $x_i$. 
Since ${\log n}/{(n \epsilon^{d/2})} = o(1) $ under the condition of the lemma, the classical Bernstein gives that for large $n$,
there is a good event $E_{i}$ which happens w.p. $>1-2n^{-10}$ under which
$\textcircled{2} = p(x_i) +  O( \epsilon, \sqrt{  \frac{\log n}{n \epsilon^{d/2}}} )$.
Because $\textcircled{1}$ is of smaller order, this gives that 
$\sum_{j=1}^n W_{ij} = p(x_i) +  O( \epsilon, \sqrt{  \frac{\log n}{n \epsilon^{d/2}}} )$ under $E_i$,
where the constant in big-$O$ depends on ($\calM, p)$ and is uniform for all $i$.
Taking the intersection of all $E_i$ as the good event proves the lemma.
\end{proof}

\subsection{Corollary \ref{cor:eigenvector}}

 \begin{proof}[Proof of Corollary \ref{cor:eigenvector}]
 Let $\phi_k $ be defined as in \eqref{eq:def-phik}, and define $\hat L :=  \frac{1}{\epsilon}\hat{L}_n$.
 Under the condition of the corollary, Theorem \ref{thm:hatLn-2norm} applies to give that, for large $n$ under the good events ($E_{(0)} \cap E_{(1)} \cap E_{(2)}$ in the proof) that happen w.p. $>1-6n^{-9}$, we have
 \begin{align}
 \| \hat L \phi_k - \mu_k \phi_k \|_2
  =   \frac{1}{\sqrt{n}} \|  \frac{1}{\epsilon}\hat{L}_n  \rho_X  (\psi_k) - \rho_X(  - \Delta_p \psi_k ) \|_2 
 \le \varepsilon_{\rm pt} := \Theta \left(\epsilon, \sqrt{\frac{\log n \log (1/\epsilon) }{n \epsilon^{d/2+1}}} \right).
 \end{align}
 When $n > n_0$,  the definition of $\gamma_K$ in \eqref{eq:def-gamma-K} and \eqref{eq:eigenvalue-crude} give that (recall that $\mu_1 = \lambda_1 = 0$)
 \begin{equation}\label{eq:eigen-stay-away}
\min_{1 \le j \le n, \, j \neq k} | \mu_k - \lambda_j | > \gamma_K > 0,
\quad 2 \le k \le k_{max}.
\end{equation}
For each $ k \le k_{max}$, let $S_k = \text{Span}\{ u_k \}$ be the 1-dimensional subspace in $\R^n$, and $S_k^\perp$  be its orthogonal complement. 
We denote by $P_W$ the orthogonal projection to the subspace $W$. 
By definition, 
\[
P_{S_k^\perp} ( \mu_k  \phi_k ) = \sum_{j\neq k, j=1}^n  \mu_k (u_j^T \phi_k) u_j,
\]
and by that $\hat L u_j = \lambda_j u_j$,
\[
P_{S_k^\perp} ( \hat L \phi_k ) = \sum_{j\neq k, j=1}^n  (u_j^T \hat L \phi_k) u_j 
= \sum_{j\neq k, j=1}^n  \lambda_j (u_j^T  \phi_k) u_j.
\]
Subtracting the two, we have 
\[
P_{S_k^\perp} ( \mu_k  \phi_k - \hat L \phi_k )
= \sum_{j\neq k, j=1}^n  (\mu_k - \lambda_j) (u_j^T \phi_k) u_j
\]
By that $u_j$ are orthonormal and \eqref{eq:eigen-stay-away},
\begin{align*}
\| P_{S_k^\perp} ( \mu_k  \phi_k - \hat L \phi_k ) \|_2^2 
&= \sum_{j\neq k, j=1}^N  (\mu_k - \lambda_j)^2 (u_j^T \phi_k)^2 \\ 
& \ge \gamma_K^2 \sum_{j\neq k, j=1}^N   (u_j^T \phi_k)^2
= \gamma_K^2 \|   P_{S_k^\perp} \phi_k \|_2^2,
\end{align*}
which gives that 
\begin{align*}
\gamma_K \|   P_{S_k^\perp} \phi_k \|_2 
& \le \| P_{S_k^\perp} ( \mu_k  \phi_k - \hat L \phi_k ) \|_2 \\
& \le \|   \mu_k  \phi_k - \hat L \phi_k  \|_2 \le \varepsilon_{\rm pt},
\end{align*}
that is
\begin{equation}\label{eq:bound-pskperp-phik-1}
 \|   P_{S_k^\perp} \phi_k \|_2  \le \frac{ \varepsilon_{\rm pt}}{\gamma_K} = \Theta \left(\epsilon, \sqrt{\frac{\log n \log (1/\epsilon) }{n \epsilon^{d/2+1}}} \right).
\end{equation}
By definition,
$P_{S_k^\perp} \phi_k  = \phi_k - (u_k^T \phi_k) u_k$,  where $\| u_k \|_2 = 1$.
Note that Lemma \ref{lemma:rhoX-isometry-whp} gives that $\phi_k$ are close to unit vectors:
 For large $n$ and under a good event $E_{(3)}$ which happens w.p.$>1-2K n^{-10}$,
 \[
 \| \phi_k \|^2 = 1 + O( \sqrt{\frac{\log n}{n}}), \quad 1 \le k \le K.
 \]
Together with \eqref{eq:bound-pskperp-phik-1},  one can verify that 
\begin{equation}\label{eq:uk-phik-align}
| u_k^T \phi_k | = 1 + O( \sqrt{\frac{\log n}{n}}) + O(  \varepsilon_{\rm pt}^2) = 1+o(1).
\end{equation}
We define 
\[
\alpha_k = \frac{1}{ u_k^T \phi_k},
\]
and have that 
\begin{align*}
\| \alpha_k \phi_k - u_k \|_2 
& =  \frac{\| \phi_k -  (u_k^T \phi_k)u_k  \|_2 }{| u_k^T \phi_k|} \\
& =  \frac{\| P_{S_k^\perp} \phi_k  \|_2 }{| u_k^T \phi_k|}  
 = \frac{ O( \varepsilon_{\rm pt}  )}{ 1 + o(1)}  \quad \text{(by \eqref{eq:bound-pskperp-phik-1} and \eqref{eq:uk-phik-align})} \nonumber \\
& = O(\varepsilon_{\rm pt}).
\end{align*}
The bound holds for each $k \le k_{max}$.
 \end{proof}

\begin{lemma}\label{lemma:rhoX-isometry-whp}
Under (A1)(A2), $x_i \sim p$ i.i.d.
For fixed $K$, 
when $n$ is sufficiently large, w.p. $> 1 -  2 K n^{-10}$,
\begin{equation}\label{eq:uk-near-orthonormal}
\frac{1}{n } \|  \rho_X ( \psi_k ) \|^2 
 = 1 + O( \sqrt{\frac{\log n}{n}}), \quad 1 \le k \le K.
\end{equation}
\end{lemma}

\begin{proof}[Proof of Lemma \ref{lemma:rhoX-isometry-whp}]
By definition, for each $k=1, \cdots, K$,
\[
\frac{1}{n } \|  \rho_X ( \psi_k ) \|^2 
= \frac{1}{n} \sum_{j=1}^n \psi_k(x_i)^2,
\]
which is an independent sum of r.v. $Y_j := \psi_k(x_i)^2$. 
$\E Y_j = \int_{\calM} \psi_k(y)^2 p(y) dV(y)  = \langle \psi_k, \psi_k \rangle_p = 1$.
For the boundedness of $Y_j$,
since $\psi_k \in C^\infty(\calM)$, 
we have $|Y_j| \le L_Y := \| \psi_k \|_{\infty,\calM }^2$ which is an $O(1)$ constant (depending on $K$).
Then the deviation bound \eqref{eq:uk-near-orthonormal} for each $k$ holds w.p.$>1-2n^{-10}$ by the classical Bernstein.
The claim for all $k$ follows by a union bound. 
\end{proof}

\end{document}